\documentclass{pspum-l}

\usepackage{amssymb}

\usepackage{}
\usepackage{amsmath,amsthm,amssymb}
\usepackage[margin=1in]{geometry}
\usepackage[bookmarks]{hyperref}
\usepackage[all,cmtip]{xy}
\usepackage{verbatim}

\newtheorem{theorem}{Theorem}[subsection]
\newtheorem{lemma}[theorem]{Lemma}
\newtheorem{proposition}[theorem]{Proposition}
\newtheorem{conjecture}[theorem]{Conjecture}
\newtheorem{corollary}[theorem]{Corollary}

\theoremstyle{definition}
\newtheorem{definition}[theorem]{Definition}

\theoremstyle{remark}
\newtheorem{remark}[theorem]{Remark}

\newtheorem{observation}[theorem]{Observation}
\numberwithin{equation}{subsection}
\renewcommand{\qedsymbol}{$_{\square}$}

\newcommand{\gfp}{G(\mathbb{F}_p)}
\newcommand{\gfpr}{G(\mathbb{F}_{q})}
\newcommand{\gfq}{G(\mathbb{F}_{p^r})}

\newcommand{\Fp}{\mathbb{F}_p}
\newcommand{\cgr}{\mathcal{G}_r(k)}
\newcommand{\ul}{\mathfrak{u}}

\newcommand{\Lie}{\operatorname{Lie}}

\newcommand{\ch}{\text{\rm ch}}
\newcommand{\ind}{\operatorname{ind}}
\newcommand{\Ext}{\operatorname{Ext}}
\newcommand{\into}{\hookrightarrow}
\newcommand{\onto}{\twoheadrightarrow}
\newcommand{\opH}{\operatorname{H}}
\renewcommand{\theenumi}{\roman{enumi}}
\renewcommand{\labelenumi}{(\theenumi)}
\newcommand{\Hom}{\operatorname{Hom}}

\renewcommand{\ch}{\operatorname{ch}}

\newcommand{\ga}{\gamma}

\newcommand{\la}{\lambda}

\newcommand{\al}{\alpha}
\newcommand{\ta}{\tilde{\alpha}}
\newcommand{\be}{\beta}
\newcommand{\ep}{\epsilon}

\newcommand{\si}{\sigma}

\begin{document}

\title[On the vanishing ranges for the cohomology of finite
groups of Lie type II]{On the vanishing ranges for the cohomology of finite
groups of Lie type II}


\author{Christopher P. Bendel}
\address{Department of Mathematics, Statistics and Computer Science,
University of
Wisconsin-Stout, 
Menomonie WI~54751, USA}
\curraddr{}
\email{bendelc@uwstout.edu}
\thanks{Research of the first author was supported in part by
NSF
grant DMS-0400558}

\author{Daniel K. Nakano}
\address{Department of Mathematics, University of Georgia,
Athens GA~30602, USA}
\curraddr{}
\email{nakano@math.uga.edu}
\thanks{Research of the second author was supported in part by
NSF
grant DMS-1002135}

\author{Cornelius Pillen}
\address{Department of Mathematics and Statistics, University
of South
Alabama,
Mobile AL~36688, USA}
\curraddr{}
\email{pillen@jaguar1.usouthal.edu}
\thanks{}

\date{November 2011}

\begin{abstract} The computation of the cohomology for finite groups of Lie type in the describing 
characteristic is a challenging and difficult problem. In \cite{BNP}, the authors constructed an induction 
functor which takes modules over the finite group of Lie type, $G({\mathbb F}_{q})$, to modules for the ambient algebraic group $G$. 
In particular this functor when applied to the trivial module yields a natural $G$-filtration. This filtration was utilized 
in \cite{BNP} to determine the first non-trivial cohomology class when the underlying root system is of 
type $A_{n}$ or $C_{n}$. In this paper the authors extend these results toward locating the first non-trivial 
cohomology classes for the remaining finite groups of Lie type (i.e., the underlying root system 
is of type $B_{n}$, $C_{n}$, $D_{n}$, $E_{6}$, $E_{7}$, $E_{8}$, $F_{4}$, and $G_{2}$) when the prime 
is larger than the Coxeter number. 
\end{abstract}

\maketitle


\section{Introduction}

\subsection{} Let $G$ be a simple algebraic group scheme over a
field $k$
of prime characteristic $p$ which is defined and split over the
prime field $\Fp$,
and $F : G \to G$ denote the Frobenius map. The fixed points of
the $r$th
iterate of the Frobenius map, denoted $\gfpr$, is a finite
Chevalley group where
${\mathbb F}_{q}$ denotes the finite field with $p^{r}$
elements. An elusive problem of major
interest has been to determine the cohomology ring
$\opH^{\bullet}(\gfpr,k)$.
Until recently, aside from small rank cases, it was not even
known in
which degree the first non-trivial cohomology class occurs.

This present paper is a sequel to \cite{BNP} where we began 
investigating three related problems of increasing levels of
difficulty:
\vskip .25cm 
\noindent 
(1.1.1) Determining Vanishing Ranges: Finding $D>0$ such that 
the cohomology group $\opH^i(\gfpr,k)=0$ for $0 <  i < D$.
\vskip .15cm 
\noindent 
(1.1.2) Locating the First Non-Trivial Cohomology Class: Finding
a $D$
satisfying (1.1.1) such that 
$$\opH^{D}(\gfpr,k)\neq 0.$$ 
A $D$ satisfying this property will be called a {\em sharp bound}.
\vskip .15cm 
\noindent
(1.1.3) Determining the Least Non-Trivial Cohomology: For a
sharp $D$ as in (1.1.2) compute $\opH^D(\gfpr,k)$.
\vskip .25cm

Vanishing ranges (1.1.1) were found in earlier work of Quillen
\cite{Q}, Friedlander \cite{F} and Hiller \cite{H} .
Sharp bounds (1.1.2) were later found by Friedlander and
Parshall for the Borel subgroup $B({\mathbb F}_{q})$ of the
$GL_{n}({\mathbb F}_{q})$, and conjectured for the general
linear group by Barbu \cite{B}.
A more detailed discussion of these results can be found in
\cite[Section 1.1]{BNP}.

In \cite{BNP}, for simple, simply connected $G$ and primes $p$
larger than the
Coxeter number $h$, we proved that $\opH^i(\gfq,k) = 0$ for $0 <
i < r(p-2)$.
This provided an answer to (1.1.1) and improved on Hiller's
bounds \cite{H}.
For a group with underlying root system of type $C_n$, we
demonstrated that $D = r(p-2)$ is in fact
a sharp bound, answering (1.1.2). The first non-vanishing
cohomology, as in (1.1.3),
was also determined. For type $A_n$, questions (1.1.2) and
(1.1.3) were also answered,
where the $r > 1$ cases required the prime to be at least twice
the Coxeter number.
Our methods also yielded a proof of Barbu's Conjecture
\cite[Conjecture 4.11]{B}.

In this paper, we continue these investigations in two
directions. First we consider the case that $G$ is a group of
adjoint type (as opposed to simply connected). For such $G$ with
$p > h$,
when the root system is simply laced, one obtains a uniform
sharp bound of $r(2p - 3)$
answering (1.1.2) (cf. Corollary~\ref{C:simplylacedadjoint}). The same uniform bound also holds for the adjoint versions of the twisted groups of 
types $A$, $D$, and $E_6$ when $p >h$. 

We
then consider
the remaining types in the simply connected case. For $G$ being
simple, simply connected and having root system
of type $D_n$ with $p > h$, (1.1.2) and (1.1.3) are answered
(cf. Theorem~\ref{T:Dsummary2}). For type $E_n$,
(1.1.2) is answered for  all primes $p >h$ (with the exceptions of $p= 17, 19$ for type $E_6$); cf.
Theorems~\ref{T:E6q-results}, \ref{T:E7summary-r}, and
\ref{T:E8summary}. 

The calculations for the non-simply-laced groups are considerably more complicated. 
For type $B$ we answer (1.1.2) when $r=1$ and $p>h$, see Theorem~\ref{T:Bsumm}. 
Some discussion of the situation for types $G_2$ and $F_4$ is given
in Sections 7 and 8 respectively. For $r = 1$ and $p > h$, we find improved answers to 
(1.1.1); cf. Theorem \ref{T:G2summary} and Theorem \ref{T:F4van}.  
Finding an answer to (1.1.2) and (1.1.3) continues to be elusive in these types
although some further information towards answering these questions is obtained.


\subsection{\bf Notation.} Throughout this paper, we will follow
the notation and conventions given in the standard reference
\cite{Jan}. $G$ will denote a simple, simply connected algebraic
group scheme which is defined
and split over the finite field ${\mathbb F}_p$ with $p$
elements (except in Section \ref{S:adjoint} where $G$ is assumed to be of adjoint
type rather than simply connected). Throughout the paper let $k$ be an algebraically
closed field of
characteristic $p$. For $r\geq 1$, let $G_r:=\text{ker }F^{r}$
be the $r$th Frobenius kernel of $G$ and $\gfpr$ be the
associated finite Chevalley group.
Let $T$ be a maximal split torus and $\Phi$ be the root system
associated to $(G,T)$. The positive (resp. negative)
roots are $\Phi^{+}$ (resp. $\Phi^{-}$), and $\Delta$ is
the set of simple roots. Let $B$ be a Borel subgroup
containing $T$ corresponding to the negative roots and $U$ be
the
unipotent radical of $B$. For a given root system of rank $n$,
the simple roots will be denoted by $\al_1, \al_2, \dots, \al_n$
(via the the Bourbaki ordering of simple roots). For type
$B_n$, $\al_n$ denotes the unique short simple root and for type
$C_n$, $\al_n$ denotes the unique long simple root.
The highest (positive) root will be denoted $\ta$, and for root
systems
with multiple root lengths, the highest short root will be
denoted $\al_0$.
Let $W$ denote the Weyl group associated to $\Phi$, and, for $w
\in W$, let $\ell(w)$
denote the length of the element $w$ (i.e., number of elements
in a reduced expression for $w$).

Let ${\mathbb E}$ be the Euclidean space associated with $\Phi$,
and
the inner product on ${\mathbb E} $ will be denoted by $\langle\
, \
\rangle$. Let
$\alpha^{\vee}=2\alpha/\langle\alpha,\alpha\rangle$
be the coroot corresponding to $\alpha\in \Phi$. The fundamental
weights (basis dual to
$\al_1^{\vee}, \al_2^{\vee}, \dots, \al_n^{\vee}$)
will be denoted by $\omega_1$, $\omega_2$, \dots,
$\omega_n$. Let $X(T)$ be the integral weight lattice spanned by these fundamental weights. The set of dominant integral weights
is
denoted by $X(T)_{+}$. For a weight $\la \in X(T)$, set 
$\la^* : = -w_0\la$ where $w_0$ is the longest word in the 
Weyl group $W$. By $w \cdot \la := w(\la + \rho) -\rho$ we mean
the ``dot" action of $W$ on $X(T)$, with $\rho$ being the
half-sum
of the positive roots. For $\al \in \Delta$, $s_{\al} \in W$
denotes the reflection in the hyperplane determined by $\al$.

For a $G$-module $M$, let $M^{(r)}$ be the module obtained by
composing
the underlying representation for $M$ with $F^{r}$. Moreover,
let
$M^*$ denote the dual module. For $\la \in X(T)_+$, let
$H^0(\la) := \ind_B^G\la$
be the induced module and $V(\la) := H^0(\la^*)^*$ be the Weyl
module of highest
weight $\lambda$.


\section{General Strategy and Techniques}


\subsection{} We will employ the basic strategy used in
\cite{BNP} in addressing (1.1.1)-(1.1.3) which
uses effective techniques developed by the authors which
relate $\opH^i(\gfpr,k)$ to extensions over $G$ via
an induction functor ${\mathcal G}_{r}(-)$. When ${\mathcal G}_{r}(-)$ is applied to the
trivial module $k$, ${\mathcal G}_{r}(k)$ has a
filtration with factors of the form $H^0(\la) \otimes
H^{0}(\la^*)^{(r)}$ \cite[Proposition 2.4.1]{BNP}.
The $G$-cohomology of these factors can be analyzed by using the
Lyndon-Hochschild-Serre (LHS) spectral sequence
involving the  Frobenius kernel $G_{r}$ (cf. \cite[Section
3]{BNP}). In particular for $r=1$, we can apply the
results of Kumar-Lauritzen-Thomsen \cite{KLT} to determine the
dimension of a
cohomology group $\opH^i(G,H^0(\la) \otimes
H^{0}(\la^*)^{(1)})$, which can in turn
be used to determine $\opH^i(\gfq,k)$. The dimension formula
involves the combinatorics                                                                                                                                                                                                                                                                                                                                                                                                                                                                                                                                                                                                                                                                                                                                                                                                                                                     
of the well-studied Kostant Partition Function. This reduces the
question of the vanishing of the finite group cohomology
to a question involving the combinatorics of the underlying root
system $\Phi$.

For root systems of types $A$ and $C$ the relevant root system
combinatorics was analyzed in \cite[Sections 5-6]{BNP}.
In the cases of the other root systems ($B$, $D$, $E$, $F$, $G$)
the combinatorics is much more involved and
we handle these remaining cases in Sections 4-8. In this section,
for the convenience of the reader, we state the key
results from \cite{BNP} which will be used throughout this
paper.


\subsection{}\label{observ} We first record here a formula for
$-w\cdot 0$ that will be used at various times in the exposition
\cite[Observation 2.1]{BNP}:


\begin{observation}\label{O:uniquedecomp}
If $w \in W$ admits a reduced expression
$w = s_{\be_1}s_{\be_2}\dots s_{\be_m}$ with $\be_i \in \Delta$
and
$m = \ell(w)$, then 
$$
-w\cdot 0 = \be_1 + s_{\be_1}(\be_2) + 
s_{\be_1}s_{\be_2}(\be_3) + \cdots + s_{\be_1}s_{\be_2}\dots
s_{\be_{m-1}}(\be_m).
$$
Moreover, this is the unique way in which $-w\cdot 0$ can be
expressed as a
sum of distinct positive roots.
\end{observation}


\subsection{\bf The Induction Functor and Filtrations.} Let
$\cgr := \ind_{\gfpr}^G(k)$. The functor ${\mathcal G}_{r}(-)$ is exact and
one can use Frobenius reciprocity to relate extensions over $G$
with extensions over $\gfpr$ \cite[Proposition 2.2]{BNP}.

\begin{proposition}\label{P:extcomparison} Let $M, N$ be
rational $G$-modules. Then, for all $i \geq 0$,
$$
\Ext^i_{\gfpr}(M,N) \cong \Ext_G^i(M,N\otimes\cgr).
$$ 

\end{proposition}

In order to make the desired computations of cohomology groups,
we will make use of Proposition~\ref{P:extcomparison} (with $M =
k = N$).
In addition, we will use a special filtration on $\cgr$ (cf.
\cite[Proposition 2.4.1]{BNP}).

\begin{proposition}\label{P:sections} The induced module $\cgr$
has a filtration with factors of the form $H^0(\la)\otimes
H^0(\la^*)^{(r)}$
with multiplicity one for each $\la \in X(T)_+$.
\end{proposition}


\subsection{\bf A Vanishing Criterion.} \label{SS:vanishing} The
filtration from Proposition \ref{P:sections}
allows one to obtain a condition on $G$-cohomology which leads
to vanishing
of $\gfq$-cohomology (cf. \cite[Corollary 2.6.1]{BNP}).

\begin{proposition} \label{P:vanishing} Let $m$ be the least
positive integer such that
there exists $\la \in X(T)_{+}$ with $\opH^m(G,H^0(\la)\otimes
H^0(\la^*)^{(r)}) \neq 0$.
Then $\opH^i(\gfpr,k) \cong \opH^i(G,\cgr) = 0$ for $0 < i < m$.\end{proposition}


\subsection{\bf Non-vanishing.} \label{SS:nonvanishing} While
the identification of an $m$
satisfying Proposition~\ref{P:vanishing} gives a vanishing range
as in (1.1.1),
it does not a priori follow that $\opH^m(\gfpr,k) \neq 0$. 
The following theorem provides conditions which assist with
addressing (1.1.2) or (1.1.3) \cite[Theorem 2.8.1]{BNP}.

\begin{theorem} \label{T:nonvanishing} Let $m$ be the least
positive integer such that there exists
$\nu \in X(T)_{+}$ with $\opH^m(G,H^0(\nu)\otimes
H^0(\nu^*)^{(r)}) \neq 0$.
Let $\la \in X(T)_+$ be such that 
$\opH^m(G,H^0(\la)\otimes H^0(\la^*)^{(r)}) \neq 0$.  
Suppose $\opH^{m+1}(G,H^0(\nu)\otimes H^0(\nu^*)^{(r)}) = 0$ for
all $\nu < \la$ that are linked to $\la$.
Then
\begin{itemize}
\item[(a)] $\opH^i(\gfpr,k) = 0$ for $0 < i < m$;
\item[(b)] $\opH^m(\gfpr,k) \neq 0$;
\item[(c)] if, in addition, $\opH^{m}(G,H^0(\nu)\otimes
H^0(\nu^*)^{(r)}) = 0$
for all $\nu \in X(T)_+$ with $\nu \neq \la$, then 
$$\opH^m(\gfpr,k) \cong \opH^m(G,H^0(\la)\otimes H^0(\la^*)^{(r)}).$$
\end{itemize}
\end{theorem}

From the filtration on $\cgr$ in Proposition~\ref{P:sections}, 
$\opH^i(\gfpr,k) \cong \opH^i(G,\cgr)$
can be decomposed as a direct sum over linkage classes of
dominant weights.
For a fixed linkage class $\mathcal{L}$, let $m$ be the least 
positive integer such that there exists
$\nu \in \mathcal{L}$ with $\opH^m(G,H^0(\nu)\otimes
H^0(\nu^*)^{(r)}) \neq 0$. Let $\la \in \mathcal{L}$ be such
that
$\opH^m(G,H^0(\la)\otimes H^0(\la^*)^{(r)}) \neq 0$.  
Suppose $\opH^{m+1}(G,H^0(\nu)\otimes H^0(\nu^*)^{(r)}) = 0$ for
all $\nu < \la$
in $\mathcal{L}$. Then analogous to
Theorem~\ref{T:nonvanishing}, it follows that
$\opH^m(\gfpr,k) \neq 0$.  See \cite[Theorem 2.8.2]{BNP}.


\subsection{\bf Reducing to $G_1$-cohomology.} From
Sections~\ref{SS:vanishing} and ~\ref{SS:nonvanishing}, the key
to
understanding the vanishing of $\opH^i(\gfq,k)$ is to understand $\opH^i(G,H^0(\la)\otimes H^0(\la^*)^{(r)})$ for all dominant
weights $\la$.
For $r = 1$, these groups can be related to $G_1$-cohomology
groups (cf. \cite[Lemma 3.1]{BNP}).

\begin{lemma} \label{L:G1cohom} Suppose $p > h$ and let $\nu_1,
\nu_2 \in X(T)_+$.
Then for all $j$ 
$$
\opH^j(G,H^0(\nu_1)\otimes H^0(\nu_2^*)^{(1)}) \cong
\Ext_{G}^j(V(\nu_2)^{(1)},H^0(\nu_1)) \cong
\Hom_{G}(V(\nu_2),\opH^j(G_1,H^0(\nu_1))^{(-1)}).
$$
\end{lemma}

We remark that the aforementioned lemma would hold for arbitrary
$r$th-twists if it was known
that the cohomology group $\opH^j(G_r,H^0(\nu))^{(-r)}$ admits a
good filtration, which is
a long-standing conjecture of Donkin. For $p > h$, this is known
for $r = 1$ by
results of Andersen-Jantzen \cite{AJ} and
Kumar-Lauritzen-Thomsen \cite{KLT}.
In that case, the lemma is only needed when $\nu_1 = \nu_2$. 
For arbitrary $r$ we can often work inductively from the $r = 1$
case.
This requires slightly more general $\Ext$-computations and the
possibility that
$\nu_1 \neq \nu_2$.


\subsection{Dimensions for $r = 1$.} \label{SS:G1analysis}
From Lemma~\ref{L:G1cohom}, for $\nu \in X(T)_{+}$, the
cohomology group
$\opH^i(G,H^0(\nu)\otimes H^0(\nu^*)^{(1)})$
 can be identified with
$\Hom_G(V(\nu),\opH^i(G_1,H^0(\nu)^{(-1)})$.  
It is well-known that, from block considerations,
$\opH^i(G_1,H^0(\nu)) = 0$
unless $\nu = w\cdot 0 + p\mu$ for $w \in W$ and $\mu \in X(T)$.  For $p > h$, from \cite{AJ} and \cite{KLT}, we have
\begin{equation}\label{ind}
\opH^i(G_1,H^0(\nu))^{(-1)} = 
\begin{cases}
\ind_B^G(S^{\frac{i - \ell(w)}{2}}(\ul^*)\otimes\mu) &\text{ if
} \nu = w\cdot 0 + p\mu\\
0 & \text{ otherwise,}
\end{cases}
\end{equation}
where $\ul = \Lie(U)$. Note also that, since $p > h$ and $\nu$
is dominant,
$\mu$ must also be dominant.

For a dominant weight $\nu = p\mu + w\cdot 0$, observe that,
from Lemma~\ref{L:G1cohom}
and (\ref{ind}), we have
\begin{align*}
\opH^i(G,H^0(\nu)\otimes H^0(\nu^*)^{(1)}) &\cong
\Hom_G(V(\nu),\opH^i(G_1,H^0(\nu))^{(-1)})\\
&\cong \Hom_G(V(\nu),\ind_B^G(S^{\frac{i -
\ell(w)}{2}}(\ul^*)\otimes\mu))\\
&\cong \Hom_B(V(\nu), S^{\frac{i -
\ell(w)}{2}}(\ul^*)\otimes\mu).
\end{align*}
Hence, if $\opH^i(G,H^0(\nu)\otimes H^0(\nu^*)^{(1)}) \neq 0$, 
then $\nu - \mu = (p - 1)\mu + w\cdot 0$ must be a sum of 
$(i - \ell(w))/2$ positive roots.

For a weight $\nu$ and $n \geq 0$, let $P_n(\nu)$ denote the
dimension of the
$\nu$-weight space of $S^n(\ul^*)$. Equivalently, for $n > 0$,
$P_n(\nu)$ denotes the
number of times that $\nu$ can be expressed as a sum of exactly
$n$
positive roots, while $P_0(0) = 1$. The function $P_{n}$ is
often referred to as
{\em Kostant's Partition Function}. By using \cite[3.8]{AJ},
\cite[Thm 2]{KLT}, Lemma~\ref{L:G1cohom},
and (\ref{ind}), we can give an explicit formula for the
dimension of
$ \opH^i(G,H^0(\la)\otimes H^0(\la^*)^{(1)})$ (cf.
\cite[Proposition 3.2.1, Corollary 3.5.1]{BNP}).

\begin{proposition} \label{P:KostantPartCohom} Let $p >h$ and
$\la = p \mu + w\cdot 0 \in X(T)_+$.
Then 
$$
\dim \opH^i(G,H^0(\la)\otimes H^0(\la^*)^{(1)}) = \sum_{u \in W}
(-1)^{\ell(u)} P_{\frac{i-\ell(w)}{2}}( u\cdot\la - \mu).
$$
\end{proposition}


\subsection{\bf Degree Bounds.} \label{SS:degreebounds} 
The following gives a fundamental constraint on non-zero $i$
such that
$$\opH^i(G,H^0(\la)\otimes H^0(\la^*)^{(1)}) = \Ext_G^i(V(\la)^{(1)},H^0(\la))\neq 0.$$
It is stated in a more general $\Ext$-context as it will also 
be used in some inductive arguments for $r > 1$ (cf.
\cite[Proposition 3.4.1]{BNP}).

\begin{proposition} \label{P:degreebound1} Let $p > h$ with
$\ga_1, \ga_2 \in X(T)_+$, both non-zero, such that
$\ga_j = p \delta_j + w_j\cdot 0 $ with $\delta_j \in X(T)_+$
and $w_j \in W$
for $j=1, 2$. Assume $\Ext^i_G(V(\ga_2)^{(1)}, H^0(\ga_1)) \neq
0$.
\begin{itemize}
\item[(a)] Let $\si \in \Phi^+$. If $\Phi$ is of type $G_2$,
assume that $\si$
is a long root. Then
$p\langle \delta_2, \si^{\vee}\rangle - \langle \delta_1,
\si^{\vee}\rangle
+ \ell(w_1) + \langle w_2\cdot 0,  \si^{\vee}\rangle \leq i.$
\item[(b)] If $\ta$ denotes the longest root in $\Phi^+$, then
$p\langle \delta_2, \ta^{\vee}\rangle - 
        \langle \delta_1, \ta^{\vee}\rangle 
+ \ell(w_1) - \ell(w_2) - 1 \leq i.$
Equality requires that $\ga_2 - \delta_1 = ((i -
\ell(w_1))/2)\ta$ and
$\langle -w_2\cdot0,\ta^{\vee}\rangle = \ell(w_2) + 1$.
\item[(c)] If $\ga_1 = \ga_2 = p\delta + w\cdot 0$, then 
$i \geq (p-1)\langle\delta,\ta^{\vee}\rangle - 1$.
\end{itemize}
\end{proposition}

Proposition~\ref{P:degreebound1} can be generalized to the
following (cf. \cite[Proposition 4.3.1]{BNP}).

\begin{proposition} \label{P:degreebound2} Let $p > h$, $0 \neq
\la \in X(T)_+$ and $i \geq 0$.
If $\opH^i(G,H^0(\la)\otimes H^0(\la^*)^{(r)}) \neq 0$, then
there exists a sequence of non-zero weights
$\la = \ga_0, \ga_1, \dots, \ga_{r-1} , \ga_r= \la \in X(T)_+$
such that
$\ga_j = p \delta_j + u_j\cdot 0$ for some $u_j \in W$ and
nonzero $\delta_j \in X(T)_+$. Moreover, for each $1 \leq j \leq
r$, there exists a nonnegative integer
$l_j$ with $\Ext^{l_j}_G(V(\ga_j)^{(1)}, H^0(\ga_{j-1})) \neq 0$
and
 $\sum_{j = 1}^rl_j = i$.  Furthermore,
\begin{equation}\label{rbound}
i \geq  \left(\sum_{j=1}^r (p-1) \langle \delta_j, 
        \ta^{\vee}\rangle\right) - r.
\end{equation}
Equality requires that $p\delta_j - \delta_{j-1} + u_j\cdot 0 =((l_j - \ell(u_{j-1}))/2)\ta$ and that $\langle -u_j\cdot
0,\ta^{\vee}\rangle
= \ell(u_j) + 1$ for all $1 \leq j \leq r$.
\end{proposition}

Note that the assumption $\opH^i(G,H^0(\la)\otimes
H^0(\la^*)^{(r)}) \neq 0$
in the proposition can be replaced by
$$\Ext_{G/G_1}^k(V(\la)^{(r)},
\opH^l(G_1,H^0(\la)))\neq 0,$$ 
where $k+l =i$. In that case one arrives at the same
conclusions with $l_1 =l$.



\section{Vanishing Ranges in the Simply Laced Case}

In this section we obtain some general vanishing information for
those cases
when the root system $\Phi$ is simply laced.  In such cases, the longest root $\ta$ and the longest short root $\al_0$ coincide.
Following the discussion in Section 2, we want to consider when
$$\opH^i(G,H^0(\la)\otimes H^0(\la^*)^{(1)}) \neq 0$$ for $i >
0$
and $\la \in X(T)_+$.


\subsection{} 
To gain information on such cohomology groups, we will use
Lemma~\ref{L:G1cohom} and
(\ref{ind}). The following proposition will aid us in showing
that certain cohomology
groups are non-zero.

\begin{proposition} \label{P:Homcalculation} Let
$\tilde{\alpha}$ denote the longest root of $\Phi$ and $l$
be a nonnegative integer. Then
$$\Hom_G(V((l+1)\tilde{\alpha}), \ind_B^G( S^l
(\ul^{\ast})\otimes \ta)) \cong k.$$
\end{proposition}

\begin{proof} The claim follows from the diagram below and the
fact that all modules in the commutative diagram below
have a one-dimensional highest weight space with weight $(l+1)
\tilde{\alpha}$. The first embedding is a consequence of the
fact that the module $\underbrace{V(\tilde{\alpha}) \otimes \cdots
\otimes V(\tilde{\alpha})}_{(l+1)\text{ times}}$ has a Weyl module
filtration. The Weyl module $V( \tilde{\alpha})$ is isomorphic
to the dual of the adjoint representation,
$\mathfrak{g}^{\ast}$. Clearly, $\mathfrak{g}^{\ast}$ maps onto
$\ul^{\ast}$ and $V(\ta)$ maps onto
$\phi_{-\ta}$ (of weight $\ta$) as $B$-modules. Hence, we obtain
the two $B$-surjections in the first line of the diagram. The
remaining maps and the commutativity of the diagram arise via
the universal property of the
induction functor.  
\begin{center}
\begin{picture}(340,60)(0,0)
\put(5,50){$V((l+1) \tilde{\alpha})$}
\put(63,50){$ \into$}
\put(80,50){$\underbrace{V(\tilde{\alpha}) \otimes \cdots \otimes
V(\tilde{\alpha})}_{(l+1)\text{ times}}$}
\put(190,50){$\underbrace{\ul^{\ast} \otimes \cdots \otimes
\ul^{\ast}}_{l \text{ times}}\otimes\; \ta$}
\put(170,50){$\onto$}
\put(272,50){$ \onto$}
\put(290,50){$ S^l (\ul^{\ast})\otimes \;\ta$}
\put(280,5){$ \ind_B^G(S^l (\ul^{\ast})  \otimes \ta)$}
\put(310,30){$ \uparrow$}
\put(170, 35){\vector(4,-1){100}}
\end{picture}
\end{center}
\end{proof}


\subsection{}
For a $G$-module $V$ and a dominant weight $\ga$ let
$[V]_{\gamma}$ denote the unique maximal summand of $V$ whose
composition factors have highest weights linked to $\gamma$.

\begin{lemma} \label{L:Simplylacedcohom} Assume that the root
system $\Phi$ of $G$ is simply laced. Let $\ta$ denote the
longest root and define $\la = p \ta + s_{\ta}\cdot 0 = (p-h+1)
\ta. $ Then
\begin{itemize}
\item[(a)] for any non-zero dominant weight $\mu$ linked to zero
we have\\ $\opH^i(G,H^0(\mu)\otimes H^0(\mu^*)^{(r)}) = 0$
whenever $ i < r(2p - 3)$;
\item[(b)] for any non-zero dominant weight $\mu$ linked to zero
we have\\
$\Ext_{G/G_1}^k(V(\mu)^{(r)}, \opH^l(G_1,H^0(\mu)))= 0$ whenever
$k+l < r(2p-3)$;
\item[(c)] $[\opH^i(G_1, H^0(\la))^{(-1)}]_0 \cong 
\begin{cases} H^0(\la) & \mbox{ if  } i = 2p-3\\ 
0 & \mbox{ if } 0 < i < 2p-3.
\end{cases}$ 
\end{itemize}
\end{lemma}

\begin{proof} We apply Proposition~\ref{P:degreebound2}.  Note
that $\mu$ being linked to zero forces all the weights
$\delta_j$ of Proposition~\ref{P:degreebound2} to be in the root
lattice.
This forces $\langle \delta_j, \ta^{\vee} \rangle\geq 2$. Parts
(a) and (b) now follow from equation (\ref{rbound}) and the
remark in Section~\ref{SS:degreebounds}.

For part (c) we make use of Proposition~\ref{P:Homcalculation}
with $l+1 = p-h+1$ and conclude that
$$\Hom_G(V(\la), \ind_B^G( S^{(p-h)} (\ul^{\ast})\otimes \ta))
\cong k.$$
Note that in the simply laced case $\ell(s_{\ta}) = 2h-3$ which
combined with (\ref{ind}) yields $ \ind_B^G( S^{p-h}
(\ul^{\ast})\otimes \ta)) \cong \opH^{2p-3}(G_1,
H^0(\la)^{(-1)}).$

The weight $\la$ is the smallest non-zero dominant weight in the
zero linkage class. Any other non-zero weight $\mu$ in the
linkage class will be of the form $\mu =\la + \sigma= (p-h+1)\ta
+ \sigma$, where $\sigma$ is a non-zero sum of positive roots.
Clearly $\mu$ cannot be a weight of $S^{m} (\ul^{\ast})\otimes
\ta$ whenever $m \leq p-h.$ Hence,
$$\Hom_G(V(\mu),  \opH^{i}(G_1, H^0(\la))^{(-1)}) \cong
\Hom_G(V(\mu), \ind_B^G( S^{\frac{i-\ell(s_{\ta})}{2}}
(\ul^{\ast})\otimes \ta))=0$$
for all $0 < i \leq 2p-3.$ Since $\opH^{2p-3}(G_1,
H^0(\la))^{(-1)}$ has a good filtration one obtains
$$[\opH^{2p-3}(G_1, H^0(\la))^{(-1)}]_0 \cong H^0(\la).$$ 

Part (b) now implies that $[\opH^{i}(G_1, H^0(\la))]_0=0$
whenever $0 < i < 2p-3.$
\end{proof}

\begin{theorem} \label{T:firstnontrivial} Assume that the root
system of $G$ is simply laced. Then
$$\operatorname{H}^{r(2p-3)}(\gfpr, k) \neq 0.$$
\end{theorem}

\begin{proof} Let $\mu$ be a weight in the zero linkage class.
From Lemma \ref{L:Simplylacedcohom}(a), we know that
$\opH^i(G,H^0(\mu)\otimes H^0(\mu^*)^{(1)}) = 0$
for $i < r(2p - 3)$. Let $\la = (p - h + 1)\ta$. We next show by
induction on $r$ that
$\opH^{r(2p - 3)}(G,H^0(\la)\otimes H^0(\la^*)^{(r)}) \neq 0$.
For $r = 1$, this
follows from Lemma~\ref{L:G1cohom} and
Lemma~\ref{L:Simplylacedcohom}(c).

For $r > 1$, we look at the Lyndon-Hochschild-Serre spectral
sequence
$$E_2^{k,l}=\Ext_{G/G_1}^k(V(\la)^{(r)}, \opH^l(G_1,H^0(\la)))
\Rightarrow
        \Ext^{k+l}_G(V(\la)^{(r)}, H^0(\la)).$$
Lemma~\ref{L:Simplylacedcohom}(b) implies that the $E_2^{k,l} =
0$ for $k+l < r(2p-3)$.
Note that
$$
E_2^{k,l} =
\Ext_{G}^{k}(V(\la)^{(r-1)},\opH^l(G_1,H^0(\la))^{(-1)}).
$$
Lemma~\ref{L:Simplylacedcohom}(c) implies that $E_2^{k,l} = 0$
for $l < 2p - 3$,
and, moreover, that $H^0(\la)$ is a summand of
$\opH^{2p-3}(G_1,H^0(\la))^{(-1)}$.
Hence, (cf. \cite[Lemma 5.4]{BNP}),
$$
E_2^{(r-1)(2p-3),2p-3} =
\Ext_{G}^{(r-1)(2p-3)}(V(\la)^{(r-1)},\opH^{2p -
3}(G_1,H^0(\la))^{(-1)})
$$
has a summand isomorphic to 
$\Ext_G^{(r-1)(2p-3)}(V(\la)^{(r-1)},H^0(\la))$.
By induction, this Ext-group is non-zero, and hence
$E_2^{(r-1)(2p-3),2p-3}\neq 0$
and transgresses to the $E_{\infty}$-page,
which implies that 
$$\Ext_{G}^{r(2p - 3)}(V(\la)^{(r)},H^0(\la))\neq 0.$$

Since $\la$ is the lowest non-zero dominant weight in the zero
linkage class, the claim now
follows by applying the argument given in
Section~\ref{SS:nonvanishing} to the weight $\la$ and
the zero linkage class.
\end{proof}


\subsection{Finite groups of adjoint type.}\label{S:adjoint} In this section we assume that $G$ is simply laced and of adjoint type. The fixed points of the $r$th iterated Frobenius map on $G$ will again be denoted by $\gfpr$. For example, if $G$ is the adjoint group of type $A$ then $\gfpr$ is the projective linear group with entries in the field with $q$ elements. Propositions \ref{P:extcomparison} and \ref{P:sections} can also be applied to groups of adjoint type. Note that the root lattice and
the weight lattice coincide in this case. Therefore all dominant weights of the
form $p \delta + w\cdot 0$ are automatically in the zero linkage
class. From Lemma~\ref{L:Simplylacedcohom},
Theorem~\ref{T:firstnontrivial}, and
Proposition~\ref{P:vanishing}, one obtains the following
corollary.

\begin{corollary} \label{C:simplylacedadjoint} Assume that the
root system $\Phi$ of $G$ is simply laced and that $G$ is of
adjoint type. Then
\begin{itemize}
\item[(a)]$\operatorname{H}^{i}(\gfpr, k) = 0$ for $0 < i <
r(2p-3)$;
\item[(b)]$\operatorname{H}^{r(2p-3)}(\gfpr, k) \neq 0$.
\end{itemize}
\end{corollary}

Note that, for type $A_n$ and $q-1$ and $n+1$ being relatively prime, the adjoint and the universal types of the finite groups coincide.  In these cases the above claim was already observed in Theorems 6.5.1 and 6.14.1 of \cite{BNP}.

\begin{remark} Let $G$ be of type $A$, $D$, $E_6$ and of adjoint type. Let  $\sigma$ denote
an automorphism of the Dynkin diagram of $G$. Then $\sigma$  induces a group automorphism
of $G$ that commutes with the Frobenius morphism, which we will also
denote by  $\sigma$. Let $G_{\sigma}(\mathbb{F}_q)$ be the finite group consisting of the fixed points of  $\sigma$
composed with $F$. Note that $\sigma$ fixes the maximal root $\tilde{\alpha}$. Therefore the discussion in this section also applies to the twisted groups $G_{\sigma}(\mathbb{F}_q)$ of adjoint type. In particular, Corollary ~\ref{C:simplylacedadjoint} holds for these groups as well.
\end{remark}


\section{Type $D_n$, $n \geq 4$}

Assume throughout this section that $\Phi$ is of type $D_n$, $n
\geq 4$, and that
$p > h = 2n - 2$. Following Section 2, our goal is to find the
least $i > 0$
such that $\opH^i(G,H^0(\la)\otimes H^0(\la^*)) \neq 0$ for some
$\la$.


\subsection{\bf Restrictions.} \label{SS:initialrestrictions}
Suppose that
$\opH^i(G,H^0(\la)\otimes H^0(\la^*)^{(1)}) \neq 0$ for some
$i > 0$ and $\la = p\mu + w\cdot 0$ with $\mu \in X(T)_{+}$ and
$w \in W$.
From Proposition~\ref{P:degreebound1}(c), $i \geq
(p-1)\langle\mu,\ta^{\vee}\rangle - 1$.

For a fundamental dominant weight $\omega_j$,
$$
\langle\omega_j,\ta^{\vee}\rangle = 
\begin{cases}
1 &\text{ if } j = 1, n - 1, n\\
2 &\text{ if } 2 \leq j \leq n - 2.
\end{cases}
$$
Therefore, if $\mu \neq \omega_1, \omega_{n-1}, \omega_n$, we
will have
$\langle\mu,\ta^{\vee}\rangle \geq 2$ and $i \geq 2p - 3$. This
reduces us to
analyzing the cases when $\mu=\omega_1, \omega_{n-1}, \omega_n$,


\subsection{\bf The case of $\omega_1$.} \label{SS:caseomega1}
We consider first
the case that $\la = p\omega_1 + w\cdot 0$ and obtain the
following
restrictions.

\begin{lemma}\label{L:omega1} Suppose $\Phi$ is of type $D_n$
with $n \geq 4$ and
$p > 2n - 2$.
Suppose $\la = p\omega_1 + w\cdot 0 \in X(T)_+$ with $w \in W$.
Then
\begin{itemize}
\item[(a)] $\opH^i(G,H^0(\la)\otimes H^0(\la^*)^{(1)}) = 0$ for
$0 < i < 2p - 2n$;
\item[(b)] if $\opH^{2p - 2n}(G,H^0(\la)\otimes
H^0(\la^*)^{(1)}) \neq 0$,
then $\la = p\omega_1 - (2n - 2)\omega_1 = (p - 2n +
2)\omega_1$.
\end{itemize}
\end{lemma}

\begin{proof} Following the discussion in Section \ref{SS:G1analysis}, $\la -
\omega_1
= (p - 1)\omega_1 + w\cdot 0$ must be a weight of $S^j(\ul^*)$
for
$j = \frac{i - \ell(w)}{2}$.  Recall that 
$\omega_1 = \al_1 + \al_2 + \cdots + \al_{n-2} +
\frac12\al_{n-1} + \frac12\al_{n}$.
Consider the decomposition of $-w\cdot 0$ into a sum of
$\ell(w)$ distinct positive roots
(see Observation \ref{O:uniquedecomp}). Write $\ell(w) = a + b$ where
$a$ is
the number of positive roots in this decomposition that contain
$\al_1$ and
$b$ is the number of roots in this decomposition that do not
contain $\al_1$.
Then $\la - \omega_1$ contains $p - 1 - a$ copies of $\al_1$.
Since any root contains at most one copy of $\al_1$, we
have
$$
\frac{i - \ell(w)}{2} = j \geq p - 1 - a.
$$
Replacing $\ell(w)$ by $a + b$ and simplifying gives
$$
i \geq 2p - 2 - a + b.
$$
The total number of positive roots containing an
$\al_1$ is $2n - 2$. Since we necessarily then have $a \leq 2n -
2$, we can
rewrite the above as
\begin{align*}
i &\geq 2p - 2 - (2n - 2) + b\\
  &= 2p - 2n + b\\
  &\geq 2p - 2n
\end{align*}
since $b \geq 0$. This proves part (a). Furthermore, we see that
$i = 2p - 2n$
if and only if $b = 0$ and $a = 2n - 2$. In other words, when
$-w\cdot 0$ is
expressed as a sum of distinct positive roots, it consists
precisely of all $2n -2$
roots which contain an $\al_1$. That is, $-w\cdot 0 = (2n -
2)\omega_1$, which
gives part (b).
\end{proof}


\subsection{\bf The case of $\omega_1$ continued.} We will show
in Proposition~\ref{P:omega1summary}
that $$\opH^{2p - 2n}(G,H^0(\la)\otimes H^0(\la^*)^{(1)}) \neq
0$$
for $\la = (p - 2n + 2)\omega_1$.
To do this, we will make use of
Proposition~\ref{P:KostantPartCohom}. We first make some
observations about
relevant partition functions. Note that $\omega_1 = \epsilon_1.$

For $\Phi$ of type $D_n$, with $n \geq 4$, and integers $m, k$,
we set
$$P(m,k,n) := \begin{cases} \sum_{u \in W} (-1)^{\ell (u) }
P_k(u\cdot m \epsilon_1)
 & \mbox{ if } m \geq 1, k\geq 0, \\
 1& \mbox{ if } m =0, k= 0, \\
 0 &\mbox{ else.}
 \end{cases}   $$
Note that 
$$P(m,k,n)=\dim \Hom_G(V(m\epsilon_1),H^0(G/B, S^{k}(\ul^*))) =
[\ch H^0(G/B, S^{k}(\ul^*)):\ch H^0(m\epsilon_1)],$$
when  $m \geq 0, k\geq 0, n \geq 4$.

\begin{lemma} \label{L:omega1combinatorics} Suppose $\Phi$ is of
type $D_n$ with $n \geq 4$ and $m \geq 0$.
\begin{itemize}
\item[(a)] $P(m,k,n) =0$ whenever $k<m$.
\item[(b)] $\sum_{u \in W} (-1)^{\ell (u) } P_k(u\cdot m
\epsilon_1-\epsilon_1) =
P(m-1,k,n). $ 
\item[(c)] $\sum_{u \in W} (-1)^{\ell (u) } P_k(u\cdot m
\epsilon_1+\epsilon_1) = P(m+1,k,n) - P(m+1,k,n-1)$, for $n
\geq 5.$
\item[(d)] 
$
\sum_{u \in W} (-1)^{\ell (u) } P_k(u\cdot m
\epsilon_1+\epsilon_1)
= P(m-1,k-2n+2,n).
$
\item[(e)] $P(m,k,n) = P(m,k,n-1) + P(m-2,k-2n+2,n)$, for $m
\geq 0, n \geq 5.$
\item[(f)] $P(m, m,n) = 1,$ for $n\geq 4$ and $m $ even.
\end{itemize}
\end{lemma}

\begin{proof}
(a) The weight $\omega_1= \epsilon_1= \frac{1}{2}(\alpha_1 +
\ta)$, written as a sum of simple roots, contains one copy of
$\alpha_1$. Assume that $0 \leq k < m$ and
$n \geq 4$. Recall that $P(m,k,n) = \sum_{u \in W} (-1)^{\ell
(u) } P_k(u\cdot m \epsilon_1)$.
Note that $u\cdot m \epsilon_1 = u(m\epsilon_1) + u\cdot 0$ is
a sum of positive roots if and only if $u( \epsilon_1) =
\epsilon_1.$ Therefore, $P(m,k,n) = \sum_{\{u \in
W|u(\epsilon_1) = \epsilon_1\}} (-1)^{\ell (u) }
P_k(u\cdot m \epsilon_1)$. If $u( \epsilon_1) = \epsilon_1$
then $u\cdot m \epsilon_1 = m\epsilon_1 + u\cdot 0.$ In
addition $-u \cdot 0$, written as a sum of simple roots,
contains no $\alpha_1$. Hence, $u\cdot m \epsilon_1$, written
as a sum of simple roots, contains exactly $m$ copies of
$\alpha_1$. Each positive root of $\Phi$ contains at most one
copy of $\alpha_1$. Therefore at least $m$ positive roots are
needed to sum up to $u\cdot m \epsilon_1$. One concludes that
$P_k(u\cdot m \epsilon_1)=0$ for $k < m$.
\\\\
(b) Again $u\cdot m \epsilon_1- \epsilon_1$ is a sum of
positive roots only if $u(\epsilon_1) = \epsilon_1$. Therefore,
\begin{eqnarray*}
\sum_{u \in W} (-1)^{\ell (u) } P_k(u\cdot m \epsilon_1-
\epsilon_1) &=&\sum_{\{u \in W|u(\epsilon_1)
= \epsilon_1\}} (-1)^{\ell (u) } P_k(u\cdot m \epsilon_1-
\epsilon_1)\\
&=& \sum_{\{u \in W|u(\epsilon_1)=\epsilon_1\}} (-1)^{\ell (u) }
P_k((m-1) \epsilon_1+ u\cdot 0) \\
&=& \sum_{\{u \in W|u(\epsilon_1)=\epsilon_1\}} (-1)^{\ell (u) }
P_k(u \cdot (m-1) \epsilon_1) \\
  &=&P(m-1,k,n).
 \end{eqnarray*}
(c) For the expression $u\cdot m \epsilon_1+\epsilon_1$ to be a
sum of positive roots one needs either $u(\epsilon_1)
=\epsilon_1$ or $u(\epsilon_1)= \epsilon_2$ and $ u(\epsilon_2)=
\epsilon_1$.
Set $A=\sum_{u \in W} (-1)^{\ell (u) } P_k(u\cdot m
\epsilon_1 +\epsilon_1)$
Then
\begin{eqnarray*}
A&=&\sum_{\{u \in W|u(\epsilon_1)= \epsilon_1\}} (-1)^{\ell (u)
} P_k(u\cdot m \epsilon_1+\epsilon_1) + \sum_{\{u \in
W|u(\epsilon_1)= \epsilon_2, u(\epsilon_2)= \epsilon_1\}}
(-1)^{\ell (u) } P_k(u\cdot m \epsilon_1+\epsilon_1) \\
&=&\sum_{\{u \in W|u(\epsilon_1)= \epsilon_1\}} (-1)^{\ell (u) }
P_k((m+1) \epsilon_1+u\cdot 0)\\
&\ &\ \ \ \ \ \ \ \ \ \ \ \ \ \ \ \ \ \ \ \ \ \ \ \ \ \ +
\sum_{\{u \in W|u(\epsilon_1)= \epsilon_1, u(\epsilon_2)=
\epsilon_2\}} (-1)^{\ell (u)+1 } P_k(s_{\alpha_1}m \epsilon_1
+u \cdot 0 -\alpha_1+ \epsilon_1) \\
&\  &    \   \\
&=&\sum_{\{u \in W|u(\epsilon_1)= \epsilon_1\}} (-1)^{\ell (u) }
P_k(u \cdot (m+1) \epsilon_1) - \sum_{\{u \in W|u(\epsilon_1)=
\epsilon_1, u(\epsilon_2)= \epsilon_2\}} (-1)^{\ell (u) } P_k(u
\cdot (m+1)\epsilon_2) \\
&=&P(m+1,k,n) - P(m+1,k,n-1).
\end{eqnarray*}  
\\\\
(d) We make use of the fact that $\omega_1 = \epsilon_1$ is a
minuscule weight and obtain:
\begin{eqnarray*}
A&=&[(\sum_{i \geq 0} (-1)^i \ch H^i(G/B, S^k(\ul^*)\otimes -
\epsilon_1)):\ch H^0(m\epsilon_1)]\\
&=&[ \ch H^0(G/B, S^k(\ul^*)\otimes - \epsilon_1)):\ch
H^0(m\epsilon_1)]\\
&=&[ \ch H^0(G/B, S^{k-2n+2}(\ul^*)\otimes \epsilon_1)):\ch
H^0(m\epsilon_1)]
	\qquad (\text{by \cite[Lemma 6]{KLT}})\\
&=&\sum_{u \in W} (-1)^{\ell (u) } P_{k-2n-2}(u\cdot m
\epsilon_1- \epsilon_1)
		\qquad (\text{by \cite[3.8]{AJ}})\\
&=& P(m-1,k-2n+2,n) \qquad (\text{by (b)}).
\end{eqnarray*} 
\\\\
Part (e) now follows directly from (c) and (d).
\\\\
(f) If $n \geq 5$, it follows from (e) and (a) that
$P(m,m,n)=P(m,m,n-1).$
So the claim holds if it holds for $n = 4$.  
If $n=4$ and $k=m$, then (e) has to be replaced by 
$$P(m,m,4) = \sum_{\{u \in W|u(\epsilon_1)= \epsilon_1,
u(\epsilon_2)= \epsilon_2\}} (-1)^{\ell (u) } P_m(u \cdot m
\epsilon_2).$$
Note that the both sides of the equation are zero unless $m$ is
even.
For even $m$, a direct computation similar to that in the proof
of \cite[Lemma 6.11]{BNP} shows that
$$\sum_{\{u \in W|u(\epsilon_1)= \epsilon_1, u(\epsilon_2)=
\epsilon_2\}} (-1)^{\ell (u) } P_m(u \cdot m\epsilon_2)= 1.$$
\end{proof}

\begin{proposition} \label{P:omega1summary} Suppose $\Phi$ is of
type $D_n$ with $n \geq 4$. Assume that $p > 2n-2$. Let $\la =
(p - 2n+2) \omega_1.$ Then
$$\opH^{2p-2n}(G, H^0(\la) \otimes H^0(\la)^{(1)} )= k.$$
\end{proposition}

\begin{proof}From the previous discussion we know that $\la = (p
- 2n +2) \omega_1 $ is of the form $p\omega_1 + w \cdot 0$ with
$\ell(w)=2n-2$. Set $k = (i-l(w))/2$. From
Proposition~\ref{P:KostantPartCohom} and
Lemma~\ref{L:omega1combinatorics}(b), one
concludes
$$\dim \opH^i(G, H^0(\la) \otimes H^0(\la)^{(1)} )=[\ch H^0(G/B,
S^k(\ul ^*)\otimes \omega_1) : \ch H^0(\la)]= P(p-2n+1,k,n).$$
By Lemma~\ref{L:omega1combinatorics}(a), this expression is zero
unless $k\geq p-2n+1$ and Lemma~\ref{L:omega1combinatorics}(d) it follows that
$P(p-2n+1,p-2n+1,n)=1$. Replacing $k$ by $(i-2n+2)/2$ and
solving for $i$ yields the claim.
\end{proof}


\subsection{\bf The case of $\omega_{n-1}$ and $\omega_n$.} 
We now consider the case that $\la = p\omega_{n-1} + w\cdot 0$
or
$\la = p\omega_n + w\cdot 0$ for $w \in W$ with $\la \in
X(T)_{+}$.

\begin{lemma} \label{L:omegan}Suppose $\Phi$ is of type $D_n$
with $n \geq 4$ and
$p > 2n - 2$.
Suppose $\la = p\omega_{n-1} + w\cdot 0 \in X(T)_+$ or
$\la = p\omega_{n} + w\cdot 0 \in X(T)_{+}$ with $w \in W$,
and $\opH^i(G,H^0(\la)\otimes H^0(\la^*)^{(1)}) \neq 0$ for $i
\neq 0$. Then
\begin{itemize}
\item[(a)] $\displaystyle i \geq \frac{(p - n)n}{2}$;
\item[(b)] if $n \geq 5$, then $i \geq 2p - 2n + 2$.
\end{itemize}
\end{lemma}

\begin{proof} We consider the case of $\omega_n$. By symmetry,
the case of $\omega_{n-1}$ can
be dealt with in a similar manner. Following the discussion in
Section~\ref{SS:G1analysis},
$\la - \omega_n = (p - 1)\omega_n + w\cdot 0$ must be a weight
of $S^j(\ul^*)$ for
$j = \frac{i - \ell(w)}{2}$.  Recall that 
$$\omega_n = \frac12(\al_1 + 2\al_2 + 3\al_3 + \cdots +
(n-2)\al_{n-2}) + \frac{(n-2)}{4}\al_{n-1} +
\frac{n}{4}\al_{n}.$$
Consider the decomposition of $-w\cdot 0$ into a sum of distinct
positive roots
(cf. Observation \ref{O:uniquedecomp}). Write $\ell(w) = a + b$
where $a$ is
the number of positive roots in this decomposition which contain
$\al_n$ and
$b$ is the number of roots in this decomposition which do not
contain $\al_n$.
Then $\la - \omega_n$ contains $\frac{(p-1)n}{4} - a$ copies of
$\al_n$.
Since any root contains at most one copy of $\al_n$, we have
$$
\frac{i - \ell(w)}{2} = j \geq \frac{(p-1)n}{4} - a.
$$
Substituting $\ell(w) = a + b$, rewriting, and simplifying, we
get
$$
i \geq \frac{(p-1)n}{2} - a + b.
$$
The total number of positive roots containing
$\al_n$ is $\frac{(n-1)n}{2}$. Since we necessarily have $a \leq
\frac{(n-1)n}{2}$
and $b \geq 0$, we get
\begin{align*}
i &\geq \frac{(p-1)n}{2} - \frac{(n-1)n}{2} + b\\
  &= \frac{(p - n)n}{2} + b\\
  &\geq \frac{(p - n)n}{2}
\end{align*}
which gives part (a).

For part (b), assume that $n \geq 5$.  We want to show that
$$
\frac{(p-n)n}{2} \geq 2p - 2n + 2.
$$
This is equivalent to showing that $(p - n)n \geq 4p - 4n + 4$.  Consider the left hand side:
$$
(p - n)n = np - n^2 = 4p + (n - 4)p - n^2.
$$
Hence the problem is reduced to showing that $(n - 4)p - n^2
\geq -4n + 4$
or $(n - 4)p - n^2 + 4n - 4 \geq 0$. Since $p \geq 2n - 1$, we
have
\begin{align*}
(n-4)p - n^2 + 4n - 4 &\geq(n - 4)(2n - 1) - n^2 + 4n - 4\\
        &= n^2 - 5n = n(n - 5) \geq 0
\end{align*}
since $n \geq 5$.  Part (b) follows.
\end{proof}

Note that if $n = 4$,
$$
\frac{(p - n)n}{2} = \frac{4(p - 4)}{2} = 2p - 8 = 2p - 2n.
$$


\subsection{\bf Summary for type $D$.} 

The following two theorems summarize our
findings when the root system is of type $D_{n}$.

\begin{theorem}\label{T:Dsummary1} Suppose $\Phi$ is of type
$D_n$ with $n \geq 4$. Assume that $p > 2n-2$. Then
\begin{itemize}
\item[(a)] $\opH^{i}(\gfp, k )  = 0$ for $0 <i < 2p-2n$;  
\item[(b)] $\opH^{2p-2n}(\gfp, k ) = \begin{cases} k \mbox{ if }
n \geq 5\\
k\oplus k \oplus k \mbox{ if } n =4.
\end{cases}$
\end{itemize}
\end{theorem}

\begin{proof} Part (a) follows from
Section~\ref{SS:initialrestrictions}, Lemma~\ref{L:omega1}(a),
Lemma~\ref{L:omegan},
and Proposition~\ref{P:vanishing}.

For part (b), when $n \geq 5$, it follows from
Section~\ref{SS:initialrestrictions}, Lemma~\ref{L:omega1},
Proposition~\ref{P:omega1summary}
and Lemma~\ref{L:omegan} that $\la= (2p -2n +2)\omega_1$ is the
only dominant weight with
$\opH^{2p - 2n}(G,H^0(\la)\otimes H^0(\la^*)^{(1)}) \neq 0$.
Since $\la$
is the lowest weight in its linkage class, the claim follows
from
Theorem~\ref{T:nonvanishing}. For $n=4$ the symmetry of the root
system yields
$\opH^{2p-2n}(G, H^0(\la) \otimes H^0(\la)^{(1)} ) = k$
for the weights $\la =(p - 6) \omega_1$, $\la =(p- 6) \omega_3$
and
$\la =(p -6) \omega_4$, and those are the only weights with
non-zero $G$-cohomology
in degree $2p - 2n$. Each weight is minimal in its own linkage
class. The claim follows.
\end{proof}

Working inductively from the $r = 1$ case, we can obtain sharp
vanishing bounds
for arbitrary $r$.

\begin{theorem} \label{T:Dsummary2} Suppose $\Phi$ is of type
$D_n$ with $n \geq 4$. Assume that $p > 2n-2$. Then
\begin{itemize}
\item[(a)] $\opH^{i}(\gfpr, k )  = 0$ for $0 <i < r(2p-2n)$;  
\item[(b)] $\opH^{r(2p-2n)}(\gfpr, k ) = \begin{cases} k \mbox{
if } n \geq 5\\
k\oplus k \oplus k \mbox{ if } n = 4.
\end{cases}$
\end{itemize}
\end{theorem}

\begin{proof} For part (a), we need to show that
$\opH^i(G,H^0(\la)\otimes H^0(\la^*)^{(r)}) = 0$
for $0 < i < r(2p - 2n)$ and $\la \in X(T)_+$. If that is true,
then the claim follows
from Proposition~\ref{P:vanishing}. For part (b), we require
precise information
on those dominant weights $\la$ for which $\opH^{r(2p -
2n)}(G,H^0(\la)\otimes H^0(\la^*)^{(r)}) \neq 0$. The root
lattice
$\mathbb{Z}\Phi$ has four cosets within $X(T)$:
$\mathbb{Z}\Phi$,
$\{\omega_{1}+\mathbb{Z}\Phi\}$,
$\{\omega_{n-1}+\mathbb{Z}\Phi\}$,
and $\{\omega_{n}+\mathbb{Z}\Phi\}$.  
If $\la$ is a weight in the root lattice claim (a) follows from
Lemma~\ref{L:Simplylacedcohom}(a).
Furthermore, no such weights can contribute to cohomology in
degree $r(2p - 2n)$.

Assume that 
$\opH^i(G,H^0(\la)\otimes H^0(\la^*)^{(r)}) \neq 0$ for some $i
> 0$ and apply Proposition~\ref{P:degreebound2}.
Suppose first that $\la = p \delta_0 +u_0 \cdot 0$ with
$\delta_0=\delta_r \in \{\omega_{1}+\mathbb{Z}\Phi\} $.
For each $1 \leq j \leq r$, in order to have 
$\Ext_G^{l_j}(V(\ga_j)^{(1)},H^0(\ga_{j-1})) \neq 0$, then 
$\ga_j - \delta_{j - 1} = p\delta_j + u_j\cdot 0 - \delta_{j -
1}$ must
be a weight in $S^{\frac{l_j - \ell(u_{j-1})}{2}}(\ul^*)$. This
implies that
$p\delta_j - \delta_{j - 1}$ must lie in the positive root
lattice.
Since $\delta_r \in \{\omega_{1}+\mathbb{Z}\Phi\}$, we
necessarily have
$p\delta_r \in \{\omega_{1}+\mathbb{Z}\Phi\} $.  Since
$p\delta_r - \delta_{r-1} \in \mathbb{Z}\Phi$, it then follows
that
$\delta_{r-1} \in \{\omega_{1}+\mathbb{Z}\Phi\}$. Inductively
one concludes
that $\delta_j \in \{\omega_{1}+\mathbb{Z}\Phi\}$ for all $j$. 

For a weight $\ga$, when expressed as a sum of simple roots, let
$N(\ga)$ denote the
number of copies of $\al_1$ that appear. Since a positive root
contains at most one
copy of $\al_1$, we have
$$
\frac{l_j - \ell(u_{j-1})}{2} \geq pN(\delta_j) - N(-u_j\cdot 0)
- N(\delta_{j-1}).
$$
From Observation~\ref{O:uniquedecomp}, we know that $-u_j\cdot
0$ can be expressed uniquely as
$\ell(u_j)$ distinct positive roots. Write $\ell(u_j) = a_j +
b_j$ where $a_j$ denotes
the number of roots containing an $\al_1$ and $b_j$ the number
that do not. Then
$N(u_j\cdot 0) = -a_j$, and rewriting the above gives
$$
l_j \geq 2pN(\delta_j) - 2N(\delta_{j - 1}) - 2a_j + a_{j-1} +
b_{j - 1}.
$$
Hence, summing over $j$ gives
$$
i = \sum_{j = 1}^r l_j \geq 
\sum_{j = 1}^r(2p - 2)N(\delta_j) - \sum_{j = 1}^r a_j + \sum_{j
= 1}^r b_j.
$$
Recall that the $\delta_j$ are non-zero dominant weights. By the
assumption on $\delta_j$, $N(\delta_j) \geq 1$. The total number
of
positive roots containing an $\al_1$ is $2n -2$. Hence, $a_j
\leq 2n - 2$.
With this, we get
\begin{equation}\label{D.5cond}
i \geq r(2p - 2) - r(2n - 2) + \sum_{j = 1}^r b_j = r(2p - 2n) +
\sum_{j = 1}^r b_j\\
	\geq r(2p - 2n),
\end{equation}
since $b_j \geq 0$. This gives the necessary condition for part
(a) for the coset
$\{\omega_1 + \mathbb{Z}\Phi\}$. Before considering the
remaining two cosets,
towards addressing part (b), we consider when equality can hold
in (\ref{D.5cond}).

As in Section~\ref{SS:caseomega1} we see that equality holds in
(\ref{D.5cond}) if and only if
$N(\delta_j)=1$ and $\ell(u_j) = 2n-2,$ which forces
$\la=\gamma_j = (p-2n+2)\omega_1$
for all $j$. Moreover, one obtains from the discussion above and
Proposition~\ref{P:omega1summary}
for $\la = (p-2n+2)\omega_1$ that
\begin{equation}\label{D2p2n}
[H^i(G_1, H^0(\la))^{(-1)}]_{\la} \cong 
\begin{cases} H^0(\la) & \mbox{ if  } i = 2p-2n\\ 
0 & \mbox{ if } 0 < i < 2p-2n.
\end{cases}
\end{equation}
Using the spectral sequence argument in the proof of
Theorem~\ref{T:firstnontrivial}
(see also the proof of \cite[Lemma 5.4]{BNP}), we can show that
$\opH^{r(2p - 2n)}(G,H^0(\la)\otimes H^0(\la^*)^{(r)}) = k$. We prove this by induction on $r$ with the $r = 1$ case being
Theorem~\ref{T:Dsummary1}.
Consider the Lyndon-Hochschild-Serre spectral sequence
\begin{align*}
E_2^{k,l} &= \Ext_{G/G_1}^k(V(\la)^{(r)},\opH^l(G_1,H^0(\la)))\\&\cong \Ext_{G}^k(V(\la)^{(r-1)},\opH^l(G_1,H^0(\la))^{(-1)})
\Rightarrow
	\Ext_{G}^{k + l}(V(\la)^{(r)},H^0(\la)).
\end{align*}
From the remarks in Section~\ref{SS:degreebounds} and the
discussion above, $E_2^{k,l} = 0$ for $k + l < r(2p - 2n)$.
Furthermore, from (\ref{D2p2n}), $E_2^{k,l} = 0$ for $l < 2p -
2n$.
Finally, if $E_2^{k,l} \neq 0$ and $k + l = r(2p - 2n)$, then,
from the
above conclusion that $\ga_j = \la$ for each $j$, we must have
$l = 2p - 2n$.
Hence, the $E_2^{(r-1)(2p - 2n),2p - 2n}$-term survives to
$E_{\infty}$
and is the only term to contribute in degree $r(2p - 2n)$.
Hence,
by (\ref{D2p2n}) and our inductive hypothesis, 
\begin{align*}
\Ext_{G}^{r(2p - 2n)}(V(\la)^{(r)},H^0(\la)) &\cong
\Ext_{G}^{(r-1)(2p - 2n)}(V(\la)^{(r-1)},\opH^{2p -
2n}(G_1,H^0(\la))^{(-1)})\\
& \cong
\Ext_{G}^{(r-1)(2p - 2n)}(V(\la)^{(r-1)},H^0(\la)) \cong k.
\end{align*}

To complete the proof of (a) we need to consider the case that
$\la=p\delta_0+u_0
\cdot 0$ with $\delta_0=\delta_r \in
\{\omega_{n-1}+\mathbb{Z}\Phi\}\cup
\{\omega_n+\mathbb{Z}\Phi\}$.
As above, $p\delta_j + u_j\cdot 0 - \delta_{j - 1}$
must lie in the positive root lattice. Since $u_j\cdot 0$ does,
this implies that
$p\delta_j - \delta_{j - 1}$ must lie in the positive root
lattice. When expressed
as a sum of simple roots $\omega_{n-1} = \frac12\al_1 + \cdots$
(as does $\omega_n$).
Whereas, for $1 \leq j \leq n - 2$, $\omega_j = \al_1 + \cdots$.  Since $\delta_0 \in \{\omega_{n-1}+\mathbb{Z}\Phi\}\cup
\{\omega_n+\mathbb{Z}\Phi\}$,
for $p\delta_1 - \delta_0$ to lie in the positive root lattice,
when expressed
as a sum of fundamental weights, $p\delta_{1}$ must contain an
odd number of copies
of $\omega_{n-1}$ and $\omega_n$ in total. Since $p$ is odd,
this also holds for
$\delta_1$.  Inductively, every $\delta_j$ has this property.  

We may assume therefore that each $\delta_j$ contains at least
one copy of $\omega_n$ or
one copy of $\omega_{n-1}$. Proceed as above, but let
$N_{\al_n}(\ga)$ and $N_{\al_{n-1}}(\ga)$ denote the number of
copies of
$\al_n$ and $\al_{n-1}$, respectively, appearing in $\ga$. Set
$N(\ga)=\max\{N_{\al_{n}}(\ga), N_{\al_{n-1}}(\ga)\}$.
Note that, for the weights $\ga$ that appear in what follows,
both
$N_{\al_{n}}(\ga)$ and $N_{\al_{n-1}}(\ga)$ are nonnegative.
Again, a positive root contains at most one 
copy of $\al_n$ or $ \al_{n-1}$. 
Just as above, we get
\begin{eqnarray*}
l_j &\geq& 2pN_{\al_n}(\delta_j) - 2N_{\al_n}(-u_j\cdot 0) -
2N_{\al_n}(\delta_{j-1})+\ell(u_{j-1})
\end{eqnarray*}
and the corresponding dual statement for $\al_{n-1}.$
By choosing the appropriate root we obtain
\begin{eqnarray*}
l_j &\geq& 2pN(\delta_j) - 2N_{\al_n}(-u_j\cdot 0) -
2N_{\al_n}(\delta_{j-1})+\ell(u_{j-1})
\end{eqnarray*}
or
\begin{eqnarray*}
l_j &\geq& 2pN(\delta_j) - 2N_{\al_{n-1}}(-u_j\cdot 0) -
2N_{\al_{n-1}}(\delta_{j-1})+\ell(u_{j-1}).
\end{eqnarray*}
Either one will result in
\begin{eqnarray*}
l_j &\geq& 2pN(\delta_j) - 2N(-u_j\cdot 0) -
2N(\delta_{j-1})+\ell(u_{j-1}).
\end{eqnarray*}
From earlier arguments we know that $\ell(u_{j-1}) \geq
N(-u_{j-1}\cdot 0)$. Hence
\begin{eqnarray*}
l_j &\geq& 2pN(\delta_j) - 2N(-u_j\cdot 0) -
2N(\delta_{j-1})+N(-u_{j-1}\cdot 0).
\end{eqnarray*}
Summing over $j$, one obtains
$$
i = \sum_{j = 1}^r l_j \geq 
\sum_{j = 1}^r\left[(2p - 2)N(\delta_j) - N(-u_j\cdot 0)\right].$$
Clearly, $N(-u_j\cdot 0) \leq N(-w_0\cdot 0) =
\frac{n(n-1)}{2}.$ Moreover, we can say that $N(\delta_j) \geq
\frac{n}{4}$.
Substituting this gives
\begin{align*}
i &\geq r(2p - 2)\left(\frac{n}{4}\right) - \frac{rn(n-1)}{2}\\
	&= r\left(\frac{(p-1)(n)}{2} - \frac{n(n-1)}{2}\right)\\
	&\geq r(2p - 2n),
\end{align*}
where the last inequality follows as in the proof of
Lemma~\ref{L:omegan}.
Thus part (a) follows.  For $n \geq 5$, the last inequality is 
strict. Hence, $\la = (p - 2n + 2)\omega_1$ is the only dominant
weight for which $\opH^{r(2p - 2n)}(G,H^0(\la)\otimes
H^0(\la^*)^{(r)}) \neq 0$.
As in the proof of Theorem~\ref{T:Dsummary1}, since $\la$ is
minimal in its linkage class,
part (b) follows. Similarly, for $n = 4$, by symmetry, part (b)
follows.
\end{proof}


\section{Type $E$}


\subsection{\bf Type $E_6$.} 
Assume for this subsection that $\Phi$ is of type $E_6$ with
$p > h = 12$ (so $p \geq 13$).
The only dominant weights $\mu$ 
with $\langle \mu, \ta^{\vee} \rangle < 2$ are $\omega_1$ and
$\omega_6$. One concludes from Proposition~\ref{P:degreebound1}
and Proposition~\ref{P:vanishing} that
$\operatorname{H}^i(\gfp, k) =0$ for all $0 < i < 2p-3$ unless
there exists a weight
$\la$ of the form $p \omega_1 + w \cdot 0$ or of the form $p
\omega_6 + w \cdot 0$ with
$\opH^i(G, H^0(\la) \otimes H^0(\lambda^*)^{(1)}) \neq 0$ for
some $ 0 < i < 2p-3.$

\begin{lemma} \label{L:E6vanishing}
Suppose $\Phi$ is of type $E_6$, $p \geq 13$ and $\la \in X(T)_+$
is of the form $p \omega_1 + w \cdot 0$ or $p \omega_6 + w \cdot
0$ with $w \in W$. Assume in addition that $p \neq 13$, $19$.
Then $\opH^i(G, H^0(\la) \otimes H^0(\lambda^*)^{(1)}) = 0$ for
all $ 0 < i < 2(p-1).$
\end{lemma}

\begin{proof} We prove the assertion for $\la =p \omega_1 + w
\cdot 0, w \in W.$ Let $N$ denote the number of times that
$\alpha_1$ appears in $- w\cdot 0$ when written as a sum of
simple roots. Note that all positive roots of $\Phi$ contain the
simple root $\alpha_1$ at most once. This implies that $N \leq
\ell(w)$. Moreover, there are exactly $16$ distinct positive
roots containing $\alpha_1$. Hence, $N \leq 16$.

Using $\omega_1 = 1/3(4\alpha_1 + 3 \alpha_2+ 5 \alpha_3 +
6\alpha_4 + 4\alpha_5 + 2\alpha_6)$, we note that
$\la-\omega_1$, written as a sum of simple roots contains at
least $4/3(p-1) - N$ copies of $\alpha_1$. From
Section~\ref{SS:G1analysis}
we know that $\opH^i(G, H^0(\la) \otimes H^0(\lambda^*)^{(1)})
\neq 0$ and $i>0$ imply that $\la-\omega_1$ is a sum of
$(i-\ell(w))/2$ many positive roots. Note that this can only
happen if $(p-1)$ is divisible by $3$. Again using the fact that
$\alpha_1$ appears at most once in each positive root, one
obtains the inequality:
$$\frac{4}{3}(p-1)-N \leq \frac{i- \ell(w)}{2}.$$
Solving for $i$ yields
$$i \geq \frac{8}{3}(p-1)- 2N+\ell(w)\geq 2(p-1)
+\frac{2}{3}(p-1)-N \geq 2(p-1)+ \frac{2}{3}(p-1)-16.$$
Note that equality holds if and only if $N= \ell(w) = 16$. 

One obtains the desired claim $i \geq 2(p-1)$ for all primes
except those of the form $p = 3t+1$ with $13 \leq p \leq 25,$ i.e.,
the primes $p=13$ and $p=19.$
\end{proof}

\begin{theorem} \label{T:E6p-results} Suppose $\Phi$ is of type
$E_6$ and $p\geq 13$.
\begin{itemize}
\item[(a)] If $p \neq 13$, then
\begin{itemize}
\item[(i)] $\opH^{i}(\gfp, k )  = 0$ for $0 <i < 2p-3$;  
\item[(ii)] $\opH^{2p-3}(\gfp, k ) \neq 0.$
\end{itemize}
\item[(b)] If $p = 13$, then
\begin{itemize}
\item[(i)] $\opH^{i}(\gfp, k )  = 0$ for $0 <i < 16$;  
\item[(ii)] $\opH^{16}(\gfp, k ) \neq 0.$
\end{itemize}
\end{itemize}
\end{theorem}

\begin{proof}
For $p \neq 19$, part (a) follows immediately from Lemma~\ref{L:E6vanishing},
Proposition~\ref{P:vanishing}, and
Theorem~\ref{T:firstnontrivial}.

For the proof of part (b), set $p=13$. Part (i) follows from the
proof of Lemma \ref{L:E6vanishing}. Let $W_I$ denote the subgroup of $W$ generated by the
simple reflections $s_{\alpha_2} , ... , s_{\alpha_6}$ and let
$w$ denote the distinguished representative of the left coset
$w_0W_I$. Then $\ell(w) =16 $ and $-w \cdot 0$ equals the sum of
all positive roots in $\Phi$ that contain $\alpha_1$, which
equals the weight $12\omega_1$. Let $\la = p\omega_1 + w \cdot
0= \omega_1.$ Clearly,
\begin{eqnarray*}
k \cong \Hom_G(V(\la), H^0(\la)) &\cong& \Hom_G(V(\la), \ind_B^G
(S^0( \ul^{*}) \otimes \omega_1))
\\
 &\cong& 
\Hom_G(V(\la), \ind_B^G (S^{(16 - \ell(w))/2}( \ul^{*}) \otimes
\omega_1))
\\
 &\cong&\Hom_G(V(\la), \opH^{16}(G_1, H^0(\la))^{(-1)})
\\
&\cong&  \opH^{16}(G, H^0(\la)\otimes H^0(\la^*)^{(1)}).
\end{eqnarray*}
Since $\la$ is the smallest dominant weight in its linkage class
the assertion follows from the remarks in
Section~\ref{SS:nonvanishing}.

For $p = 19$, part (a)(ii) follows from Theorem~\ref{T:firstnontrivial}.
It remains to show part (a)(i).
If $\opH^i(G,H^0(\la)\otimes
H^0(\la^*)^{(1)}) \neq 0$ for
$i < 35 = 2p - 3$ and $\la \in X(T)_{+}$, then $\la = 19\omega_1
+ w\cdot 0$ or
$\la = 19\omega_6 + w\cdot 0$ for some $w \in W$. From the proof
of Lemma \ref{L:E6vanishing},
one can see that $i \geq 32$. Consider the case that $\la =
19\omega_1 + w\cdot 0$.
The $\omega_6$ case is dual and analogous. One can explicitly,
with the aid of MAGMA,
identify all $w \in W$ such that $\la \in X(T)_{+}$ and $\la -
\omega_1$ lies in the
positive root lattice. By considering the number of copies of
$\al_1$ appearing in
$\la - \omega_1$ (as in the proof of Lemma \ref{L:E6vanishing}), one can
identify the least $k$ such
that $\la - \omega_1$ is a weight in $S^k(\ul^*)$, and hence the
least possible value
of $i$. The three weights which can give a value of $i < 35$ are
listed in the following table
along with the minimum possible value of $k$.

\vskip.4cm
\begin{center}
\begin{tabular}{|c|c|c|c|}\hline
$\la$ & $\ell(w)$ & $k$ & $i = 2k + \ell(w)$ \\
\hline 
$7\omega_1 + \omega_4$ & 14 & 10 & 34\\
\hline
$7\omega_1 + \omega_2$ & 15 & 9 & 33\\
\hline
$7\omega_1$ & 16 & 8 & 32\\
\hline
\end{tabular}
\end{center}

For these weights, one can use MAGMA to explicitly compute
$$
\sum_{u \in W} (-1)^{\ell(u)} P_{k}( u\cdot\la - \omega_1)
$$
in order to apply Proposition~\ref{P:KostantPartCohom}. For $\la
= 7\omega_1$, one finds that
in fact
$$
\dim\opH^{32}(G,H^0(\la)\otimes H^0(\la^*)^{(1)}) = 
\sum_{u \in W} (-1)^{\ell(u)} P_{8}( u\cdot\la - \omega_1) = 0$$
and
$$
\dim\opH^{34}(G,H^0(\la)\otimes H^0(\la^*)^{(1)}) = 
\sum_{u \in W} (-1)^{\ell(u)} P_{9}( u\cdot\la - \omega_1) = 0.$$
So, for $\la = 7\omega_1$, we have $i \geq 36$.

For $\la = 7\omega_1 + \omega_2$ one finds
$$
\dim\opH^{33}(G,H^0(\la)\otimes H^0(\la^*)^{(1)}) = 
\sum_{u \in W} (-1)^{\ell(u)} P_{9}( u\cdot\la - \omega_1) = 0.$$
Therefore, $i \geq 35$ in this case. 

Finally, for $\la = 7\omega_1 + \omega_4$, one finds
$$
\dim\opH^{34}(G,H^0(\la)\otimes H^0(\la^*)^{(1)}) = 
\sum_{u \in W} (-1)^{\ell(u)} P_{10}( u\cdot\la - \omega_1) = 0.$$
Therefore, $i \geq 36$ in this case and the claim follows.
\end{proof}

We now consider the situation for arbitrary $r$.  Sharp 
vanishing can be obtained for primes about twice the Coxeter
number.

\begin{theorem} \label{T:E6q-results} Suppose $\Phi$ is of type
$E_6$ and $p\geq 13 $.
\begin{itemize}
\item[(a)] If $p \neq 13, 19$ or $p \neq 17$ when $r$ is even,
then
\begin{itemize}
\item[(i)] $\opH^{i}(\gfpr, k )  = 0$ for $0 <i < r(2p-3)$;  
\item[(ii)] $\opH^{r(2p-3)}(\gfpr, k ) \neq 0.$
\end{itemize}
\item[(b)] If $p = 13$, then
\begin{itemize}
\item[(i)] $\opH^{i}(\gfpr, k )  = 0$ for $0 <i < 16r$;  
\item[(ii)] $\opH^{16r}(\gfpr, k ) \neq 0.$
\end{itemize}
\item[(c)] If $p =17$ and $r$ is even, then
\begin{itemize}
\item[(i)] $\opH^{i}(\gfpr, k )  = 0$ for $0 <i < 27r$;  
\item[(ii)] $\opH^{31r}(\gfpr, k ) \neq 0.$
\end{itemize}
\item[(d)] If $p =19$, then
\begin{itemize}
\item[(i)] $\opH^{i}(\gfpr, k )  = 0$ for $0 <i < 32r$;  
\item[(ii)] $\opH^{35r}(\gfpr, k ) \neq 0.$
\end{itemize}
\end{itemize}
\end{theorem}

\begin{proof} The validity of parts (a)(ii), (c)(ii), and
(d)(ii) follows from Theorem~\ref{T:firstnontrivial}.
For part (a)(i), we need to show that 
$\opH^i(G,H^0(\la)\otimes H^0(\la^*)^{(r)}) = 0$ for all
dominant
weights $\lambda$ and all $0 < i < r(2p - 3)$.  
We argue along lines similar to that of the proof
of Theorem~\ref{T:Dsummary2}.  The root lattice 
$\mathbb{Z}\Phi$ has three cosets within $X(T)$:
$\mathbb{Z}\Phi$,
$\{\omega_{1}+\mathbb{Z}\Phi\}$, and
$\{\omega_{6}+\mathbb{Z}\Phi\}$.
If $\la$ is a weight in the root lattice, claim (a)(i) follows
from Lemma~\ref{L:Simplylacedcohom}(a).
 
Assume that 
$\opH^i(G,H^0(\la)\otimes H^0(\la^*)^{(r)}) \neq 0$ for some $i
> 0$ and
apply Proposition~\ref{P:degreebound2}. From above, we may
assume that
$\la = p \delta_0 +u_0 \cdot 0$ with $\delta_0=\delta_r \in
\{\omega_{1}+\mathbb{Z}\Phi\} \cup \{\omega_{6}+\mathbb{Z}\Phi\}
$.
For each $1 \leq j \leq r$, in order to have 
$\Ext_G^{l_j}(V(\ga_j)^{(1)},H^0(\ga_{j-1})) \neq 0$, then 
$\ga_j - \delta_{j - 1} = p\delta_j + u_j\cdot 0 - \delta_{j -
1}$ must
be a weight in $S^{\frac{l_j - \ell(u_{j-1})}{2}}(\ul^*)$. This
implies that
$p\delta_j - \delta_{j - 1}$ must lie in the positive root
lattice.
Since $\delta_r \in \{\omega_{1}+\mathbb{Z}\Phi\} \cup 
\{\omega_{6}+\mathbb{Z}\Phi\} $, we necessarily have
$p\delta_r \in \{\omega_{1}+\mathbb{Z}\Phi\} \cup
\{\omega_{6}+\mathbb{Z}\Phi\}$.
Since
$p\delta_r - \delta_{r-1} \in \mathbb{Z}\Phi$, it then follows
that
$\delta_{r-1} \in \{\omega_{1}+\mathbb{Z}\Phi\} \cup
\{\omega_{6}+\mathbb{Z}\Phi\}$.
Inductively one concludes  
that $\delta_j \in \{\omega_{1}+\mathbb{Z}\Phi\} \cup
\{\omega_{6}+\mathbb{Z}\Phi\}$
for all $j$. 

Before continuing, we investigate this condition on $\delta_j$ a
bit further.
Recall that $\omega_1 = \frac43\al_1 + \cdots + \frac23\al_6$
and
$\omega_6 = \frac23\al_1 + \cdots + \frac43\al_6$.
Suppose that $\delta_j \in \{\omega_{1}+\mathbb{Z}\Phi\}$ and
$\delta_{j-1} \in \{\omega_{1}+\mathbb{Z}\Phi\}$.  In order for $p\delta_j - \delta_{j-1}$ to lie in the root lattice, 
$\frac43p - \frac43 = \frac43(p-1)$ would need to be an integer. In other words, $p-1$ must be divisible by 3. The same argument
holds if we assume
that both $\delta_j$ and $\delta_{j-1}$ lie in
$\{\omega_{6}+\mathbb{Z}\Phi\}$.
On the other hand, suppose that $\delta_j \in
\{\omega_{1}+\mathbb{Z}\Phi\}$
and $\delta_{j-1} \in \{\omega_{6}+\mathbb{Z}\Phi\}$ (or vice
versa).
Then $\frac43p - \frac23 = \frac23(2p - 1)$ (or $\frac23p -
\frac43 = \frac23(p - 2)$,
respectively) must be an integer which implies that $p - 2$ is
divisible by 3.
Since $p$ is a prime greater than three, either $p-1$ is
divisible by 3 or $p - 2$ is
divisible by 3.  
Summarizing, if $3 | (p - 1)$, then either each 
$\delta_j \in \{\omega_{1}+\mathbb{Z}\Phi\}$ or each
$\delta_j \in \{\omega_{6}+\mathbb{Z}\Phi\}$. We refer to this
as the ``consistent'' case.
Whereas, if $3 | (p - 2)$, then we have an ``alternating''
situation with the $\delta_j$s
alternately lying in $\{\omega_{1}+\mathbb{Z}\Phi\}$ or
$\{\omega_{6}+\mathbb{Z}\Phi\}$.
Note further that since $\delta_0 = \delta_r$, the alternating
case can only occur if
$r$ is {\it even}.

Consider first the consistent case (when $3 | (p-1)$). Suppose without loss of generality that each
$\delta_j \in \{\omega_{1}+\mathbb{Z}\Phi\}$.  
For a weight $\ga$, when expressed as a sum of simple roots, let
$N(\ga)$ denote
the number of copies of $\al_1$ that appear.  
Since a positive root contains at most one 
copy of $\al_1$, we have
$$
\frac{l_j - \ell(u_{j-1})}{2} \geq pN(\delta_j) - N(-u_j\cdot 0)
- N(\delta_{j-1}).
$$
Rewriting this and using the fact that (see
Observation~\ref{O:uniquedecomp})
$\ell(u_{j-1}) \geq N(-u_{j-1}\cdot 0)$ gives
\begin{align*}
l_j &\geq 2pN(\delta_j) - 2N(-u_j\cdot 0) - 2N(\delta_{j-1}) +
\ell(u_{j-1})\\
&\geq 2pN(\delta_j) - 2N(-u_j\cdot 0) - 2N(\delta_{j-1}) +
N(-u_{j-1}\cdot 0).
\end{align*}
Therefore, 
\begin{align*}
i &= \sum_{j = 1}^rl_j
\geq \sum_{j = 1}^r (2pN(\delta_j) - 2N(-u_j\cdot 0) -
2N(\delta_{j-1}) + N(-u_{j-1}\cdot 0))\\
&= \sum_{j = 1}^r((2p - 2)N(\delta_j) - N(-u_j\cdot 0)).
\end{align*}
There are only 16 positive roots which contain an $\al_1$.
Hence,
$N(-u_j\cdot 0) \leq 16$. Since $N(\delta_j) \geq \frac 43$, we
get
$$
i \geq \sum_{j = 1}^r\left(\frac43(2p - 2) - 16\right) =
r\left(\frac43(2p - 2) - 16\right)\\
	= r\left(2p - 2 + \frac13(2p - 2) - 16\right).
$$
For $p \geq 25$, we get $i \geq r(2p-2)$ as desired. Note that
for $p = 17$ and
$p = 23$, $3 \nmid (p - 1)$, and so the only ``bad'' cases are
$p = 13$ and $p = 19$.
For $p = 13$, we conclude that $i \geq 16r$, and for $p = 19$,
we conclude that
$i \geq 32r$.

Now consider the alternating case (which requires $p - 2$ being
divisible by 3).
Analogous to the proof of Theorem~\ref{T:Dsummary2} for the type
$D_n$ case, for a weight $\ga$,
let $N_{\al_1}(\ga)$ (or $N_{\al_6}(\ga)$) denote the
coefficient of $\al_1$
(or $\al_6$) when $\ga$ is expressed as a sum of simple roots.
And then
set $N(\ga) =
\operatorname{max}\{N_{\al_1}(\ga),N_{\al_6}(\ga)\}$
(where the max is considered only in cases where the quantities
involved are
nonnegative). Then we reach the same conclusion on $i$ as above.
In this case, $p = 13$ and $p = 19$ cannot occur. Moreover, $p =
17$ and
$p = 23$ are potentially ``bad.'' However, for $p = 23$, since
$i$ is an
integer, we still conclude that $i \geq r(2p - 3)$ as needed.  
For $p = 17$, we conclude that $i \geq 27r$.

That completes the proof of all parts except for part (b)(ii)
with $p = 13$.
This follows inductively from the $r = 1$ case by using the the
spectral
sequence argument as in the proofs of
Theorem~\ref{T:firstnontrivial} and Theorem~\ref{T:Dsummary2}.
\end{proof}

For $p = 17$ when $r$ is even and $p = 19$, the theorem does not give a sharp vanishing bound.


\subsection{\bf Type $E_7$.} 
Assume for this subsection that $\Phi$ is of type $E_7$ with
$p > h = 18$ (so $p \geq 19$.)
The only dominant weight $\mu$ with 
$\langle \mu, \ta \rangle < 2$ is $\omega_7$. Again we conclude
from
Proposition~\ref{P:degreebound1} and
Proposition~\ref{P:vanishing} that
$\operatorname{H}^i(\gfp, k) =0$ for all $0 < i < 2p-3$ unless
there exists a weight
$\la$  of the form $p \omega_7 + w \cdot 0$  with 
$\opH^i(G, H^0(\la) \otimes H^0(\lambda^*)^{(1)}) \neq 0$ for
some $ 0 < i < 2p-3.$

\begin{lemma} \label{L:E7vanishing}
Suppose $\Phi$ is of type $E_7$, $p\geq 19$ and $\la \in X(T)_+$
is of the form $p \omega_7 + w \cdot 0$ with $w \in W$. Assume
in addition that $p \neq 19$, $23$. Then $\opH^i(G, H^0(\la)
\otimes H^0(\lambda^*)^{(1)}) = 0$ for all $ 0 < i < 2(p-1).$
\end{lemma}

\begin{proof} Assume $\la =p \omega_7 + w \cdot 0, w \in W.$ Let
$N$ denote the number of times that $\alpha_7$ appears in $-
w\cdot 0$ when written as a sum of simple roots. Note that all
positive roots of $\Phi$ contain the simple root $\alpha_7$ at
most once. This implies that $N \leq \ell(w)$. Moreover, there
are exactly $27$ distinct positive roots containing $\alpha_7$.
Hence, $N \leq 27$.

When writing $\omega_7$ as a sum of simple roots the coefficient
for $\alpha_7$ is $3/2$. Therefore $\la-\omega_7$, written as a
sum of simple roots contains at least $3/2(p-1) - N$ copies of
$\alpha_7$. From Section~\ref{SS:G1analysis}, we know that
$\opH^i(G, H^0(\la) \otimes H^0(\lambda^*)^{(1)}) \neq 0$ and
$i>0$ imply that $\la-\omega_1$ is a sum of $(i-\ell(w)/2$ many
positive roots. Using the the fact that $\alpha_7$ appears at
most once in each positive root, one obtains the inequality:
$$\frac{3}{2}(p-1)-N \leq \frac{i- \ell(w)}{2}.$$
Solving for $i$ yields
$$i \geq 3(p-1)- 2N+\ell(w)\geq 2(p-1) +p-1-N \geq 2(p-1)+
p-1-27.$$
Note that equality holds if and only if $N= \ell(w) = 27$. 

Hence, $i \geq 2(p-1)$ for all primes except for $18 <p \leq
28,$ i.e., the primes $p=19$ and $p=23.$
\end{proof}

\begin{theorem} \label{T:E7summary} Suppose $\Phi$ is of type
$E_7$ and $p\geq 19 $.
\begin{itemize}
\item[(a)] If $p \neq  19, 23$, then
\begin{itemize}
\item[(i)] $\opH^{i}(\gfp, k )  = 0$ for $0 <i < 2p-3$;  
\item[(ii)] $\opH^{2p-3}(\gfp, k ) \neq 0.$
\end{itemize}
\item[(b)] If $p = 19$, then
\begin{itemize}
\item[(i)] $\opH^{i}(\gfp, k )  = 0$ for $0 <i < 27$;  
\item[(ii)] $\opH^{27}(\gfp, k ) \neq 0.$
\end{itemize}
\item[(c)] If $p = 23$, then
\begin{itemize}
\item[(i)] $\opH^{i}(\gfp, k )  = 0$ for $0 <i < 39$;  
\item[(ii)] $\opH^{39}(\gfp, k ) \neq 0.$
\end{itemize}
\end{itemize}
\end{theorem}

\begin{proof}
Part (a) follows from Lemma \ref{L:E7vanishing},
Proposition~\ref{P:vanishing}, and
Theorem~\ref{T:firstnontrivial}.

For the proof of part (b), set $p=19$. Part (i) follows from the
proof of Lemma~\ref{L:E7vanishing}. Let $W_I$ denote the
subgroup of $W$ generated by the simple reflections
$s_{\alpha_1} , ... , s_{\alpha_6}$ and let $w$ denote the
distinguished representative of the left coset $w_0W_I$. Then
$\ell(w) =27 $ and $-w \cdot 0$ equals the sum of all positive
roots in $\Phi$ that contain $\alpha_7$, which equals the weight
$18\omega_7$. Let $\la = p\omega_7 + w \cdot 0= \omega_7.$
Using the same argument as for $E_6$, we obtain $\opH^{27}(G_1,
H^0(\la))^{(-1)} \cong
H^0(\la)$ and hence $\opH^{27}(G, H^0(\la)\otimes
H^0(\la^*)^{(1)}) \cong k.$
Again $\la$ is the smallest dominant weight in its linkage class
and the assertion follows from the remarks in
Section~\ref{SS:nonvanishing}.

For part (c), set $p = 23$. If $\opH^i(G,H^0(\la)\otimes
H^0(\la^*)^{(1)}) \neq 0$ for
$i < 43 = 2p - 3$ and $\la \in X(T)_{+}$, then $\la = 23\omega_7
+ w\cdot 0$ for some $w \in W$.
From the proof of Lemma \ref{L:E7vanishing}, one can see that $i \geq 39$.  
One can explicitly, with the aid of MAGMA,
identify all $w \in W$ such that $\la \in X(T)_{+}$ and $\la -
\omega_7$ lies in the
positive root lattice. By considering the number of copies of
$\al_7$ appearing in
$\la - \omega_7$ (as in the proof of Lemma \ref{L:E7vanishing}), one can
identify the least $k$ such
that $\la - \omega_7$ is a weight in $S^k(\ul^*)$, and hence the
least possible value
of $i$. The four weights which can give a value of $i < 43$ are
listed in the following table
along with the minimum possible value of $k$.

\vskip.4cm
\begin{center}
\begin{tabular}{|c|c|c|c|}\hline
$\la$ & $\ell(w)$ & $k$ & $i = 2k + \ell(w)$ \\
\hline 
$5\omega_7 + \omega_4$ & 24 & 9 & 42\\
\hline
$5\omega_7 + \omega_3$ & 25 & 8 & 41\\
\hline
$5\omega_7 + \omega_1$ & 26 & 7 & 40\\
\hline
$5\omega_7$ & 27 & 6 & 39\\
\hline
\end{tabular}
\end{center}

For these weights, one can use MAGMA to explicitly compute
$$
\sum_{u \in W} (-1)^{\ell(u)} P_{k}( u\cdot\la - \omega_7)
$$
in order to apply Proposition~\ref{P:KostantPartCohom}. For $\la
= 5\omega_7$, one finds that
in fact
$$
\dim\opH^{39}(G,H^0(\la)\otimes H^0(\la^*)^{(1)}) = 
\sum_{u \in W} (-1)^{\ell(u)} P_{6}( u\cdot\la - \omega_7) = 1.$$
Since there are no weights less than $\la$ which can give
cohomology in degree 40,
$\opH^{39}(\gfp,k) \neq 0$.
\end{proof}

We next consider the situation for arbitrary $r$.

\begin{theorem}\label{T:E7summary-r} Suppose $\Phi$ is of type $E_7$ and $p\geq 19$.\begin{itemize}
\item[(a)] If $p \neq 19, 23$, then
\begin{itemize}
\item[(i)] $\opH^{i}(\gfpr, k )  = 0$ for $0 <i < r(2p-3)$;  
\item[(ii)] $\opH^{r(2p-3)}(\gfpr, k ) \neq 0.$
\end{itemize}
\item[(b)] If $p = 19$, then
\begin{itemize}
\item[(i)] $\opH^{i}(\gfpr, k )  = 0$ for $0 <i < 27r$;  
\item[(ii)] $\opH^{27r}(\gfpr, k ) \neq 0.$
\end{itemize}
\item[(c)] If $p = 23$, then
\begin{itemize}
\item[(i)] $\opH^{i}(\gfpr, k )  = 0$ for $0 <i < 39r$;  
\item[(ii)] $\opH^{39r}(\gfpr, k ) \neq 0.$
\end{itemize}
\end{itemize}
\end{theorem}

\begin{proof} The validity of part (a)(ii) follows from
Theorem~\ref{T:firstnontrivial}.
For part (a)(i), we need to show that 
$\opH^i(G,H^0(\la)\otimes H^0(\la^*)^{(r)}) = 0$ for all
dominant
weights $\lambda$ and all $0 < i < r(2p - 3)$.  
An argument similar to that in the proof of
Theorem~\ref{T:E6q-results} works here as well.
The root lattice $\mathbb{Z}\Phi$ has two cosets within $X(T)$:
$\mathbb{Z}\Phi$
and $\{\omega_{7}+\mathbb{Z}\Phi\}$.  
If $\la$ is a weight in the root lattice claim (a)(i) follows
from Lemma~\ref{L:Simplylacedcohom}(a).

Consider then the case that $\la \in \{\omega_{7}+\mathbb{Z}\Phi\}$ and apply 
Proposition \ref{P:degreebound2}. As before, one finds that each 
$\delta_j \in \{\omega_{7}+\mathbb{Z}\Phi\}$. Further, if we let
$N(\ga)$ denote the coefficient of $\al_7$, when $\ga$ is
expressed as
a sum of simple roots, then we again conclude that 
$$
i \geq \sum_{j = 1}^r((2p - 2)N(\delta_j) - N(-u_j\cdot 0)).
$$
Here, since $\omega_7 = \al_1 + \cdots + \frac32\al_7$, we have
$N(\delta_j) \geq \frac32$.
Furthermore, there are 27 positive roots which contain an
$\al_7$, and so
$N(-u_j\cdot 0) \leq 27$.  Therefore, we get
$$
i \geq r\left(\frac32(2p - 2) - 27\right) = r\left(2p - 3 + p -
27\right).
$$
For $p \geq 27$, we have $i \geq r(2p - 3)$ as needed, which
completes part (a).

For $p = 19$ and $p = 23$, we conclude only that $i \geq 27r$ or
$i \geq 39r$,
respectively, which gives parts (b)(i) and (c)(i). Parts (b)(ii)
and (c)(ii) again
follows inductively from the $r = 1$ case by the spectral
sequence argument in
Theorem~\ref{T:firstnontrivial} and Theorem~\ref{T:Dsummary2}.
\end{proof}


\subsection{\bf Type $E_8$.} 
Assume for this subsection that $\Phi$ is of type $E_8$ with
$p > h = 30$ (so $p \geq 31$).
Here the weight lattice and root lattice always coincide. 
From Corollary~\ref{C:simplylacedadjoint} we obtain the
following.

\begin{theorem} \label{T:E8summary} Suppose $\Phi$ is of type
$E_8$ and $p\geq 31$.  Then
\begin{itemize}
\item[(a)] 
 $\opH^{i}(\gfpr, k )  = 0$ for $0 <i < r(2p-3)$;  
\item[(b)] $\opH^{r(2p-3)}(\gfpr, k ) \neq 0.$
\end{itemize}
\end{theorem}


\section{Type $B_n$, $n \geq 3$}

Assume throughout this section that $\Phi$ is of type $B_n$, $n
\geq 3$, and that
$p > h = 2n$. Note that type $B_2$ is equivalent to type $C_2$
which
was discussed in \cite{BNP}. However, for certain inductive
arguments, at points we will
allow $n = 1, 2$. Following Section 2, our goal is to find the
least $i > 0$
such that $\opH^i(G,H^0(\la)\otimes H^0(\la^*)^{(1)}) \neq 0$
for some $\la$.
From Proposition \ref{P:degreebound1}, we know that $i \geq p -
2$.

\subsection{\bf Restrictions.}\label{SS:B1} Suppose that 
$\opH^i(G,H^0(\la)\otimes H^0(\la^*)^{(1)}) \neq 0$ for some
$i > 0$ and $\la = p\mu + w\cdot 0$ with $\mu \in X(T)_{+}$ and
$w \in W$.
In this case, the longest root $\ta = \omega_2$. From
Proposition \ref{P:degreebound1},
$i \geq (p-1)\langle\mu,\ta^{\vee}\rangle - 1$.

For a fundamental dominant weight $\omega_j$,
$$
\langle\omega_j,\ta^{\vee}\rangle = 
\begin{cases}
1 &\text{ if } j = 1, n\\
2 &\text{ if } 2 \leq j \leq n - 1.
\end{cases}
$$
Therefore, if $\mu \neq \omega_1, \omega_n$, we will have 
$\langle\mu,\ta^{\vee}\rangle \geq 2$ and $i \geq 2p - 3$.

The following lemma shows that if $n$ is sufficiently large,
and $\la = p\omega_n + w\cdot 0$, then one also has $i \geq 2p -
3$.
In fact strictly greater.

\begin{lemma}\label{L:bwn} Suppose $\Phi$ is of type $B_n$ with
$n \geq 7$ and
$p > 2n$.
Suppose $\la = p\omega_n + w\cdot 0 \in X(T)_+$ with $w \in W$
and $\opH^i(G,H^0(\la)\otimes H^0(\la^*)^{(1)}) \neq 0$. 
Then $i > 2p - 3$.
\end{lemma}

\begin{proof} Following the discussion in Section
\ref{SS:G1analysis}, $\la - \omega_n
= (p - 1)\omega_n + w\cdot 0$ must be a weight of $S^j(\ul^*)$
for
$j = \frac{i - \ell(w)}{2}$.  Recall that 
$2\omega_n = \al_1 + 2\al_2 + 3\al_3 + \cdots + n\al_n$.  
Consider the decomposition of $-w\cdot 0$ into a sum of distinct
positive roots
(cf. Observation \ref{O:uniquedecomp}). Write $\ell(w) = a + b +
c$ where $a$ is
the number of positive roots in this decomposition which contain
$2\al_n$,
$b$ is the number of roots in this decomposition which contain
$\al_n$, and
$c$ is the number of roots in this decomposition that do not
contain $\al_n$.
Then $\la - \omega_n$ contains 
$$
\left(\frac{p - 1}{2}\right)n - 2a - b
$$
copies of $\al_n$. Since any root contains at most 2 copies of
$\al_n$, we
have
$$
\frac{i - \ell(w)}{2} = j \geq \frac12\left(\left(\frac{p -
1}{2}\right)n - 2a - b\right).
$$
Replacing $\ell(w)$ by $a + b + c$ and simplifying gives
$$
i \geq \left(\frac{p - 1}{2}\right)n - a + c.
$$
The total number of positive roots which contain $2\al_n$ is
$n(n-1)/2$.
Hence, $a \leq n(n-1)/2$ and $c \geq 0$.  So we have
\begin{equation}\label{E:omegan}
i \geq \left(\frac{p - 1}{2}\right)n - \left(\frac{n -
1}{2}\right)n =
\left(\frac{p - n}{2}\right)n = 2p + \left(\frac{n}{2} -
2\right)p - \frac{n^2}{2}.
\end{equation}
Finally, using the assumption that $p \geq 2n + 1$, one 
finds
$$
i \geq 2p + \left(\frac{n}{2} - 2\right)(2n + 1) - \frac{n^2}{2}
=
        2p + \frac{n^2}{2} - \frac72n - 2.
$$
For $n \geq 7$, we have $i \geq 2p - 2$ as claimed.
\end{proof}

For $3 \leq n \leq 6$, the lemma is in fact false. These cases
will be
discussed in Sections \ref{SS:B3} - \ref{SS:B6}.


\subsection{\bf The case of $\omega_1$.}\label{SS:bw1} In this
section we investigate
the case that $\la = p\omega_1 + w\cdot 0$.  
Throughout this section $\Phi$ is of type $B_n$. In order to
make use of some inductive arguments we allow $n \geq 1$.
We will frequently switch between the bases consisting of the
simple roots $\{ \alpha_1, \dots, \alpha_n\}$, the fundamental
weights $\{ \omega_1, \dots , \omega_n\}$, and the canonical
basis $\{\epsilon_1, \dots , \epsilon_n \}$ of ${\mathbb R}^n$.
Following \cite{Hum1} we have $\alpha_i = \epsilon_i -
\epsilon_{i+1}$, for $1 \leq i \leq n-1$, and $\alpha_n =
\epsilon_n$. Note that $ \epsilon_1=\alpha_0$ is the maximal
short root. The fundamental weights are $\omega_j = \epsilon_1 +
\cdots + \epsilon_j$, for $1\leq j \leq n-1$, and $\omega_n =
1/2(\epsilon_1 + \cdots + \epsilon_n)$. In particular, $\omega_1
= \epsilon_1$.

\begin{definition}\label{D:PandT} For $\Phi$ of type $B_n$, with
$n \geq 1$, we define
\begin{itemize}
\item[]
 $$P(m,k,j,n) := 
\begin{cases} \sum_{u \in W} (-1)^{\ell (u) } P_k(u\cdot(m
\epsilon_1 + (\epsilon_1 + \dots + \epsilon_j)))
 & \mbox{ if } m \geq 0, k\geq 0, \\
 & \qquad \mbox{ and } 1 \leq j \leq n,\\
 1& \mbox{ if } m =-1, k= 0, \\
  & \qquad \mbox{ and } j=1,\\
 0 &\mbox{ else;}
 \end{cases}   $$
\item[]
 $$T(m,k,j,n) := 
\begin{cases} \dim \Hom_G (V(m\epsilon_1 +(\epsilon_1 + \cdots +
\epsilon_j)),H^0(G/B, S^{k}(\ul^*))\otimes H^0(\epsilon_1))\\
 \hspace{2in}\mbox{ if } m \geq 0, k\geq 0, 
  \mbox{ and } 1 \leq j \leq n,\\
 0 \hspace{1.9in}\mbox{ else.}
 \end{cases}   $$
\end{itemize}
\end{definition}
Note that for  $p > 2n$,
$$P(m,k,j,n)=\dim \Hom_G (V(m\epsilon_1 +(\epsilon_1 + ... +
\epsilon_j)),H^0(G/B, S^{k}(\ul^*))),$$ which equals the
multiplicity of $H^0(m\epsilon_1 +(\epsilon_1 + \cdots +
\epsilon_j))$ in a good filtration of $H^0(G/B, S^{k}(\ul^*))$
(cf. \cite[3.8]{AJ}).
In particular, $P(m,k,j,n) \geq 0$ for all $m, k, j$, and $n$.

\begin{lemma}\label{L:Prelations} Suppose $\Phi$ is of type
$B_n$ with $n \geq 1$ and $p > 2n$. If $m \geq 0$, $k\geq 0$,
and $1 \leq j \leq n$, then
\begin{itemize}
\item[(a)] $\sum_{u \in W} (-1)^{\ell (u) } P_k(u\cdot(m
\epsilon_1 + (\epsilon_1 + ... + \epsilon_j))- \epsilon_1) =
P(m-1,k,j,n) + P(m,k,j-1,n-1)$;
\item[(b)] $\sum_{u \in W} (-1)^{\ell (u) } P_k(u\cdot m
\epsilon_1 + \epsilon_1) = P(m,k,1,n) - P(m,k,1,n-1)$;
\item[(c)] $P(m,k,j,n) =0$ whenever $k<m+1$;
\item[(d)] $T(m,k,j,n) = \sum_{i=1}^{2n}P(m-1,k-i +1,j,n)
+\sum_{i=1}^{2n}P(m,k-i +1,j-1,n-1) + P(m,k-n,j,n)$;
   \item[(e)] for $n \geq 2$ and $1 \leq j \leq n-1$,\\
  $T(m,k,j,n) \geq P(m-1,k,j,n)  + P(m,k,j+1,n)+P(m,k,j-1,n)
$;
  \item[(f)] 
  $T(m,k,n,n) \geq P(m-1,k,n,n)  + P(m,k,n,n)+P(m,k,n-1,n)
$;
  \item[(g)] for $l \geq 0$, $P(2l,k,1,n)
= P(2l-1,k-n,1,n)$;
\item[(h)] for $l \geq 0$, $\sum_{u \in W} (-1)^{\ell (u) }
P_k(u\cdot (2l+1) \epsilon_1 + \epsilon_1) = P(2l,k-n,1,n).$

 \end{itemize}
 
\end{lemma}

\begin{proof}  
(a) $P(m-1,k,j,n) = \sum_{u \in W} (-1)^{\ell (u) }
P_k(u\cdot((m-1) \epsilon_1+ (\epsilon_1 + \cdots +
\epsilon_j)))$. The expression $u\cdot((m-1) \epsilon_1+
(\epsilon_1 + \cdots + \epsilon_j) )= u((m-1) \epsilon_1+
(\epsilon_1 + \cdots + \epsilon_j)) + u \cdot 0$ will be a sum
of positive roots only if either $u(\epsilon_1) = \epsilon_1$ or
$u(\epsilon_2) = \epsilon_1$. If $u$ is of the second type, then
$u s_{\alpha_1}$ stabilizes $\epsilon_1$. Setting $v = u
s_{\alpha_1}$, one obtains
 \begin{eqnarray*}
P(m-1,k,j,n) 
&=&
\sum_{\{u \in W|u(\epsilon_1)= \epsilon_1\}} (-1)^{\ell (u) }
P_k(u\cdot((m-1) \epsilon_1 + (\epsilon_1 + \cdots +
\epsilon_j))) \\
&& \quad + 
\sum_{\{u \in W|u(\epsilon_2)= \epsilon_1\}} (-1)^{\ell (u) }
P_k(u\cdot((m-1) \epsilon_1 + (\epsilon_1 + \cdots +
\epsilon_j))) \\
&=&
\sum_{u \in W} (-1)^{\ell (u) } P_k(u\cdot(m \epsilon_1 +
(\epsilon_1 + \cdots + \epsilon_j))- \epsilon_1) \\
&& \quad - \sum_{v \in W} (-1)^{\ell (v) } P_k(v\cdot(m
\epsilon_2+ (\epsilon_2 + \cdots + \epsilon_j))).
\end{eqnarray*} 
The second term is just $P(m,k,j-1,n-1)$ as claimed. Note that
the formula also holds for $m=0$.

(b) The expression $\sum_{u \in W} (-1)^{\ell (u) }
P_k(u\cdot m \epsilon_1+ \epsilon_1)$ will be a sum of positive
roots only if either $u(\epsilon_1) = \epsilon_1$ or
$u(\epsilon_2) = \epsilon_1$. Arguing as above one obtains
 \begin{eqnarray*}
\sum_{u \in W} (-1)^{\ell(u)} P_k(u\cdot m \epsilon_1+
\epsilon_1)&=&
\sum_{\{u \in W|u(\epsilon_1)= \epsilon_1\}} (-1)^{\ell (u) }
P_k(u\cdot m \epsilon_1+ \epsilon_1) \\
&& \quad + 
\sum_{\{u \in W|u(\epsilon_2)= \epsilon_1\}} (-1)^{\ell (u) }
P_k(u\cdot m \epsilon_1+ \epsilon_1)\\
&=&
\sum_{u \in W} (-1)^{\ell (u) } P_k(u\cdot(m \epsilon_1 +
\epsilon_1)) \\
&& \quad - \sum_{v \in W} (-1)^{\ell (v) } P_k(v\cdot(m
\epsilon_2+ \epsilon_2)).
\end{eqnarray*} 
The first term is $P(m,k,1,n)$ and the second term is just
$P(m,k,1,n-1)$.

(c) Assume that $0 \leq k < m+1$. Part (a) implies that
$$P(m,k,j,n) \leq \sum_{u \in W} (-1)^{\ell (u) }
P_k(u\cdot((m+1) \epsilon_1+ (\epsilon_1 + \cdots + \epsilon_j))
-\epsilon_1).$$
Note that $u\cdot((m+1) \epsilon_1+ (\epsilon_1 + \cdots +
\epsilon_j)) -\epsilon_1$ is a sum of positive roots if and only
if $u( \epsilon_1) = \epsilon_1.$ If $u( \epsilon_1) =
\epsilon_1$ then $u\cdot((m+1) \epsilon_1+ (\epsilon_1 + \cdots
+ \epsilon_j)) -\epsilon_1= (m+1)\epsilon_1 + u\cdot (\epsilon_2
+ \cdots + \epsilon_j)$ and $-u\cdot (\epsilon_2 + \cdots +
\epsilon_j)$, written as a sum of simple roots, contains no
$\alpha_1$. However, $(m+1) \epsilon_1$, written as a sum of
simple roots, contains exactly $m+1$ copies of $\alpha_1$. Each
positive root of $\Phi$ contains at most one copy of $\alpha_1$.
Therefore at least $m+1$ positive roots are needed to obtain the
weight $u\cdot((m+1) \epsilon_1+ (\epsilon_1 + \cdots +
\epsilon_j)) -\epsilon_1$. One concludes that $P_k(u\cdot((m+1)
\epsilon_1+ (\epsilon_1 + \cdots + \epsilon_j)) -\epsilon_1)=0$
and the assertion follows.

(d) 
For a simple root $\alpha$, let $P_{\alpha}$ denote the minimal
parabolic subgroup corresponding to $\alpha$, and let
$\ul_{\alpha}$ denote the Lie algebra of the unipotent radical
of $P_{\alpha}$. From the short exact sequence
\begin{equation*} 
0 \to \alpha \to \ul^* \to \ul_{\alpha}^* \to 0
\end{equation*} 
one obtains the Koszul resolution
\begin{equation*} 
0 \to S^{k-1}(\ul^*) \otimes \alpha \to S^k (\ul^*) \to
S^k(\ul_{\alpha}^*) \to 0.
\end{equation*} 
Tensoring with a weight $\mu$ yields
\begin{equation*} 
0 \to S^{k-1}(\ul^*) \otimes \alpha \otimes \mu \to S^k
(\ul^*)\otimes \mu \to S^k(\ul_{\alpha}^*)\otimes \mu \to 0.
\end{equation*} 
Induction from $B$ to $G$ yields the long exact sequence
\begin{equation}\label{b1}
\cdots \to H^i(G/B, S^{k-1}(\ul^*) \otimes \alpha \otimes \mu)
\to H^i(G/B, S^k (\ul^*)\otimes \mu) \to H^i(G/B,
S^k(\ul_{\alpha}^*)\otimes \mu) \to \cdots.
\end{equation}
We apply (\ref{b1}) with $\alpha = \alpha_j = \epsilon_j -
\epsilon_{j+1}$ and $\mu = -\epsilon_{j}$, where $1 \leq j \leq
n-1$, giving
\begin{equation}\label{b2}
\cdots \to H^i(G/B, S^{k-1}(\ul^*) \otimes - \epsilon_{j+1}) \to
H^i(G/B, S^k (\ul^*)\otimes -\epsilon_j) \to H^i(G/B,
S^k(\ul_{\alpha}^*)\otimes -\epsilon_j) \to \cdots.
\end{equation}
Note that, $\langle -\epsilon_j, \alpha_j^{\vee} \rangle = -1$,
forces $H^i(P_{\al}/B,-\epsilon_j)= 0$ for all $i$. The spectral
sequence
$$H^{r}(G/P_{\al}, S^k(\ul_{\alpha}^*))\otimes
H^s(P_{\al}/B,\mu) \Rightarrow H^{r+s}(G/B,
S^k(\ul_{\alpha}^*)\otimes \mu)$$ yields
$H^i(G/B, S^k(\ul_{\alpha}^*)\otimes -\epsilon_j) = 0$ for all
$i$.
Therefore, from (\ref{b2}), one obtains for $1\leq j \leq n-1$
and $i\geq 0$
\begin{equation}\label{b3}
H^i(G/B, S^{k-1}(\ul^*) \otimes - \epsilon_{j+1}) \cong H^i(G/B,
S^k (\ul^*)\otimes -\epsilon_j).
\end{equation} 
Iterating this process yields
\begin{equation}\label{b4}
H^i(G/B, S^{k}(\ul^*) \otimes - \epsilon_{i}) \cong H^i(G/B,
S^{k-n+i} (\ul^*)\otimes -\epsilon_n).
\end{equation} 
Note that if $k < n - i$, the right hand side is identically
zero, and the isomorphism
still holds.

Next we apply (\ref{b1}) with $\alpha = \alpha_n= \epsilon_n $
and $\mu = -\epsilon_{n}$ in order to obtain
\begin{equation}\label{b5}
\cdots \to H^i(G/B, S^{k-1}(\ul^*)) \to H^i(G/B, S^k
(\ul^*)\otimes -\epsilon_n) \to H^i(G/B,
S^k(\ul_{\alpha}^*)\otimes -\epsilon_n) \to \cdots.
\end{equation}
Here $\langle -\epsilon_n, \alpha_n^{\vee} \rangle = -2$. Using
the spectral sequence
as above, one obtains  
\begin{equation}\label{b6}\begin{split}
H^i(G/B, S^k(\ul_{\alpha}^*)\otimes -\epsilon_n)  
&\cong H^{i-1}(G/P_{\al}, S^k(\ul_{\alpha}^*))\otimes
H^1(P_{\al}/B,-\epsilon_n) \\
&\cong H^{i-1}(G/P_{\al}, S^k(\ul_{\alpha}^*))\cong H^{i-1}(G/B,
S^k(\ul_{\alpha}^*)).
\end{split}\end{equation}
Since $H^0(G/B, S^k(\ul_{\al}^*)\otimes - \epsilon_n) = 0$, one
obtains via by (\ref{b5}),
\begin{equation}\label{b7}
H^0(G/B, S^{k-1}(\ul^*)) \cong H^0(G/B, S^k (\ul^*)\otimes
-\epsilon_n).
\end{equation}

From $H^i(G/B, S^{k}(\ul^*))=0$ for $i\geq 1$, using (\ref{b5}),
one concludes
\begin{equation}\label{b8}
H^i(G/B, S^k (\ul^*)\otimes -\epsilon_n) \cong H^{i-1}(G/B,
S^k(\ul_{\alpha}^*)) \mbox{ for } i\geq 1.
\end{equation}
Next, we apply (\ref{b1}) with $\alpha = \alpha_n= \epsilon_n $
and $\mu = 0$ in order to obtain
\begin{equation}\label{b9}
\cdots \to H^i(G/B, S^{k-1}(\ul^*)\otimes \epsilon_n) \to
H^i(G/B, S^k (\ul^*)) \to H^i(G/B, S^k(\ul_{\alpha}^*)) \to
\cdots.
\end{equation}
From $H^i(G/B, S^{k}(\ul^*))=0$ and $H^i(G/B,
S^{k}(\ul^*)\otimes \epsilon_n)=0$ for $i\geq 1$
\cite[2.8]{KLT}, one concludes
\begin{equation}\label{b10}
0 \to H^0(G/B, S^{k-1}(\ul^*)\otimes \epsilon_n) \to H^0(G/B,
S^k (\ul^*)) \to H^0(G/B, S^k(\ul_{\alpha}^*)) \to 0,
\end{equation}
and
\begin{equation}\label{b11}
H^i(G/B, S^k (\ul^*)\otimes -\epsilon_n) \cong H^{i-1}(G/B,
S^k(\ul_{\alpha}^*))=0\mbox{ for }i \geq 2.
\end{equation} 

Similarly, apply (\ref{b1}) with $\alpha = \alpha_j =
\epsilon_{j-1} - \epsilon_{j}$ and $\mu = \epsilon_{j}$, where
$2 \leq j \leq n$, to obtain
\begin{equation*} 
\cdots \to H^i(G/B, S^{k-1}(\ul^*) \otimes \epsilon_{j-1}) \to
H^i(G/B, S^k (\ul^*)\otimes \epsilon_j) \to H^i(G/B,
S^k(\ul_{\alpha}^*)\otimes \epsilon_j) \to \cdots.
\end{equation*}
As before this yields
\begin{equation}\label{b12}
H^i(G/B, S^{k-1}(\ul^*) \otimes \epsilon_{j-1}) \cong H^i(G/B,
S^k (\ul^*)\otimes \epsilon_j).
\end{equation} 
Iterating this process results in
\begin{equation}\label{b13}
H^i(G/B, S^{k-i+1}(\ul^*) \otimes \epsilon_{1}) \cong H^i(G/B,
S^k (\ul^*)\otimes \epsilon_i).
\end{equation}
Again, this isomorphism holds even if $k < i - 1$ (when the left
side is identically zero).

For any finite $B$-module $M$, we denote its Euler
characteristic by
$$\chi(M)=\sum_{i \geq 0} (-1)^i \ch H^i(G/B,M).$$ 
From the above, one obtains
\begin{eqnarray*}
&&\hspace{-.5in}\ch (H^0(G/B, S^{k}(\ul^*))\otimes
H^0(\epsilon_1))= \chi (S^{k}(\ul^*)\otimes H^0(\epsilon_1)) \\
&=& \sum_{i=1}^{n} \chi (S^{k}(\ul^*)\otimes \epsilon_i) +\chi
(S^{k}(\ul^*))
+ \sum_{i=1}^{n} \chi (S^{k}(\ul^*)\otimes -\epsilon_i)  \\
&=& \sum_{i=1}^{n} \chi (S^{k-i+1}(\ul^*)\otimes \epsilon_1)
+\chi (S^{k}(\ul^*))
+ \sum_{i=1}^{n} \chi (S^{k-i+1}(\ul^*)\otimes -\epsilon_n)
\mbox{ (by \ref{b4}, \ref{b13})}\\
&=& \sum_{i=1}^{n} \ch H^0(G/B,S^{k-i+1}(\ul^*)\otimes
\epsilon_1)
+\ch H^0(G/B, S^{k}(\ul^*))\mbox{ (by \cite[2.8]{KLT})}\\
&& \quad + \sum_{i=1}^{n} \ch H^0(G/B,S^{k-i+1}(\ul^*)\otimes
-\epsilon_n)
- \sum_{i=1}^{n} \ch H^1(G/B,S^{k-i+1}(\ul^*)\otimes
-\epsilon_n) \\
& & \hskip4.0in \mbox{ (by \ref{b8}, \ref{b11})}\\
&=& \sum_{i=1}^{n} \ch H^0(G/B,S^{k-i+1}(\ul^*)\otimes
\epsilon_1)
+\ch H^0(G/B, S^{k}(\ul^*))\\
&& \quad + \sum_{i=1}^{n} \ch H^0(G/B,S^{k-i}(\ul^*)) 
+\sum_{i=1}^{n} \ch H^0(G/B,S^{k-i-n+1}(\ul^*)\otimes
\epsilon_1)\\
&& \quad - \sum_{i=1}^{n} \ch H^0(G/B,S^{k-i+1}(\ul^*)) \mbox{
(by \ref{b7}, \ref{b8}, \ref{b10}, \ref{b13})}\\
&=& \sum_{i=1}^{2n} \ch H^0(G/B,S^{k-i+1}(\ul^*)\otimes
\epsilon_1)
+\ch H^0(G/B, S^{k-n}(\ul^*)).
\end{eqnarray*}
The last equality yields
\begin{eqnarray*}
T(m,k,j,n) &=& 
\dim \Hom_G (V(m\epsilon_1 +(\epsilon_1 + \cdots +
\epsilon_j)),H^0(G/B, S^{k}(\ul^*))\otimes H^0(\epsilon_1))\\
&=&\sum_{i=1}^{2n} \dim \Hom_G (V(m\epsilon_1 +(\epsilon_1 +
\cdots + \epsilon_j)),H^0(G/B, S^{k-i+1}(\ul^*)\otimes
\epsilon_1))\\
&&\quad + \dim \Hom_G (V(m\epsilon_1 +(\epsilon_1 + \cdots +
\epsilon_j)),H^0(G/B, S^{k-n}(\ul^*)).
\end{eqnarray*}
The assertion follows now from part (a).

(e) A direct computation shows that 
\begin{eqnarray*}
&&\hspace{-.4in}\ch (V(m\epsilon_1 +(\epsilon_1 + \cdots +
\epsilon_j))\otimes V(\epsilon_1))
= 
\ch V((m-1)\epsilon_1 +(\epsilon_1 + \cdots + \epsilon_j)) \\
&+&
\ch V(m\epsilon_1 +(\epsilon_1 + \cdots + \epsilon_{j+1}))
+
\ch V(m\epsilon_1 +(\epsilon_1 + \cdots + \epsilon_{j-1}))\\
 &+&
\ch V((m+1)\epsilon_1 +(\epsilon_1 + \cdots + \epsilon_j))
+
\ch V((m-1)\epsilon_1 +(\epsilon_1 + \epsilon_2)+(\epsilon_1 +
\cdots + \epsilon_j)).
\end{eqnarray*}
It follows that
\begin{eqnarray*}
T(m,k,j,n) &=& 
\dim \Hom_G (V(m\epsilon_1 +(\epsilon_1 + \cdots +
\epsilon_j)),H^0(G/B, S^{k}(\ul^*))\otimes H^0(\epsilon_1))\\
&=& \dim \Hom_G (V(m\epsilon_1 +(\epsilon_1 + \cdots +
\epsilon_j))\otimes V(\epsilon_1)),H^0(G/B, S^{k}(\ul^*)))\\
&\geq& P(m-1,k,j,n)  + P(m,k,j+1,n)+P(m,k,j-1,n),
\end{eqnarray*}
as claimed. 

Part (f) follows in similar fashion.

(g) It is well-known that, for $m \geq 2$,
$\ch(H^0(m\omega_1))$ is equal to the difference of the $m$th
and the $(m-2)$nd symmetric power of the natural representation.
The natural representation has one-dimensional weight spaces and
includes the zero weight space. One concludes that the dimension
of the zero weight space of the $2l$th symmetric power equals
the dimension of the zero weight space of the $(2l+1)$st
symmetric power. The same is true for the pair $H^0(2l\omega_1)$
and $H^0((2l+1)\omega_1)$. It follows from Kostant's Theorem
\cite[24.2]{Hum1} that
\begin{equation}\label{b14}
\sum_{k\geq 0}P(2l-1,k,1,n)= \sum_{k\geq 0}P(2l,k,1,n).
\end{equation}
From (\ref{b10}) and (\ref{b13}) one obtains
\begin{equation*} 
0 \to H^0(G/B, S^{k-n}(\ul^*)\otimes \epsilon_1) \to H^0(G/B,
S^k (\ul^*)) \to H^0(G/B, S^k(\ul_{\alpha_n}^*)) \to 0.
\end{equation*}
All three modules have good filtrations. 
Moreover, by part (i)\\ 
$\dim \Hom_G (V(2l\epsilon_1 +\epsilon_1),H^0(G/B,
S^{k-n}(\ul^*)\otimes \epsilon_1)) = P(2l-1,k-n,1,n)$. Hence for
$l \geq 0$
$$
P(2l,k,1,n)
= P(2l-1,k-n,1,n)\\
+ \dim \Hom_G (V(2l\epsilon_1 +\epsilon_1),H^0(G/B,
S^{k}(\ul_{\alpha_n}^*))).
$$
Summing over all $k\geq 0$ yields 
\begin{equation*}
\sum_{k\geq 0}P(2l-1,k,1,n)= \sum_{k\geq 0}P(2l,k,1,n)+
\sum_{k\geq 0}\dim \Hom_G (V(2l\epsilon_1 +\epsilon_1)),H^0(G/B,
S^{k}(\ul_{\alpha_n}^*))).
\end{equation*}
Comparing with (\ref{b14}) yields $\dim \Hom_G (V(2l\epsilon_1
+\epsilon_1),H^0(G/B, S^{k}(\ul_{\alpha_n}^*)))=0$, which forces
$P(2l,k,1,n)
= P(2l-1,k-n,1,n)$, for all $k \geq 0$.

(h) Following \cite[3.8]{AJ}, the multiplicity of $\ch
V(2l\epsilon_1 +\epsilon_1)$ in
$\chi (S^{k}(\ul^*)\otimes - \epsilon_1)$ equals 
$$\sum_{u \in W}(-1)^{\ell (u) } P_k(u\cdot((2l+1) \epsilon_1) + \epsilon_1).$$

Moreover, by (\ref{b11}) and (\ref{b4}),
$$\chi (S^{k}(\ul^*)\otimes - \epsilon_1) = \ch H^0(G/B,
S^{k}(\ul^*)\otimes - \epsilon_1) - \ch H^1(G/B,
S^{k}(\ul^*)\otimes - \epsilon_1)).$$
In addition, from (\ref{b8}) and (\ref{b4}), the vanishing of
$\Hom_G (V(2l\epsilon_1 +\epsilon_1),H^0(G/B,
S^{k}(\ul_{\alpha_n}^*)))$ forces the vanishing of
$\Hom_G (V(2l\epsilon_1 +\epsilon_1),H^1(G/B,
S^{k}(\ul^*)\otimes - \epsilon_1))$, for all $k$. Hence, $\sum_{u
\in W} (-1)^{\ell (u) } P_k(u\cdot((2l+1) \epsilon_1) +
\epsilon_1) = \dim \Hom_G (V(2l\epsilon_1 +\epsilon_1),H^0(G/B,
S^{k}(\ul^*)\otimes - \epsilon_1))$, which equals
$P(2l,k-n,1,n)$ by (\ref{b4}) and (\ref{b7}).
\end{proof}


\begin{proposition}\label{P:Pvanishing} Suppose $\Phi$ is of
type $B_n$ with $n \geq 1$ and $p >2n$. For $m \geq 0$, $k\geq
0$, and $ 1 \leq j \leq n$,
define
\begin{equation}\label{b15}
t (m,j,n)= \begin{cases} 
m+ \frac{j+1}{2} &\mbox{ if } j\mbox{ is odd and } m \mbox{ is
odd},\\
m+ \frac{j}{2} &\mbox{ if } j \mbox{ is even and } m \mbox{ is
even},\\
m+1+ \frac{2n-j}{2} &\mbox{ if } j \mbox{ is even and } m \mbox{
is odd},\\
m+1+ \frac{2n-j-1}{2} &\mbox{ if } j \mbox{ is odd and } m
\mbox{ is even}.
\end{cases}
\end{equation}
Then
$P(m,k,j,n) = 0 \mbox{ whenever } 
k  < t(m,j,n).$
\end{proposition}

\begin{proof} By Lemma \ref{L:Prelations}(c), $P(m,k,j,n) = 0$
if $k < m + 1$.
We will prove the slightly stronger statement in the proposition
inductively.
To do so, we make some general observations.  Define 
$$T'(m,k,j,n) = \sum_{i=2}^{2n}P(m-1,k-i +1,j,n)
+\sum_{i=1}^{2n}P(m,k-i +1,j-1,n-1) + P(m,k-n,j,n).$$
Observe that by Lemma \ref{L:Prelations}(d), $T'(m,k,j,n) =
T(m,k,j,n)$.
Note further that if $r$ is the smallest value of $k$ for which
$T'(m,k,j,n) \neq 0$,
then $P(m-1,r-1,j,n) + P(m,r,j-1,n-1) + P(m,r-n,j,n) \neq 0$. 

Suppose that $P(m-1,k,j,n)=0$ whenever $k < t(m-1,j,n)$ and that
$P(m,k,j-1,n-1)=0$
whenever $k < t(m,j-1,n-1)$, then one could conclude that
\begin{equation}\label{b16}
T'(m,k,j,n) = 0 \mbox{ whenever } k < \min\{t(m-1,j,n)+1,
t(m,j-1,n-1), m+1+n\}.
\end{equation}
Moreover, parts (d) and (e) of Lemma \ref{L:Prelations} would
imply that, for $ 2 \leq j \leq n-1$,
  \begin{equation}\label{b17}
P(m,k,j+1,n)+P(m,k,j-1,n) =0 \mbox{ whenever } T'(m,k,j,n)=0,   \end{equation}
   and from Lemma \ref{L:Prelations}(f)
  \begin{equation}\label{b18}
    P(m,k,n,n)+P(m,k,n-1,n)=0 \mbox{ whenever }  T'(m,k,n,n)=0.
 \end{equation}

In order to prove the proposition, we will use induction on $n$
and or $j$.
If $n=1$ the claim follows from part (c) of Lemma
\ref{L:Prelations}. Moreover, parts (c) and (d) of the Lemma~\ref{L:Prelations} 
imply that the claim holds for $j=1$ and $n \geq 1$.
\vskip .25cm 
\noindent 
{\bf Step 1:} Here we will show that $P(m,k,j,n)=0$, whenever $k
< t(m,j,n)$ and $m +j$
is odd.  We will use induction on $j$. 
\vskip .25cm 
\noindent 
{\bf Assumption: } $P(m,k,l,n)=0$, whenever $k < t(m,l,n)$, $m
+l$ is odd, and $l\leq j$.

Suppose that $m+j+1$ is odd. Then $m + j - 1$ is also odd and
the induction assumption implies that
 (\ref{b16}) holds. Together with  (\ref{b17})   one obtains 
\begin{equation*}
P(m,k,j+1,n) = 0 \mbox{ whenever } k < \min\{t(m-1,j,n)+1,
t(m,j-1,n-1), m+1+n\}.
\end{equation*}
It suffices therefore to verify that 
\begin{equation}\label{b19}
t(m,j+1,n) \leq \min\{t(m-1,j,n)+1, t(m,j-1,n-1), m+1+n\}.
\end{equation}
From (\ref{b15}) it follows that
\begin{equation*}
t(m,j+1,n)= \begin{cases} 
m+1+\frac{2n-(j+1)-1}{2} = m+n - \frac{j}{2}&\mbox{ if } j
\mbox{ is even and } m \mbox{ is even},\\
m +1 + \frac{2n-(j+1)}{2}=m +n -\frac{j-1}{2} &\mbox{ if } j
\mbox{ is odd and } m \mbox{ is odd},
\end{cases}
\end{equation*}
while
\begin{equation*}
t(m-1,j,n)+1 = \begin{cases} 

m+\frac{2n-j}{2}+1=m+n-\frac{j}{2} + 1 &\mbox{ if } j \mbox{ is
even and } m \mbox{ is even},\\
m+\frac{2n-j-1}{2}+1= m+n -\frac{j-1}{2} &\mbox{ if } j \mbox{
is odd and } m \mbox{ is odd},
\end{cases}
\end{equation*}
 and
\begin{equation*}
t(m,j-1,n-1) = \begin{cases} 

m+1+\frac{2n-2-(j-1)-1}{2}=m+n-\frac{j}{2} &\mbox{ if } j \mbox{
is even and } m \mbox{ is even},\\
m +1 + \frac{2n-2-(j-1)}{2}= m+n -\frac{j-1}{2}&\mbox{ if } j
\mbox{ is odd and } m \mbox{ is odd}.
\end{cases}
\end{equation*} 
Inequality (\ref{b19}) indeed holds and Step 1 is complete.
\vskip .25cm 
\noindent 
{\bf Step 2:} Here we will show that $P(m,k,n,n)=0$, whenever $k
< t(m,n,n)$ and $m+n$ is even.

Suppose that $m+n$ is even. Step 1 implies that 
(\ref{b16}) holds for $j=n$. Together with (\ref{b18}) one
obtains

$P(m,k,n,n)=0$ whenever $k < \min\{t(m-1,n,n)+1, t(m,n-1,n-1),
m+1+n\}.$
\\
It suffices therefore to verify that 
\begin{equation*}
t(m,n,n) \leq \min\{t(m-1,n,n)+1, t(m,n-1,n-1), m+1+n\}.
\end{equation*}
This can easily be done by looking at (\ref{b15}). It is left to
the interested reader.
\vskip .25cm 
\noindent 
{\bf Step 3:} Here we will show that $P(m,k,j,n)=0$ whenever $k
< t(m,j,n)$ and $m +j$
is even. We use induction on $n$ and on $j$. For $j$ we work in
decreasing order. The case $j=n$ was settled above.
\vskip .25cm 
\noindent 
{\bf Assumption: } We assume that $P(m,k,l,n-1)=0$ whenever $k <
t(m,l,n-1)$. In addition, we assume that $P(m,k,l,n)=0$ whenever
$k < t(m,l,n)$, $m +l$ is even, and $l\geq j$.

Suppose that $m+j-1$ is even. The induction assumptions imply
that
 (\ref{b16}) holds.  
By (\ref{b17})   one obtains 
\begin{equation*}
P(m,k,j-1,n) = 0 \mbox{ whenever } k < \min\{t(m-1,j,n)+1,
t(m,j-1,n-1), m+1+n\}.
\end{equation*}
It suffices therefore to verify that 
\begin{equation}\label{b20}
t(m,j-1,n) \leq \min\{t(m-1,j,n)+1, t(m,j-1,n-1), m+1+n\}.
\end{equation}
From (\ref{b15}) one obtains:
\begin{equation*}
t(m,j-1,n) = \begin{cases} 
m + \frac{j-1}{2}&\mbox{ if } j \mbox{ is odd and } m \mbox{ is
even},\\
m +\frac{j}{2} &\mbox{ if } j \mbox{ is even and } m \mbox{ is
odd}.
\end{cases}
\end{equation*}
while
\begin{equation*}
t(m-1,j,n)+1 = \begin{cases} 
m+\frac{j-1}{2}+1 &\mbox{ if } j \mbox{ is odd and } m \mbox{ is
even},\\
m +\frac{j}{2} &\mbox{ if } j \mbox{ is even and } m \mbox{ is
odd},
\end{cases}
\end{equation*}
and
 \begin{equation*}
t(m,j-1,n-1) = \begin{cases} 
m+\frac{j-1}{2} &\mbox{ if } j \mbox{ is odd and } m \mbox{ is
even},\\
m + \frac{j}{2}&\mbox{ if } j \mbox{ is even and } m \mbox{ is
odd}.
\end{cases}
\end{equation*} 
This proves inequality (\ref{b20}). 
\end{proof}


\begin{theorem}\label{T:bw1} Suppose $\Phi$ is of type $B_n$
with $n \geq 2$. Assume that $p > 2n$. Let $\la = p\omega_1 +
w\cdot 0$ be a dominant weight. Then
\begin{itemize}
\item[(a)]
$\opH^i(G, H^0(\la) \otimes H^0(\la)^{(1)} )=0$ for $0 < i <
2p-2,$
whenever $\ell(w)$ is even;
\item[(b)]
$\opH^i(G, H^0(\la) \otimes H^0(\la)^{(1)} )=0$ for $0 < i <
2p-3,$
whenever $\ell(w)$ is odd;
\item[(c)]
$\opH^{2p-3}(\gfp, k )\neq0$.
\end{itemize}
\end{theorem}

\begin{proof}
The set of dominant weights of the form $\la =p \omega_1 + w
\cdot 0$, written in the $\epsilon$-basis, are
\begin{equation*}
(p-\ell(w)-1)\epsilon_1 + (\epsilon_1 + \cdots +
\epsilon_{\ell(w)+1}), \mbox{ with } 0 \leq \ell(w) \leq n-1,
\end{equation*}  
and
\begin{equation*}
(p-\ell(w)-1)\epsilon_1 + (\epsilon_1 + \cdots +
\epsilon_{2n-\ell(w)}), \mbox{ with } n \leq \ell(w) \leq 2n-1 .
 \end{equation*}

Using Proposition \ref{P:Pvanishing} and Lemma~\ref{L:Prelations}(a) , a direct computation shows that
$$\sum_{u \in W} (-1)^{\ell (u) } P_k(u\cdot((p-\ell(w)-1)
\epsilon_1+ (\epsilon_1+ \cdots +
\epsilon_{\ell(w)+1}))-\epsilon_1) = 0 \mbox{ whenever }
k  < t, 
$$
where
$$t =\begin{cases} 
(p-1)- \frac{\ell(w)}{2} &\mbox{ for } 0 \leq \ell(w) \leq n-1
\mbox{ and } \ell(w) \mbox{ even},\\
(p-1)- \frac{\ell(w)+1}{2} &\mbox{ for } 0 \leq \ell(w) \leq n-1
\mbox{ and } \ell(w) \mbox{ odd},
\end{cases}
$$ 
and
$$\sum_{u \in W} (-1)^{\ell (u) } P_k(u\cdot((p-\ell(w)-1)
\epsilon_1+ (\epsilon_1+ .... +
\epsilon_{2n-\ell(w)}))-\epsilon_1) = 0 \mbox{ whenever }
k  < t, 
$$
where
$$t =\begin{cases} 
(p-1)- \frac{\ell(w)}{2} &\mbox{ for } n \leq \ell(w) \leq 2n-1
\mbox{ and } \ell(w) \mbox{ even},\\
(p-1)- \frac{\ell(w)+1}{2} &\mbox{ for } n \leq \ell(w) \leq
2n-1 \mbox{ and } \ell(w) \mbox{ odd}.
\end{cases}
$$ 
Parts (a) and (b) follow from Proposition
\ref{P:KostantPartCohom}. Note that $i = 2k + \ell(w)$.

Let $\la$ be the lowest dominant weight of the form $p \omega_1
+ w\cdot0$. Then $\la= (p-2n+1)\epsilon_1$ and $\ell(w)=2n-1$.
We will show that
\begin{equation}\label{b21}
\sum_{u \in W} (-1)^{\ell (u) } P_{p-n-1}(u\cdot((p-2n+1)
\epsilon_1)-\epsilon_1) \neq 0.
\end{equation}
By Lemma \ref{L:Prelations}(a) this is equivalent to showing
that $P(p-2n-1,p-n-1,1,n)$ is not zero. Lemma
\ref{L:Prelations}(b) and (h) imply that
\begin{equation}\label{b22}
P(2l+1,k,1,n)=P(2l,k-n,1,n) + P(2l+1,k,1,n-1).
\end{equation}
Note that (\ref{b22}) also holds for $l=-1$. Obviously $
P(2l-1,2l,1,1)=1$. It follows inductively from (\ref{b22}) that
$ P(2l-1,2l,1,n)\neq 0,$ for all $n\geq 0,$ $l \geq 0.$ From
Lemma \ref{L:Prelations}(g) one obtains now that $
P(2l,2l+n,1,n)\neq 0.$ Setting $2l = p-2n-1$ yields
$P(p-2n-1,p-n-1,1,n)\neq 0$. Hence, (\ref{b21}) holds.
In Proposition \ref{P:KostantPartCohom}, $i=2k -\ell(w)=2(p-n-1)
+ 2n-1=2p-3$ and one obtains $\opH^{2p-3}(G, H^0(\la) \otimes
H^0(\la)^{(1)} )\neq 0$. The weight $\la$ is the lowest non-zero
weight in its linkage class. Part (c) of the theorem follows
now from the discussion after Theorem \ref{T:nonvanishing}.
\end{proof}


\subsection{\bf Type $B_3$.}\label{SS:B3} Let $\Phi$ be of type
$B_3$ with $p > h = 6$ (so $p \geq 7$).
From the discussion in Section \ref{SS:B1} and Theorem
\ref{T:bw1}, in order to have
$\opH^i(G,H^0(\la)\otimes H^0(\la^*)^{(1)}) \neq 0$ for $0 < i <
2p -3$,
we must have $\la = p\omega_3 + w\cdot 0$ for
$w \in W$. With the aid of MAGMA \cite{BC,BCP} or other
software, one can
explicitly compute all $w\cdot 0$ and determine which resulting
$\la$ are dominant.
Further, $\la - \omega_3$ must be a weight of $S^{\frac{i -
\ell(w)}{2}}(\ul^*)$.
By direct computation, one can determine the least possible
value of
$k$ for which $\la - \omega_3$ can be a weight of $S^k(\ul^*)$.
The following table
summarizes the weights which can give a value of $i < 2p - 6$.

\vskip.4cm
\begin{center}
\begin{tabular}{|c|c|c|c|}\hline
$\la = p\omega_3 + w\cdot 0$ & $\ell(w)$ & $k$ & $i = 2k + \ell(w)$ \\
\hline
$(p - 6)\omega_3 + 2\omega_2$ & 3 & $p-5$ & $2p - 7$\\
\hline
$(p - 6)\omega_3 + \omega_1$ & 5 & $p - 6$ & $2p - 7$\\
\hline
$(p-6)\omega_3$ & 6 & $p-7$ & $2p - 8$\\
\hline
\end{tabular}
\end{center}

\begin{lemma}\label{L:B3A} Suppose that $\Phi$ is of type $B_3$
with $p\geq 7$.
Let $\la = p\mu + w\cdot 0 \in X(T)_+$ with $\mu \in X(T)_{+}$
and $w \in W$.
\begin{itemize}
\item[(a)] $\opH^i(G,H^0(\la)\otimes H^0(\la^*)^{(1)}) = 0$ for
$0 < i < 2p - 8$.
\item[(b)] If $\opH^{2p - 8}(G,H^0(\la)\otimes H^0(\la^*)^{(1)})
\neq 0$, then
$\la = (p - 6)\omega_3$.
\item[(c)] $\opH^{2p-8}(G,H^0((p-6)\omega_3)\otimes
H^0((p-6)\omega_3^*)^{(1)}) = k.$
\item[(d)] If $\opH^{2p - 7}(G,H^0(\la)\otimes H^0(\la^*)^{(1)})
\neq 0$, then
$\la = (p - 6)\omega_3 + \omega_1$ or $\la = (p-6)\omega_3 +
2\omega_2$.
\item[(e)] $\opH^{2p-8}(\gfp,k) = k$.
\end{itemize}
\end{lemma}

\begin{proof} Parts (a), (b), and (d) follow from the discussion
preceding the lemma.
Part (c) follows from Proposition \ref{P:KostantPartCohom} and 
Lemma \ref{L:B3B} below with $m = p - 7$.   Since 
the weights in part (d) are larger than $(p - 6)\omega_3$, 
by Theorem \ref{T:nonvanishing} and Theorem \ref{T:bw1}, we obtain part (e).
\end{proof}

\begin{remark} The weights in part (d) also appear to give
cohomology classes
as verified for $p = 7, 11, 13$ by computer.  For $p = 7$, 
$\la = (p - 6)\omega_3 + \omega_1$ gives a one-dimensional
cohomology group.
But for $p = 11, 13$, one gets a two-dimensional cohomology
group.
For all three primes, the weight $\la = (p - 6)\omega_3 +
2\omega_2$ gives a
one-dimensional cohomology group.
\end{remark}

\begin{lemma}\label{L:B3B} Suppose that $\Phi$ is of type $B_3$.
Let $m \geq 0$ be
an even integer. Then
$$
\sum_{u \in W}(-1)^{\ell(u)}P_m(u\cdot((m+1)\omega_3) -
\omega_3) = 1.
$$
\end{lemma}

\begin{proof} Let $n$ be such that $m = 2n$. For $n = 0$, the
claim readily follows,
so we assume that 
$n \geq 1$.\footnote{Indeed, for small values of $n$ the claim
can be verified by hand,
and it has been verified for $n \leq 6$ using MAGMA.} We work
with the epsilon basis
for the root system. Then the positive roots are $\epsilon_1$,
$\ep_2$, $\ep_3$,
$\ep_1 + \ep_2$, $\ep_1 + \ep_3$, $\ep_2 + \ep_3$, $\ep_1 -
\ep_2$, $\ep_1 - \ep_3$,
and $\ep_2 - \ep_3$. Further $\omega_3 = \frac12(\ep_1 + \ep_2 +
\ep_3)$.
Relative to the $\epsilon$ basis, for any $u \in W$, $u(\ep_i) =
\pm\ep_j$.
That is, $u$ permutes the $\ep_i$ up to a sign.  

For $u \in W$, set $x_u := u\cdot((m+1)\omega_3) - \omega_3$.  
Using the fact that $2\rho = 5\ep_1 + 3\ep_2 + \ep_3$, one finds
that
\begin{equation}\label{E:uonw3}
x_u = u((n+3)\ep_1 + (n+2)\ep_2 + (n+1)\ep_3) - 3\ep_1 - 2\ep_2
- \ep_1.
\end{equation}
By direct calculation, one can identify all $u \in W$ for which $x_u$ is a positive root.  There are twelve such 
elements which are summarized in the following table (using
permutation notation)
along with the parity of their lengths. An element marked with
a superscript
negative sign denotes the operation which consists of the given
permutation of the $\ep_i$s
followed by sending $\ep_3$ to $-\ep_3$. For example, let $u =
(123)^{-}$.
Then $u(\ep_1) = \ep_2$, $u(\ep_2) = - \ep_3$, and $u(\ep_3) =
\ep_1$.

\medskip
\begin{center}
\begin{tabular}{|c|c|c|c|c|c|c|c|c|c|c|c|c|}
\hline
$u$ & (1) & (12) & (13) & (23) & (123) & (132) & $(1)^{-}$ &
$(12)^{-}$ & $(13)^{-}$ & $(23)^{-}$ & $(123)^{-}$ &
$(132)^{-}$\\
\hline
$\ell(u)$ & even & odd & odd & odd & even & even & odd & even &
even & even & odd & odd\\
\hline
\end{tabular}
\end{center}

\medskip
For these twelve $u$, using (\ref{E:uonw3}), one can explicitly
compute $x_u$.
The values are summarized in the following table. Recall that $m
= 2n$. For small values
of $n$, some of these cannot be sums of positive roots. The
necessary condition on $n$ is
given in the third column.

\medskip
\begin{center}
\begin{tabular}{|c|c|c|}
\hline
$u$ & $x_u := u\cdot ((m+1)\omega_3) - \omega_3$ & positive root
sum\\
\hline
(1) & $n\ep_1 + n\ep_2 + n\ep_3$ & $n \geq 1$\\
\hline
(12) & $(n-1)\ep_1 + (n+1)\ep_2 + n\ep_3$ & $n \geq 1$\\
\hline
(13) & $(n-2)\ep_1 + n\ep_2 + (n+2)\ep_3$ & $n \geq 2$\\
\hline
(23) & $n\ep_1 + (n-1)\ep_2 + (n+1)\ep_3$ & $n \geq 1$\\
\hline
(123) & $(n-2)\ep_1 + (n+1)\ep_2 + (n + 1)\ep_3$ & $n \geq 2$\\
\hline
(132) & $(n-1)\ep_1 + (n-1)\ep_2 + (n + 2)\ep_3$ & $n \geq 1$\\
\hline
$(1)^{-}$ & $n\ep_1 + n\ep_2 - (n+2)\ep_3$ & $n \geq 2$\\
\hline
$(12)^{-}$ & $(n-1)\ep_1 + (n+1)\ep_2 - (n+2)\ep_3$ & $n \geq
2$\\
\hline
$(13)^{-}$ & $(n-2)\ep_1 + n\ep_2 - (n+4)\ep_3$ & $n \geq 6$\\
\hline
$(23)^{-}$ & $n\ep_1 + (n-1)\ep_2 - (n+3)\ep_3$ & $n \geq 4$\\
\hline
$(123)^{-}$ & $(n-2)\ep_1 + (n+1)\ep_2 - (n+3)\ep_3$ & $n \geq
4$\\
\hline
$(132)^{-}$ & $(n-1)\ep_1 + (n-1)\ep_2 - (n+4)\ep_3$ & $n \geq
6$\\
\hline
\end{tabular}
\end{center}

\medskip
We need to compute the appropriate alternating sum of the values
of $P_{2n}(x_u)$
for these twelve values of $u$. We show below that there are four pairs of $u$s for
which the
lengths have opposite parity and the values of $P_{2n}(x_u)$ are
the same.
Hence those cancel each other out. We will further show that there
is also
a relationship between the remaining partition functions that
will lead to the
desired claim. To see these relationships, we make a few
observations whose
proofs are left to the interested reader.

\begin{observation}\label{O:B31} Let $x = a_1\ep_1 + a_2\ep_2 +
a_3\ep_3$ with
$a_1 + a_2 + a_3 = 3n$. Suppose that $x$ is expressed as a sum
of $2n$ positive roots.
\begin{itemize}
\item[(a)]At least $n$ of the roots have the form $\ep_i +
\ep_j$ (not necessarily
all the same).  
\item[(b)] For any pair $i, j \in \{1, 2, 3\}$ (with $i \neq
j$), if $a_i + a_j = 2n + c$,
then the root sum decomposition contains 
at least $c$ copies of $\ep_i + \ep_j$.  
\end{itemize}
\end{observation}

\begin{observation}\label{O:B32} Let $x = a_1\ep_1 + a_2\ep_2 -
a_3\ep_3$
with $1 \leq a_1, a_2 < a_3$ and $a_1 + a_3 > 2n$. Suppose that
$x$ is expressed as
a sum of $2n$ positive roots.
Then at least one of the roots is $\ep_1 - \ep_3$ and at least
one is $\ep_2 - \ep_3$.
\end{observation}

We now identify the four pairs (of opposite parity) where
$P_{2n}(x_u)$ is the same.

\medskip\noindent
{\bf Case 1.} $(13)$ and $(132)$.

If $n = 1$, as noted above, $x_{(13)} = -\ep_1 + \ep_2 + 3\ep_3$
cannot be expressed
as a sum of positive roots. On the other hand, $x_{(123)} =
3\ep_3$ can be.
However, it cannot be expressed as a sum of $2n = 2$ positive
roots.
So we assume $n \geq 2$. 
If $x_{(13)}$ is expressed as a sum of $2n$ positive roots, by
Observation \ref{O:B31}(b),
at least one of the roots is $\ep_2 + \ep_3$ (in fact, at least
two).
Hence, $P_{2n}(x_{(13)}) = P_{2n-1}(x_{(13)} - (\ep_2 + \ep_3)) = 
P_{2n-1}((n-2)\ep_1 + (n-1)\ep_2 + (n+1)\ep_3)$. Similarly, if
$x_{(132)}$ is expressed
as a sum of $2n$ positive roots, at least one 
of the roots is $\ep_1 + \ep_3$ (and one is $\ep_2 + \ep_3$).
Hence, $P_{2n}(x_{(132)}) =
P_{2n-1}((n-2)\ep_1 + (n-1)\ep_2 + (n+1)\ep_3) =
P_{2n}(x_{(13)})$.

\medskip
For the remaining three pairs, if $n$ is not sufficiently large
for $x_u$ to admit
a positive root sum decomposition, then $P_{2n}(x_u) = 0$ in
both cases. So we assume
in what follows that $n$ is 
sufficiently large to admit a root sum decomposition.

\medskip\noindent
{\bf Case 2.} $(1)^{-}$ and $(12)^{-}$.

If $x_{(1)^{-}}$ is expressed as a sum of $2n$ positive roots,
by Observation \ref{O:B32},
at least one of the roots is $\ep_1 - \ep_3$. Hence, removing
this root,
$P_{2n}(x_{(1)^{-}}) = P_{2n-1}((n-1)\ep_1 + n\ep_2 - (n +
1)\ep_3)$.
Similarly, again by Observation \ref{O:B32},
if $x_{(12)^{-}}$ is expressed as a sum of $2n$ positive roots,
then at
least one of the roots is $\ep_2 - \ep_3$. Hence,
$P_{2n}(x_{(12)^{-}}) = P_{2n-1}((n-1)\ep_1 + n\ep_2 -
(n+1)\ep_3) = P_{2n}(x_{(1)^{-}}).$

\medskip\noindent
{\bf Case 3.} $(13)^{-}$ and $(132)^{-}$.

Similar to Case 2, by removing an $\ep_2 - \ep_3$ for $(13)^{-}$
and removing
an $\ep_1 - \ep_3$ for $(132)^{-}$, one finds that
$P_{2n}(x_{(12)^{-}}) =
P_{2n-1}((n-2)\ep_1 + (n-1)\ep_2 - (n + 3)\ep_3) =
P_{2n}(x_{(132)^{-}})$.

\medskip\noindent
{\bf Case 4.} $(23)^{-}$ and $(123)^{-}$.

Again, similar to Case 2 with a slight generalization of
Observation \ref{O:B32}, by removing
two copies of $\ep_1 - \ep_3$ for $(23)^{-}$ and two copies of
$\ep_2 - \ep_3$
for $(123)^{-}$, one finds that $P_{2n}(x_{(23)^{-}}) = 
P_{2n-2}((n-2)\ep_1 + (n-1)\ep_2 - (n + 1)\ep_3) =
P_{2n}(x_{(123)^{-}})$.

\medskip
From Cases 1-4, we have that
\begin{equation}\label{E:B3part}
\sum_{u \in W}(-1)^{\ell(u)}P_m(u\cdot((m+1)\omega_3) -
\omega_3) =
P_{2n}(x_{(1)}) - P_{2n}(x_{(12)}) - P_{2n}(x_{(23)}) +
P_{2n}(x_{(123)}).
\end{equation}
We now deduce several relationships among the terms on the right
hand side.
If $n = 1$, only the first three terms can be non-zero, and one
can readily
check that the claim holds. So we assume that $n \geq 2$. Note
that the following
argument does still hold even when $n = 1$. 

Consider the identity element $(1)$. Write $P_{2n}(x_{(1)}) =
M_1 + M_2 + M_3$ where
$M_1$ denotes the number of root sum decompositions which
contain an $\ep_1 + \ep_2$,
$M_2$ denotes the number which contain an $\ep_1 + \ep_3$ but 
not an $\ep_1 + \ep_2$, and $M_3$ denotes the number which
contain
neither an $\ep_1 + \ep_2$ nor an $\ep_1 + \ep_3$. For $M_1$, by
assumption, the
decomposition contains an $\ep_1 + \ep_2$. Removing this root
gives
\begin{equation}\label{E:B3M1}
M_1 = P_{2n-1}((n-1)\ep_1 + (n-1)\ep_2 + n\ep_3). 
\end{equation}
For $M_2$, by assumption, the decomposition contains an $\ep_1 +
\ep_3$.
Removing this root gives
\begin{equation}\label{E:B3M2}
M_2 = P_{2n-1}^*((n-1)\ep_1 + n\ep_2 + (n-1)\ep_3), 
\end{equation}
where $P^*$ denotes the fact that we are only counting
decompositions which
contain no copies of $\ep_1 + \ep_2$.
By assumption, $M_3$ is the number of root decompositions
of $n\ep_1 + n\ep_2 + n\ep_3$ (into $2n$ positive roots) which
do not contain an
$\ep_1 + \ep_2$ nor an $\ep_1 + \ep_3$. By Observation
\ref{O:B31}(a),
any such decomposition contains at least $n$ copies of $\ep_2 +
\ep_3$.
Removing those leaves $n\ep_1$ which 
must be expressed as a sum of $n$ positive roots without using
$\ep_1 + \ep_2$
nor $\ep_1 + \ep_3$.  There is clearly only one such 
decomposition (using $n$ copies of $\ep_1$).  Hence, $M_3 = 1$.

Consider now the word $(12)$.  Write
$P_{2n}(x_{(12)}) = N_1 + N_2$ where $N_1$ denotes the number of
root sum
decompositions which contain at least one copy of $\ep_1 +
\ep_2$ and $N_2$
denotes the number which do not contain an $\ep_1 + \ep_2$.  
In the first case, by removing an $\ep_1 + \ep_2$, we have 
\begin{equation}\label{E:B3N1}
N_1 = P_{2n-1}((n-2)\ep_1 + n\ep_2 + n\ep_3).
\end{equation}
In the second case (as well as the first case), by Observation
\ref{O:B31}(b),
any decomposition must include an $\ep_2 + \ep_3$. Removing
that, we see that
$$
N_2 = P^*_{2n-1}((n-1)\ep_1 + n\ep_2 + (n-1)\ep_3) = M_2
$$
from (\ref{E:B3M2}).

From Observation \ref{O:B31}(b), by removing an $\ep_1 +
\ep_3$,
$$
P_{2n}(x_{(23)}) = P_{2n-1}((n-1)\ep_1 + (n-1)\ep_2 + n\ep_3) =
M_1,
$$
where the second equality follows from (\ref{E:B3M1}).  
From Observation \ref{O:B31}(b), by removing an $\ep_2 + \ep_3$, 
$$
P_{2n}(x_{(123)}) = P_{2n-1}((n-2)\ep_1 + n\ep_2 + n\ep_3) =
N_1,
$$
where the second equality follows from
(\ref{E:B3N1})\footnote{If $n = 1$,
$x_{(123)}$ cannot be expressed as a sum of positive roots. 
However, in that case, $N_1$ is necessarily zero, and so we
still have
$P_{2n}(x_{(123)}) = N_1$.}.

From (\ref{E:B3part}) and the preceding relationships, we have
\begin{align*}
\sum_{u \in W}(-1)^{\ell(u)}P_m(u\cdot((m+1)\omega_3) -
\omega_3) &=
P_{2n}(x_{(1)}) - P_{2n}(x_{(12)}) - P_{2n}(x_{(23)}) +
P_{2n}(x_{(123)})\\
&= M_1 + M_2 + M_3 - N_1 - N_2 - P_{2n}(x_{(23)}) +
P_{2n}(x_{(123)})\\
&= M_1 + M_2 + 1 - N_1 - M_2 - M_1 + N_1\\
&= 1
\end{align*}
as claimed.

\end{proof}


\subsection{\bf Type $B_4$.}\label{SS:B4} Let $\Phi$ be of type
$B_4$ with $p > h = 8$ (so $p \geq 11$).
As discussed in Section \ref{SS:B3} for type $B_3$, in order to
have
$\opH^i(G,H^0(\la)\otimes H^0(\la^*)^{(1)}) \neq 0$ for $0 < i <
2p -3$,
we must have $\la = p\omega_4 + w\cdot 0$ for $w \in W$.  
Again, by direct computation with MAGMA, the following table
summarizes the weights which can give a value of $i < 2p - 6$.

\vskip.4cm
\begin{center}
\begin{tabular}{|c|c|c|c|}\hline
$\la = p\omega_4 + w\cdot 0$ & $\ell(w)$ & $k$ & $i = 2k + \ell(w)$ \\
\hline
$(p - 8)\omega_4 + 2\omega_2$ & $p-7$ & 7 & $2p - 7$\\
\hline
$(p - 8)\omega_4 + \omega_1$ & $p-8$ & 9 & $2p - 7$\\
\hline
$(p-8)\omega_4$ & $p-9$ & 10 & $2p - 8$\\
\hline
\end{tabular}
\end{center}

\begin{lemma}\label{L:B4A} Suppose that $\Phi$ is of type $B_4$
with $p \geq 11$.
Let $\la = p\mu + w\cdot 0 \in X(T)_+$ with $\mu \in X(T)_{+}$
and $w \in W$.
\begin{itemize}
\item[(a)] $\opH^i(G,H^0(\la)\otimes H^0(\la^*)^{(1)}) = 0$ for
$0 < i < 2p - 8$.
\item[(b)] If $\opH^{2p - 8}(G,H^0(\la)\otimes H^0(\la^*)^{(1)})
\neq 0$, then
$\la = (p - 8)\omega_4$.
\item[(c)] $\opH^{2p-8}(G,H^0((p-8)\omega_4)\otimes
H^0((p-8)\omega_4^*)^{(1)}) = k$.
\item[(d)] If $\opH^{2p - 7}(G,H^0(\la)\otimes H^0(\la^*)^{(1)})
\neq 0$, then
$\la = (p - 8)\omega_4 + \omega_1$ or $\la = (p-8)\omega_4 +
2\omega_2$.
\item[(e)] $\opH^{2p-8}(\gfp,k) = k$.
\end{itemize}
\end{lemma}

\begin{proof} Parts (a), (b), and (d) follow from the discussion
preceding the lemma.
Part (c) follows from Proposition \ref{P:KostantPartCohom} and
Lemma \ref{L:B4B} below
with $m = p - 9$.
Since the weights in part (d) are larger than $(p - 8)\omega_4$,
by Theorem
\ref{T:nonvanishing} and Theorem \ref{T:bw1}, we obtain part (e).  
\end{proof}

\begin{lemma}\label{L:B4B} Suppose that $\Phi$ is of type $B_4$.
Let $m \geq 0$ be
an even integer. Then
$$
\sum_{u \in W}(-1)^{\ell(u)}P_m(u\cdot((m+1)\omega_4) -
\omega_4) = 1.
$$
\end{lemma}

\begin{proof} The arguments to follow are quite similar to those
in the proof of
Lemma \ref{L:B3B}. Let $n$ be such that $m = 2n$. For $n = 0$,
the claim readily follows,
so we assume that $n \geq 1$. As with the proof of Lemma
\ref{L:B3B}, we work with
the epsilon basis for the root system. Then the positive roots
are
$\epsilon_1$, $\ep_2$, $\ep_3$, $\ep_4$, $\ep_1 + \ep_2$, $\ep_1
+ \ep_3$, $\ep_1 + \ep_4$,
$\ep_2 + \ep_3$, $\ep_2 + \ep_4$, $\ep_3 + \ep_4$, $\ep_1 -
\ep_2$, $\ep_1 - \ep_3$,
$\ep_1 - \ep_4$, $\ep_2 - \ep_3$, $\ep_2 - \ep_4$, and $\ep_3 -
\ep_4$.
Further $\omega_4 = \frac12(\ep_1 + \ep_2 + \ep_3 + \ep_4)$.
Relative to the $\epsilon$ basis, for any $u \in W$, $u(\ep_i) =
\pm\ep_j$.
That is, $u$ permutes the $\ep_i$ up to a sign.  

For $u \in W$, let 
$x_u := u\cdot((m+1)\omega_4) - \omega_4$.  Using the fact that $2\rho = 7\ep_1 + 5\ep_2 + 3\ep_3 + \ep_4$, one finds that
\begin{equation}\label{E:uonw4}
x_u = u((n+4)\ep_1 + (n+3)\ep_2 + (n+2)\ep_3 + (n+1)\ep_4)
- 4\ep_1 - 3\ep_2 - 2\ep_3 - \ep_1.
\end{equation}
By direct calculation, one finds that if $u$ sends any $\ep_i$
to $-\ep_j$ (any $j$),
then $x_u$ is either not expressible as a sum of 
positive roots or requires at least $2n+1$ roots to do so.
Therefore, the only
$u$ that can contribute to the alternating sum under
consideration are those $u$
for which $u(\ep_i) = \ep_j$. That is, $u$ is simply one of the
24 permutations of
the $\ep_i$s.  

Let $u \in S_4 \subset W$.  From (\ref{E:uonw4}), one finds that $x_u = a_1\ep_1 + a_2\ep_2 + a_3\ep_3 + a_4\ep_4$ where
$a_1 + a_2 + a_3 + a_4 = 4n$. 
Since the positive roots are of the form $\ep_i$, $\ep_i +
\ep_j$,
or $\ep_i - \ep_j$, for this to be expressed as a sum of $2n$
roots, each such
root must be of the form $\ep_i + \ep_j$. That is the other two
types of roots are not
allowable. Similar to the arguments in the proof of Lemma
\ref{L:B3B}, one can further
see the following.

\begin{observation}\label{O:B41} Suppose that $a_1\ep_1 +
a_2\ep_2 + a_3\ep_3 + a_4\ep_4$
is expressed as a sum of $2n$ positive roots where $a_1 + a_2 +
a_3 + a_4 = 4n$.
For any pair $i, j \in \{1,2,3,4\}$ 
(with $i \neq j$), if $a_i + a_j = 2n + c$, then the root sum
decomposition contains
at least $c$ copies of $\ep_i + \ep_j$.  
\end{observation}

Using Observation \ref{O:B41}, by direct calculation, one can
show that the 18
permutations $u$ for which $u(\ep_1) \neq \ep_1$ can be
separated into 9 pairs
of opposite parity having equal values of $P_{2n}(x_u)$.
Hence the terms for those values of $u$ cancel in the
alternating sum.
For example, consider the permutations $(12)$ and $(12)(43)$ of
opposite parity.
From (\ref{E:uonw4}), $x_{(12)} = (n-1)\ep_1 + (n+1)\ep_2 +
n\ep_3 + n\ep_4$.
By Observation \ref{O:B41}, a decomposition of $x_{(12)}$ must
contain at least one
copy of $\ep_2 + \ep_3$ (as well as a copy of $\ep_2 + \ep_4$).
Subtracting that root
shows that
$$
P_{2n}(x_{(12)}) = P_{2n-1}((n-1)\ep_1 + n\ep_2 + (n-1)\ep_3 +
n\ep_4).
$$
On the other hand, $x_{(12)(34)} = (n-1)\ep_1 + (n+1)\ep_2 +
(n-1)\ep_3 + (n+1)\ep_4$.
Here, $x_{(12)(34)}$ must contain a copy of $\ep_2 + \ep_4$ (in
fact, at least two copies).
Subtracting this root gives
$$
P_{2n}(x_{(12)(34)}) = P_{2n-1}((n-1)\ep_1 + n\ep_2 + (n-1)\ep_3
+ n\ep_4) =
P_{2n}(x_{(12)}).
$$
The eight other pairings (which may not be unique) are
$(13)$ with $(13)(24)$; $(14)$ with $(14)(23)$; $(123)$ with
$(1243)$;
$(132)$ with $(1342)$; $(124)$ with $(1234)$; $(142)$ with
$(1432)$;
$(134)$ with $(1324)$; and $(143)$ with $(1423)$. We leave the
details to the
interested reader.

That leaves the six values of $u$ for which $u(\ep_1) = \ep_1$:
$(1)$, $(23)$, $(24)$,
$(34)$, $(234)$, and $(243)$. However, as above, one can show
that
$P_{2n}(x_{(24)}) = P_{2n}(x_{(243)})$. So those terms cancel as
well and we are reduced to
$$
\sum_{u \in W}(-1)^{\ell(u)}P_m(u\cdot((m+1)\omega_4) -
\omega_4) =
P_{2n}(x_{(1)}) - P_{2n}(x_{(23)}) - P_{2n}(x_{(34)}) +
P_{2n}(x_{(234)}).
$$

From (\ref{E:uonw4}), we have
$x_{(1)} = n\ep_1 + n\ep_2 + n\ep_3 + n\ep_4$.
Write $P_{2n}(x_{(1)}) = M_1 + M_2 + M_3$ where $M_1$ denotes
the number of root sum
decompositions which contain at least one copy of $\ep_2 +
\ep_3$,
$M_2$ denotes the number which contain no copies of $\ep_2 +
\ep_3$ but contain
at least one copy of $\ep_1 + \ep_2$, and $M_3$ denotes the
number which
contain neither an $\ep_2 + \ep_3$ nor an $\ep_1 + \ep_2$. By
assumption,
subtracting a copy of $\ep_2 + \ep_3$, we have
\begin{equation}\label{E:B4M1}
M_1 = P_{2n-1}(n\ep_1 + (n-1)\ep_2 + (n-1)\ep_3 + n\ep_4).
\end{equation}
For $M_2$, subtracting a copy of $\ep_1 + \ep_2$ gives
\begin{equation}\label{E:B4M2}
M_2 = P_{2n-1}^*((n-1)\ep_1 + (n-1)\ep_2 + n\ep_3 + n\ep_4),
\end{equation}
where the $P^*$ denotes the fact that the sum is only over those
decompositions
which do not contain
a copy of $\ep_2 + \ep_3$.  For $M_3$, in order to 
get the $n\ep_2$ appearing in $x_{(1)}$, there must be exactly
$n$ copies of $\ep_2 + \ep_4$.
But then the remaining $n$ factors must all be $\ep_1 + \ep_3$.
In other words,
$M_3 = 1$.  

From (\ref{E:uonw4}), we have
$x_{(23)} = n\ep_1 + (n-1)\ep_2 + (n+1)\ep_3 + n\ep_4$.
Write $P_{2n}(x_{(23)}) = N_1 + N_2$ where $N_1$ denotes the
number of root
sum decompositions which contain at least one copy of $\ep_2 +
\ep_3$ and
$N_2$ denotes the number which contain no copies of $\ep_2 +
\ep_3$.
Subtracting a copy of $\ep_2 + \ep_3$, we have
\begin{equation}\label{E:B4N1}
N_1 = P_{2n-1}(n\ep_1 + (n-2)\ep_2 + n\ep_3 + n\ep_4).
\end{equation}
For $N_2$, by Observation \ref{O:B41}, any decomposition of
$x_{(23)}$ contains
at least one copy of $\ep_1 + \ep_3$ (as well as a copy of
$\ep_3 + \ep_4$).
Subtracting the $\ep_1 + \ep_3$ gives
$$
N_2 = P_{2n-1}^*((n-1)\ep_1 + (n-1)\ep_2 + n\ep_3 + n\ep_4) =
M_2
$$
from (\ref{E:B4M2}).

From (\ref{E:uonw4}), we have
$x_{(34)}= n\ep_1 + n\ep_2 + (n-1)\ep_3 + (n+1)\ep_4$.
From Observation \ref{O:B41}, any decomposition of $x_{(34)}$
contains at least
one copy of $\ep_2 + \ep_4$ (as well as a copy of $\ep_1 +
\ep_4$). Subtracting the
$\ep_2 + \ep_4$ gives
$$
P_{2n}(x_{(34)}) = P_{2n-1}(n\ep_1 + (n-1)\ep_2 + (n-1)\ep_3 +
n\ep_4) = M_1
$$
from (\ref{E:B4M1}).

From (\ref{E:uonw4}), we have
$x_{(234)} = n\ep_1 + (n-2)\ep_2 + (n+1)\ep_3 + (n+1)\ep_4$.
From Observation \ref{O:B41}, any decomposition of $x_{(234)}$
contains at least
one copy of $\ep_3 + \ep_4$ (in fact, at least two copies).
Subtracting this gives
$$
P_{2n}(x_{(234)}) = P_{2n-1}(n\ep_1 + (n-2)\ep_2 + n\ep_3 +
n\ep_4) = N_1
$$
from (\ref{E:B4N1}).

In summary, we have
\begin{align*}
\sum_{u \in W}(-1)^{\ell(u)}P_m(u\cdot((m+1)\omega_4) -
\omega_4) &=
P_{2n}(x_{(1)}) - P_{2n}(x_{(23)}) - P_{2n}(x_{(34)}) +
P_{2n}(x_{(234)})\\
&= M_1 + M_2 + M_3 - N_1 - N_2 - P_{2n}(x_{34}) +
P_{2n}(x_{234})\\
&= M_1 + M_2 + 1 - N_1 - M_2 - M_1 + N_1\\
&= 1
\end{align*}
as claimed.
\end{proof}


\subsection{\bf Type $B_5$.}\label{SS:B5} Let $\Phi$ be of type
$B_5$ with $p > h = 10$ (so $p \geq 11$).
As discussed in Section \ref{SS:B3} for type $B_3$, in order to
have
$\opH^i(G,H^0(\la)\otimes H^0(\la^*)^{(1)}) \neq 0$ for $0 < i <
2p-3$,
we must have $\la = p\omega_5 + w\cdot 0$ for $w \in W$. 
Specifically, substituting $n = 5$ into (\ref{E:omegan}) gives
\begin{equation}\label{E:B5condition}
i \geq 2p - \frac{p - 25}{2}.
\end{equation}
We obtain the following.

\begin{lemma}\label{L:B5} Suppose that $\Phi$ is of type $B_5$
with $p \geq 11$.
Let $\la = p\omega_5 + w\cdot 0 \in X(T)_{+}$ with $w \in W$.
\begin{itemize}
\item[(a)] If $p = 17$ or $p \geq 23$, then
$\opH^i(G,H^0(\la)\otimes H^0(\la^*)^{(1)}) = 0$
for $0 < i \leq 2p - 3$.
\item[(b)] Suppose $p = 11$. Then
\begin{itemize}
\item[(i)] $\opH^i(G,H^0(\la)\otimes H^0(\la^*)^{(1)}) = 0$ for
$0 < i < 2p - 7$;
\item[(ii)] if $\opH^{2p - 7}(G,H^0(\la)\otimes
H^0(\la^*)^{(1)}) \neq 0$, then
$\la = (p - 10)\omega_5 = \omega_5$;
\item[(iii)] $\opH^{2p-7}(G,H^0(\omega_5)\otimes
H^0(\omega_5^*)^{(1)}) = k$;
\item[(iv)] if $\opH^{2p - 6}(G,H^0(\la)\otimes
H^0(\la^*)^{(1)}) \neq 0$, then
$\la = \omega_1 + \omega_5$ or $\la = 2\omega_2 + \omega_5$;
\item[(v)] $\opH^{2p-6}(G,H^0(\la)\otimes H^0(\la^*)^{(1)}) = k$
for
$\la = \omega_1 + \omega_5$ or $\la = 2\omega_2 + \omega_5$;
\item[(vi)] $\opH^{2p-7}(\gfp,k) = k$.
\end{itemize}
\item[(c)] Suppose $p = 13$. Then
\begin{itemize}
\item[(i)] $\opH^i(G,H^0(\la)\otimes H^0(\la^*)^{(1)}) = 0$ for
$0 < i < 2p - 5$;
\item[(ii)] if $\opH^{2p - 5}(G,H^0(\la)\otimes
H^0(\la^*)^{(1)}) \neq 0$, then
$\la = (p - 10)\omega_5 = 3\omega_5$;
\item[(iii)] $\opH^{2p-5}(G,H^0(3\omega_5)\otimes
H^0(3\omega_5^*)^{(1)}) = k$;
\item[(iv)] if $\opH^{2p - 4}(G,H^0(\la)\otimes
H^0(\la^*)^{(1)}) \neq 0$, then
$\la = \omega_1 + 3\omega_5$ or $\la = 2\omega_2 + 3\omega_5$;
\item[(v)] $\dim\opH^{2p-4}(G,H^0(\la)\otimes H^0(\la^*)^{(1)})
= 2$ for
$\la = \omega_1 + 3\omega_5$;
\item[(vi)] $\opH^{2p-4}(G,H^0(\la)\otimes H^0(\la^*)^{(1)}) =
k$ for
$\la = 2\omega_2 + 3\omega_5$;
\item[(vii)] $\opH^{2p-5}(\gfp,k) = k$.
\end{itemize}
\item[(d)] Suppose $p = 19$. Then
\begin{itemize}
\item[(i)] $\opH^i(G,H^0(\la)\otimes H^0(\la^*)^{(1)}) = 0$ for
$0 < i < 2p - 3$;
\item[(ii)] if $\opH^{2p - 3}(G,H^0(\la)\otimes
H^0(\la^*)^{(1)}) \neq 0$, then
$\la = (p - 10)\omega_5 = 9\omega_5$;
\item[(iii)] $\dim\opH^{2p-3}(G,H^0(9\omega_5)\otimes
H^0(9\omega_5^*)^{(1)}) = 15$.
\end{itemize}
\end{itemize}
\end{lemma}

\begin{proof} For $p \geq 23$, part (a) follows from
(\ref{E:B5condition}). Parts (b)(i)-(v), (c)(i)-(vi), and (d)
as well as part (a) for $p = 17$
follow by explicitly computing (with the aid of MAGMA) all
possible $w \cdot 0$,
and then computing partition functions by hand or with the aid
of MAGMA. For $p = 11$, since the weights in part (b)(iv)
are larger than that in
(b)(ii), part (b)(vi) follows from Theorem \ref{T:nonvanishing}
and Theorem \ref{T:bw1}.  Similarly, part (c)(vii) follows.
\end{proof}


\subsection{\bf Type $B_6$.}\label{SS:B6} Let $\Phi$ be of type
$B_6$ with $p > h = 12$ (so $p \geq 13$).
As discussed in Section \ref{SS:B3} for type $B_3$, in order to
have
$\opH^i(G,H^0(\la)\otimes H^0(\la^*)^{(1)}) \neq 0$ for $0 < i <
2p-3$,
we must have $\la = p\omega_6 + w\cdot 0$ for $w \in W$. 
Recall the arguments in Section \ref{SS:B1}.  Specifically, 
substituting $n = 6$ into (\ref{E:omegan}) gives
\begin{equation}\label{E:B6condition}
i \geq 2p + (p - 18).
\end{equation}
We obtain the following.

\begin{lemma}\label{L:B6} Suppose that $\Phi$ is of type $B_6$
with $p\geq 13$.
Let $\la = p\omega_6 + w\cdot 0 \in X(T)_{+}$ with $w \in W$.
\begin{itemize}
\item[(a)] If $p \geq 17$, then $\opH^i(G,H^0(\la)\otimes
H^0(\la^*)^{(1)}) = 0$
for $0 < i \leq 2p - 3$.
\item[(b)] Suppose $p = 13$. Then
\begin{itemize}
\item[(i)] $\opH^i(G,H^0(\la)\otimes H^0(\la^*)^{(1)}) = 0$ for
$0 < i < 2p - 5$;
\item[(ii)] if $\opH^{2p - 5}(G,H^0(\la)\otimes
H^0(\la^*)^{(1)}) \neq 0$, then
$\la = (p - 12)\omega_6 = \omega_6$;
\item[(iii)] $\opH^{2p-5}(G,H^0(\omega_6)\otimes
H^0(\omega_6^*)^{(1)}) = k$;
\item[(iv)] if $\opH^{2p - 4}(G,H^0(\la)\otimes
H^0(\la^*)^{(1)}) \neq 0$, then
$\la = \omega_1 + \omega_6$ or $\la = 2\omega_2 + \omega_6$;
\item[(v)] $\opH^{2p-4}(G,H^0(\la)\otimes H^0(\la^*)^{(1)}) = k$
for
$\la = \omega_1 + \omega_6$ or $\la = 2\omega_2 + \omega_6$;
\item[(vi)] $\opH^{2p - 5}(\gfp,k) = k$.
\end{itemize}
\end{itemize}
\end{lemma}

\begin{proof} Part (a) follows from (\ref{E:B6condition}). Parts
(b)(i)-(v)
follow by explicitly computing (with the aid of MAGMA) all
possible $w \cdot 0$,
and then computing partition functions by hand or with the aid
of MAGMA. Since the weights in part (b)(iv)
are larger than that in part (b)(ii), part (b)(vi) follows from 
Theorem \ref{T:nonvanishing} and Theorem \ref{T:bw1}.
\end{proof}


\subsection{\bf Summary for type $B$.}\label{S:Bsumm}

\begin{theorem}\label{T:Bsumm} Suppose $\Phi$ is of type $B_n$ with $n \geq 3$.
Assume that $p > 2n$.
\begin{itemize}
\item[(a)] If $n \geq 7$ or $p > 13$ when $n \in \{5, 6\}$, then\begin{itemize}
\item[(i)] $\opH^i(\gfp,k) = 0$ for $0 < i < 2p - 3$;
\item[(ii)] $\opH^{2p - 3}(\gfp,k) \neq 0$.
\end{itemize}
\item[(b)] If $n \in \{5, 6\}$ and $p = 13$, then
\begin{itemize}
\item[(i)] $\opH^i(\gfp,k) = 0$ for $0 < i < 2p - 5$;
\item[(ii)] $\opH^{2p - 5}(\gfp,k) = k$.
\end{itemize}
\item[(c)] If $n = 5$ and $p = 11$, then 
\begin{itemize}
\item[(i)] $\opH^i(\gfp,k) = 0$ for $0 < i < 2p - 7$;
\item[(ii)] $\opH^{2p - 7}(\gfp,k) = k$.
\end{itemize}
\item[(d)] If $n \in \{3,4\}$, then 
\begin{itemize}
\item[(i)] $\opH^i(\gfp,k) = 0$ for $0 < i < 2p - 8$;
\item[(ii)] $\opH^{2p - 8}(\gfp,k) = k$.
\end{itemize}
\end{itemize}
\end{theorem}

\begin{proof} This follows from the discussion in Section
\ref{SS:B1}, Theorem \ref{T:bw1},
Lemma \ref{L:B3A}, Lemma \ref{L:B4A}, Lemma \ref{L:B5}, and
Lemma \ref{L:B6}.
\end{proof}


\section{Type $G_2$}\label{typeG}

Assume throughout this section that $\Phi$ is of type $G_2$ and
that $p > h = 6$ (so $p \ge 7$). Following Section 2, our goal is to find the
least $i > 0$
such that $\opH^i(G,H^0(\la)\otimes H^0(\la^*)) \neq 0$ for some
$\la \in X(T)_+$.


\subsection{\bf Restrictions.}\label{Gres} Suppose that 
$\opH^i(G,H^0(\la)\otimes H^0(\la^*)^{(1)}) \neq 0$ for some
$i > 0$ and $\la = p\mu + w\cdot 0$ with $\mu \in X(T)_{+}$ and
$w \in W$.
From Proposition~\ref{P:degreebound1}(c), $i \geq
(p-1)\langle\mu,\ta^{\vee}\rangle - 1$.
Consider the two fundamental weights $\omega_1$ and $\omega_2$.
Note that
$\omega_1 = \al_0$ and $\omega_2 = \ta$.  Furthermore, we have
$\langle\omega_1,\ta^{\vee}\rangle = 1$ and
$\langle\omega_2,\ta^{\vee}\rangle = 2$.
Therefore, unless $\mu = \omega_1 = \al_0$, we have
$\langle\mu,\ta^{\vee}\rangle \geq 2$ and $i \geq 2p - 3$.

Suppose now that $\la = p\omega_1 + w\cdot 0$ for some $w \in
W$. In order to
have $\opH^i(G,H^0(\la)\otimes H^0(\la^*)^{(1)}) \neq 0$, as
discussed in Section~\ref{SS:G1analysis},
$\la$ must be dominant and $\la - \omega_1$ must be a weight of
$S^j(\ul^*)$ for some $j$.
In other words, 
$\la - \omega_1$ must be expressible as a non-negative linear
combination of positive roots.
By direct calculation (by hand or with the aid of MAGMA), one
can identify all possible $\la$.
These are listed in the following table. As usual, $s_i :=
s_{\al_i}$ and $e$ is the identity element.

\vskip.4cm
\begin{center}
\begin{tabular}{|c|c|c|c|c|}\hline
$w$ & $\ell(w)$ & $\la = p\omega_1 + w\cdot 0$\\
\hline 
$e$ & $0$ & $p\omega_1$\\
\hline
$s_1$ & $1$ & $(p - 2)\omega_1 + \omega_2$\\
\hline
$s_1s_2$ & $2$ & $(p - 5)\omega_1 + 2\omega_2$\\
\hline
$s_1s_2s_1$ & $3$ & $(p - 6)\omega_1 + 2\omega_2$\\
\hline
$s_1s_2s_1s_2$ & $4$ & $(p - 6)\omega_1 + \omega_2$\\
\hline
$s_1s_2s_1s_2s_1$ & $5$ & $(p - 5)\omega_1$\\
\hline
\end{tabular}
\end{center}

\vskip.4cm
Note that each $\la$ has the form $\la = a\omega_1 + b\omega_2$
for
$a \geq 1$ and $0 \leq b \leq 2$. From
Proposition~\ref{P:KostantPartCohom}, we know that
for $\la = p\mu + w\cdot 0$,
$$
\dim\opH^i(G,H^0(\la)\otimes H^0(\la^*)^{(1)}) = \sum_{u \in
W}(-1)^{\ell(u)}
	P_{\frac{i - \ell(w)}{2}}(u\cdot\la - \mu).
$$
In the next section, we consider such partition functions. Since
the prime
$p$ does not per se play a role in the partition function
computations, we
will work in a general setting.


\subsection{\bf Partitions I}\label{part1}

Let $\la = a\omega_1 + b\omega_2$ for $a \geq 1$ and
$0 \leq b \leq 2$. From the previous section, our goal is to
make computations of
\begin{equation}\label{sum}
\sum_{u \in W}(-1)^{\ell(u)}P_k(u\cdot \la - \omega_1).
\end{equation}

In particular, we will identify the least value of $k$ for which
this sum is non-zero
along with the value of the sum in that case. See
Proposition~\ref{P:G2prop1} and Proposition~\ref{P:G2prop2}.

In order for $P_k(u\cdot \la - \omega_1)$ to be non-zero,
$u\cdot\la - \omega_1$
must lie in the positive (more precisely, non-negative) root
lattice. By direct computation,
one finds that there are only four elements $u \in W$ for which
this is true
(under our assumptions on $a$ and $b$ above). This is summarized
in the following table.
The value of $u\cdot\la - \omega_1$ is given in the root basis.
\vskip.4cm
\begin{center}
\begin{tabular}{|c|c|c|c|c|}\hline
$u$ & $\ell(u)$ & $u\cdot\la - \omega_1$ \\
\hline 
$e$ & $0$ & $(2a + 3b - 2)\al_1 + (a + 2b - 1)\al_2$\\
\hline
$s_1$ & $1$ & $(a + 3b - 3)\al_1 + (a + 2b - 1)\al_2$\\
\hline
$s_2$ & $1$ & $(2a + 3b - 2)\al_1 + (a + b - 2)\al_2$\\
\hline
$s_1s_2$ & $2$ & $(a - 6)\al_1 + (a + b - 2)\al_2$\\
\hline
\end{tabular}
\end{center}
\vskip.4cm

Note that in some of the cases $a$ must be sufficiently large in
order for
$u\cdot\la - \omega_1$ to lie in the positive root lattice.  Specifically,
for $s_1$, one needs $a \geq 3$ or $b \geq 1$; for $s_2$, one
needs
$a \geq 2$ or $a \geq 1$ and $b \geq 1$ or $b \geq 2$;  and for
$s_1s_2$, one
needs $a \geq 6$.


\subsection{\bf Partitions II}\label{part2}
As noted in Section~\ref{part1}, our goal is to find the least
value of $k$
such that the sum (\ref{sum}) is non-zero. In this section, we
notice
some relationships among the partition functions which will
allow us to identify
a range under which the sum is zero.  

\begin{lemma} \label{L:G2comb1} Let $\la = a\omega_1 +
b\omega_2$ with $a \geq 3$ and $0 \leq b \leq 2$.
Suppose that $k \leq a + b - 2$. Then 
$$
P_k(\la - \omega_1) = P_k(s_2\cdot\la - \omega_1).
$$
\end{lemma}

\begin{proof} Recall the table in Section~\ref{part1}, and set
\begin{align*}
\ga_1 &:= \la - \omega_1 = (2a + 3b - 2)\al_1 + (a + 2b -
1)\al_2,\\
\ga_2 &:= s_2\cdot\la - \omega_1 = (2a + 3b - 2)\al_1 + (a + b -
2)\al_2.
\end{align*}
Consider a decomposition of $\ga_1$ into $k$ not necessarily
distinct positive
roots. Since $a + 2b - 1 = (a + b - 2) + (b + 1)$ and $k \leq a
+ b - 2$,
at least $b + 1$ of those roots must contain $2\al_2$. However,
the only root containing
$2\al_2$ is $\ta = 3\al_1 + 2\al_2$.
Hence, any decomposition of $\ga_1$ into $k$ roots must contain
at least $b + 1$
copies of $\ta$.  Therefore
\begin{equation}\label{rel1}
P_k(\ga_1) = P_{k - b - 1}(\ga_1 - (b+1)(3\al_1 + 2\al_2)) = 
	P_{k - b - 1}((2a - 5)\al_1 + (a - 3)\al_2).
\end{equation}

Now consider $\ga_2$ and the difference between the number of
$\al_1$s and $\al_2$s
appearing. Suppose $\ga_2 = \sum(m_i\al_1 + n_i\al_2)$ is
expressed as a sum of
$k$ positive roots.  Then
$$
\sum(m_i - n_i) = \sum m_i - \sum n_i = (2a + 3b - 2) - (a + b -
2) = a + 2b.
$$
Note that for each $i$, $m_i - n_i \in \{-1, 0, 1, 2\}$.
Since $k \leq a + b - 2$, for at least $b + 1$ values of $i$ (in
fact, at least $b + 2$ values),
we must have $m_i - n_i = 2$. However, the only root where that
occurs is $3\al_1 + \al_2$.
Hence, any decomposition of $\ga_2$ into $k$ roots must contain
at least $b + 1$ copies of
$3\al_1 + \al_2$. Therefore, 
$$
P_k(\ga_2) = P_{k - b - 1}(\ga_2 - (b + 1)(3\al_1 + \al_2)) = 
	P_k((2a - 5)\al_1 + (a - 3)\al_1).
$$
Combining this with (\ref{rel1}) gives the claim.
\end{proof}

\begin{lemma} \label{L:G2comb2} Let $\la = a\omega_1 +
b\omega_2$ with $a \geq 6$.
Suppose that $k \leq a + b - 2$. Then 
$$
P_k(s_1\cdot\la - \omega_1) = P_k(s_1s_2\cdot\la - \omega_1).
$$
\end{lemma}

\begin{proof} Recall the table in Section \ref{part1}, and set
\begin{align*}
\ga_3 &:= s_1\cdot\la - \omega_1 = (a + 3b - 3)\al_1 + (a + 2b -
1)\al_2,\\
\ga_4 &:= s_1s_2\cdot\la - \omega_1 = (a - 6)\al_1 + (a + b -
2)\al_2.
\end{align*}
For $\ga_3$, the same argument as in the preceding lemma gives
\begin{equation}\label{rel2}
P_k(\ga_3) = P_{k - b - 1}(\ga_3 - (b+1)(3\al_1 + 2\al_2)) = 
	P_{k - b - 1}((a - 6)\al_1 + (a - 3)\al_2).
\end{equation}

Now consider $\ga_4$ and the difference between the number of
$\al_1$s and $\al_2$s
appearing. Suppose $\ga_4 = \sum(m_i\al_1 + n_i\al_2)$ is
expressed as a sum of
$k$ positive roots.  Then
$$
\sum(m_i - n_i) = \sum m_i - \sum n_i = (a - 6) - (a + b - 2) =
-b - 4.
$$
Note that for each $i$, $m_i - n_i \in \{-1, 0, 1, 2\}$.
For at least $b + 1$ values of $i$ (in fact, at least $b + 4$
values), we must
have $m_i - n_i = -1$. However, the only root where that occurs
is $\al_2$.
Hence, any decomposition of $\ga_4$ into $k$ roots must contain
at least $b + 1$ copies of
$\al_2$. Therefore
$$
P_k(\ga_4) = P_{k - b - 1}(\ga_2 - (b + 1)\al_2) = 
	P_{k-b-1}((a - 6)\al_1 + (a - 3)\al_2).
$$
Combining this with (\ref{rel2}) gives the claim.
\end{proof}

With the two aforementioned lemmas we can now prove the
following proposition.

\begin{proposition}\label{P:G2prop1} Let $\la = a\omega_1 +
b\omega_2$ for $a \geq 1$ and $0 \leq b \leq 2$.
For $k \leq a + b - 2$,
$$
\sum_{u \in W}(-1)^{\ell(u)}P_k(u\cdot \la - \omega_1) = 0.
$$
\end{proposition}

\begin{proof} From the discussion in Section~\ref{part1},
$$
\sum_{u \in W}(-1)^{\ell(u)}P_k(u\cdot \la - \omega_1) = 
	P_k(\la - \omega_1) - P_k(s_1\cdot\la - \omega_1) 
- P_k(s_2\cdot\la - \omega_1) + P_k(s_1s_2\cdot\la - \omega_1).$$
For $a \geq 6$, the claim follows from Lemma~\ref{L:G2comb1}
and Lemma~\ref{L:G2comb2} above.
For $a < 6$, one can see from the proof of Lemma~\ref{L:G2comb2}
that the 2nd and fourth
terms are zero. Hence, for $3 \leq a \leq 5$, the result follows
from Lemma~\ref{L:G2comb1}.
For $1 \leq a \leq 2$, one can see from the proof of
Lemma~\ref{L:G2comb1} that both the first and
third terms vanish, and so the result follows. When $a$ is
small, the claim could also be
readily verified by hand.
\end{proof}


\subsection{\bf Partitions III}\label{part3}
Let $\la = a\omega_1 + b\omega_2$ for $a \geq 1$ and $0 \leq b
\leq 2$.
The goal of this section is to determine
$$
\sum_{u \in W}(-1)^{\ell(u)}P_{a + b - 1}(u\cdot \la -
\omega_1).
$$
See Proposition~\ref{P:G2prop2}.

From the discussion in Section~\ref{part1} we need to consider
the following
weights (with notation following Section~\ref{part2}): 
\begin{align*}
\ga_1 &:= \la - \omega_1 = (2a + 3b - 2)\al_1 + (a + 2b -
1)\al_2,\\
\ga_2 &:= s_2\cdot\la - \omega_1 = (2a + 3b - 2)\al_1 + (a + b -
2)\al_2,\\
\ga_3 &:= s_1\cdot\la - \omega_1 = (a + 3b - 3)\al_1 + (a + 2b -
1)\al_2,\\
\ga_4 &:= s_1s_2\cdot\la - \omega_1 = (a - 6)\al_1 + (a + b -
2)\al_2.
\end{align*}
More precisely, 
\begin{equation}\label{partk}
\sum_{u \in W}(-1)^{\ell(u)}P_{a + b - 1}(u\cdot \la - \omega_1)
=
P_{a + b - 1}(\ga_1) - P_{a + b - 1}(\ga_2) - P_{a + b -
1}(\ga_3) + P_{a + b - 1}(\ga_4).
\end{equation}
We first make some reduction observations as done in the proofs
in Section~\ref{part2}. Note
that when $a$ is small, some of the statements are trivially
true since both sides are zero.
But we include them here (and in the following statements) for
simplicity of exposition.
Observe also that the right hand side is independent of the
value of $b$.

\begin{lemma}\label{L:G2comb3} Let $\la$, $\ga_1$, $\ga_2$,
$\ga_3$, and $\ga_4$ be as above, and
let $k = a + b - 1$.
Then
\begin{itemize}
\item[(a)] $P_k(\ga_1) = P_{a - 1}(2(a - 1)\al_1 + (a -
1)\al_2)$;
\item[(b)] $P_k(\ga_2) = P_{a - 2}((2a - 5)\al_1 + (a -
3)\al_2)$;
\item[(c)] $P_k(\ga_3) = P_{a - 1}((a - 3)\al_1 + (a -
1)\al_2)$;
\item[(d)] $P_k(\ga_4) = P_{a - 2}((a - 6)\al_1 + (a -
3)\al_2)$.
\end{itemize}
\end{lemma}

\begin{proof} (a) Suppose that $\ga_1$ is decomposed as a sum of
$k$ positive roots.
Similar to the argument in Lemma~\ref{L:G2comb1}, since 
$a + 2b - 1 = (a + b - 1) + b$, at least $b$ of those roots must
(contain $2\al_2$ and
hence) be $\ta = 3\al_1 + 2\al_2$. Hence, $P_k(\ga_1) = P_{k -
b}(\ga_1 - b\ta)$,
and the claim follows.

(b) Again, as in the proof of Lemma~\ref{L:G2comb1}, since the
difference in the
number of $\al_1$s and $\al_2$s appearing in $\ga_2$ is $a +
2b$, if $\ga_2$ is
expressed as $k$ roots, then at least $b + 1$ of them must be
$3\al_1 + \al_2$.
Hence, $P_k(\ga_2) = P_{k - b - 1}(\ga_2 - (b + 1)(3\al_1 +
\al_2))$, and the
claim follows.

(c) As in part (a), we must have $P_k(\ga_3) = P_{k - b}(\ga_3 -
b\ta)$, and
the claim follows.

(d) As in part (b), similar to the proof of
Lemma~\ref{L:G2comb2}, we consider the
difference in the number of $\al_1$s and $\al_2$s appearing in
$\ga_4$. Since
this number is $-b - 4$, we can in particular assume that if
$\ga_4$ is decomposed
into $k$ roots, then at least $b + 1$ of them are $\al_2$.
Hence,
$P_k(\ga_4) = P_{k - b - 1}(\ga_4 - (b + 1)\al_2)$, and the
claim follows.
\end{proof}

With the aid of Lemma \ref{L:G2comb3}, we now observe that there
are some relationships
among the $P_k(\ga_i)$. To this end, we introduce a bit of
notation.
Consider an arbitrary integer $k \geq 0$ and weight $\ga =
c\al_1 + d\al_2$ for
$c, d \geq 0$.  Any decomposition of $\ga$ into a sum of $k$
positive roots is of one of two types: either the sum contains
at least one copy of
$\ta$ or it does not contain any copies of $\ta$.
Correspondingly, let $P_{k,\ta}(\ga)$ and
$P_{k,\not\ta}(\ga)$ denote the number of such root sums.  Then
$P_k(\ga) = P_{k,\ta}(\ga) + P_{k,\not\ta}(\ga)$.  Observe that
\begin{equation}\label{rellong}
P_{k,\ta}(\ga) = P_{k-1}(\ga - \ta).
\end{equation}

\begin{lemma} \label{L:G2comb4} Let $\la$, $\ga_1$, $\ga_2$,
$\ga_3$, and $\ga_4$ be as above, and
let $k = a + b - 1$.
Then
\begin{itemize}
\item[(a)] $P_k(\ga_1) = P_k(\ga_2) + P_{a - 1,\not\ta}(2(a -
1)\al_1 + (a - 1)\al_2)$;
\item[(b)] $P_k(\ga_3) = P_k(\ga_4) + P_{a - 1,\not\ta}((a -
3)\al_1 + (a - 1)\al_2)$.
\end{itemize}
\end{lemma}

\begin{proof} (a) We have
\begin{align*}
P_k(\ga_1) &= P_{a - 1}(2(a - 1)\al_1 + (a - 1)\al_2) \hskip.5in
(\text{by Lemma~\ref{L:G2comb3}(a)})\\
&= P_{a-1,\ta}(2(a - 1)\al_1 + (a - 1)\al_2) +
P_{a-1,\not\ta}(2(a - 1)\al_1 + (a - 1)\al_2)\\
&= P_{a-2}((2a - 5)\al_1 + (a - 3)\al_2) + P_{a-1,\not\ta}(2(a -
1)\al_1 + (a - 1)\al_2)
			\hskip.3in (\text{by (\ref{rellong})})\\
	&= P_k(\ga_2) + P_{a-1,\not\ta}(2(a - 1)\al_1 + (a - 1)\al_2) 
			\hskip.5in (\text{by Lemma~\ref{L:G2comb3}(b)}).
\end{align*}

(b) We have
\begin{align*}
P_k(\ga_3) &= P_{a - 1}((a - 3)\al_1 + (a - 1)\al_2) \hskip.5in
(\text{by Lemma~\ref{L:G2comb3}(c)})\\
&= P_{a-1,\ta}((a - 3)\al_1 + (a - 1)\al_2) + P_{a-1,\not\ta}((a
- 3)\al_1 + (a - 1)\al_2)\\
&= P_{a-2}((a - 6)\al_1 + (a - 3)\al_2) + P_{a-1,\not\ta}((a -
3)\al_1 + (a - 1)\al_2)
			\hskip.5in (\text{by (\ref{rellong})})\\
	&= P_k(\ga_4) + P_{a-1,\not\ta}((a - 3)\al_1 + (a - 1)\al_2) 
			\hskip.5in (\text{by Lemma~\ref{L:G2comb3}(d)}).
\end{align*}
\end{proof}

From (\ref{partk}), Lemma~\ref{L:G2comb3} and
Lemma~\ref{L:G2comb4}, we see that

\begin{equation}\label{partk2}
\sum_{u \in W}(-1)^{\ell(u)}P_{a + b - 1}(u\cdot \la - \omega_1)
=
P_{a-1,\not\ta}(2(a - 1)\al_1 + (a - 1)\al_2) -
P_{a-1,\not\ta}((a - 3)\al_1 + (a - 1)\al_2).
\end{equation}

\begin{lemma} \label{L:G2comb5} Let $c \geq 0$.  Then
$$
P_{c,\not\ta}(2c\al_1 + c\al_2) - P_{c,\not\ta}((c-2)\al_1 +
c\al_2) =
	\left\lceil\frac{c+1}{3}\right\rceil,
$$
where $\lceil x \rceil$ denotes the least integer greater than
or equal to $x$.
\end{lemma}

\begin{proof} Let $\eta_1 := 2c\al_1 + c\al_2$ and $\eta_2 := (c
- 2)\al_1 + c\al_2$.
Observe first that if $c < 2$, then $P_{c,\not\ta}(\eta_2) = 0$.
On the other hand,
we have $P_{0,\not\ta}(0) = 1$ and $P_{1,\not\ta}(2\al_1 +
\al_2) = 1$, and so the
claim holds for $c < 2$. Assume for the remainder of the proof
that $c \geq 2$.

Observe that if $\eta_i$ is expressed as a sum of $c$ positive
roots, none of
which are $\ta$, then each
root is necessarily of the form $a\al_1 + \al_2$ for $a \in \{0,
1, 2, 3\}$.
So the question of possible decompositions involves looking only
at the coefficients
of $\al_1$. For nonnegative integers $m, n$, let $P_m(n)$ denote
the number
of ways that $n$ can be expressed as a sum of $m$ integers
$$
n = n_1 + n_2 + \cdots + n_m
$$
where $n_i \in \{0,1,2,3\}$. With this notation,
$P_{c,\not\ta}(\eta_1) = P_c(2c)$,
$P_{c,\not\ta}(\eta_2) = P_c(c - 2)$, and our goal is to compute $P_c(2c) - P_c(c - 2)$ (when $c \geq 2$).  

For $m$, $n$ as above, let $S_m(n)$ denote the set of such
partitions of $n$
into $m$ integers.  
We first show that there is an injection $\varphi: S_c(c-2) \to
S_c(2c)$.
Let $\tau \in S_c(c - 2)$.  Say
$$
\tau: c - 2 = \tau_1 + \tau_2 + \cdots + \tau_c,
$$
where $\tau_i \in \{0,1,2,3\}$. Let $s$ denote the number of
$\tau_i$s which
equal 3. The remaining $c - s$ values must sum to $c - 2 - 3s$,
and hence
at most $c - 2 - 3s$ of those terms can be non-zero. In other
words, at least
$(c - s) - (c - 2 - 3s) = 2s + 2$ of the remaining terms are
zero.
That is, we may assume that $\tau$ has the form:
$$
c - 2 = \underbrace{3 + \cdots + 3}_{s \text{ times}} + 
\underbrace{0 + \cdots + 0}_{(2s + 2) \text{ times}} + \tau_{3s
+ 3} + \cdots + \tau_c,
$$
where, for $(3s + 3) \leq i \leq c$, $\tau_i \in \{0,1,2\}$. Let
$\varphi(\tau)$ be the
partition:
$$
2c = \underbrace{3 + \cdots + 3}_{2s \text{ times}} + 2 + 2 + 
\underbrace{0 + \cdots + 0}_{s \text{ times}} + (\tau_{3s + 3} +
1) +
		(\tau_{3s + 4} + 1) + \cdots + (\tau_c + 1).
$$
In words, the map $\varphi$ leaves the initial $s$ copies of $3$
fixed, sends $s$ of
the zeros to $3$, sends two of the zeros to $2$, leaves the
other $s$ zeros fixed,
and adds one to the unknown integers at the end. Note that those
unknown integers
are each at most $2$, so adding one is allowable. One can also
readily check that
the new sum does indeed add up to $2c$. It is clear that
$\varphi$ is an injection,
but we will explicitly construct an inverse below.

Observe that the resulting partition of $2c$ contains $2$ at
least twice. We claim
that the image of $\varphi$ is in fact precisely the subset $X
\subset S_c(2c)$ consisting
of those partitions where $2$ appears two or more times. Indeed,
we can define a
function $\psi: X \to S_c(c - 2)$ as follows. Let $\xi \in X$
and $s$ denote the
number of times that zero appears in $\xi$. The remaining $c -
s$ values in $\xi$
must sum to $2c$. We know that at least two of those have value
$2$. The remaining
$c - s - 2$ terms must sum to $2c - 4$. Since $2(c - s - 2) = 2c
- 4 - 2s$, at
least $2s$ of those terms must have value $3$. In other words,
$\xi$ has the
form:
$$
2c = \underbrace{0 + \cdots + 0}_{s \text{ times}} + 2 + 2 + 
\underbrace{3 + \cdots + 3}_{2s \text{ times}} + \xi_{3s + 3} +
\cdots + \xi_c,
$$
where (for $(3s + 3) \leq i \leq c$) $\xi_i \in \{1,2,3\}$. Let
$\psi(\xi)$
be the partition:
$$
c - 2 = \underbrace{0 + \cdots + 0}_{(2s + 2) \text{ times}} + 
\underbrace{3 + \cdots + 3}_{s \text{ times}} + (\xi_{3s + 3} -
1) +
			(\xi_{3s + 4} - 1) + \cdots + (\xi_c - 1).
$$
In words, the map $\psi$ leaves the zeros fixed, sends the two
$2$s to zero,
sends $s$ copies of $3$ to zero, leaves the other $s$ copies of
$3$ fixed, and subtracts
one from each of the remaining integers. Clearly $\psi$ is an
inverse to $\phi$.
Hence, $P_c(c - 2) = |X|$.

It remains to compute $P_c(2c) - |X|$. That is, we need to count
the number of
partitions
$$
2c = n_1 + n_2 + \cdots + n_c,
$$
where $n_i \in \{0,1,2,3\}$ but for which at most one value of
$n_i = 2$.
Write $c = 3m + t$ for $m \geq 0$ and $t < 3$. Then it is a
straightforward (but somewhat lengthy)
computation to show that the number of such partitions is $m +
1$. This is left
to the interested reader.  The lemma follows.
\end{proof}

Applying Lemma~\ref{L:G2comb5} with $c = a - 1$, we obtain the
following from (\ref{partk2}).

\begin{proposition}\label{P:G2prop2} Let $\la = a\omega_1 +
b\omega_2$ for $a \geq 1$ and $0 \leq b \leq 2$.
Then
$$
\sum_{u \in W}(-1)^{\ell(u)}P_{a + b - 1}(u\cdot \la - \omega_1)
=
	\left\lceil\frac{a}{3}\right\rceil.
$$
\end{proposition}


\subsection{\bf Vanishing Ranges}\label{SS:Gvan}
Suppose that
$\opH^i(G,H^0(\la)\otimes H^0(\la^*)^{(1)}) \neq 0$ for some $\la
\in X(T)_+$
and $i > 0$.  From the discussion in 
Section~\ref{Gres}, we know that if $i < 2p - 3$, then $\la$
must be of the
form $\la = p\omega_1 + w\cdot 0$, and more precisely, that it
must be one
of the weights listed in Table \ref{Gres}. For each such $\la$,
from
Proposition~\ref{P:G2prop1} and Proposition~\ref{P:G2prop2}, we
can identify the least value of
$k$ such that 
$$
\sum_{u \in W}(-1)^{\ell(u)}P_k(u\cdot \la - \omega_1) \neq 0,
$$
and moreover, identify the value of the sum. Further, from
Proposition~\ref{P:KostantPartCohom} (with
$k = (i - \ell(w))/2)$), we can then identify the least
non-negative $i$
with $\opH^i(G,H^0(\la)\otimes H^0(\la^*)^{(1)}) \neq 0$ along
with the dimension
of the cohomology group. This information is summarized in the
following table.
Here $k$ and $i$ are minimum possible values, and $\dim$ gives
the dimension of the
cohomology group (equivalently the value of (\ref{sum})).

\vskip.4cm
\begin{center}
\begin{tabular}{|c|c|c|c|c|c|}\hline
$w$ & $\ell(w)$ & $\la = p\omega_1 + w\cdot 0$ & $k$ & $i$ & $\dim$
\\
\hline 
&&&&&\\
$e$ & $0$ & $p\omega_1$ & $p - 1$ & $2p - 2$ &
$\left\lceil\frac{p}{3}\right\rceil$\\
&&&&&\\
\hline
&&&&&\\
$s_1$ & $1$ & $(p - 2)\omega_1 + \omega_2$ & $p - 2$ & $2p - 3$
& $\left\lceil\frac{p}{3}\right\rceil-1$\\
&&&&&\\
\hline
&&&&&\\
$s_1s_2$ & $2$ & $(p - 5)\omega_1 + 2\omega_2$ & $p - 4$ & $2p -
6$ & $\left\lceil\frac{p}{3}\right\rceil-2$\\
&&&&&\\
\hline
&&&&&\\
$s_1s_2s_1$ & $3$ & $(p - 6)\omega_1 + 2\omega_2$ & $p - 5$ &
$2p - 7$ & $\left\lceil\frac{p}{3}\right\rceil-2$\\
&&&&&\\
\hline
&&&&&\\
$s_1s_2s_1s_2$ & $4$ & $(p - 6)\omega_1 + \omega_2$ & $p - 6$ &
$2p - 8$ & $\left\lceil\frac{p}{3}\right\rceil-2$\\
&&&&&\\
\hline
&&&&&\\
$s_1s_2s_1s_2s_1$ & $5$ & $(p - 5)\omega_1$ & $p - 6$ & $2p - 7$
& $\left\lceil\frac{p}{3}\right\rceil-2$\\
&&&&&\\
\hline
\end{tabular}
\end{center}
\vskip.4cm

\begin{theorem} \label{T:G2summary} Suppose $\Phi$ is of type
$G_2$ and $p\geq 7$. 
\begin{itemize}
\item[(a)] $\opH^i(G,H^0(\la)\otimes H^0(\la^*)^{(1)}) = 0$ for $i < 2p - 8$.
\item[(b)] $\dim \opH^{2p - 8}(G,H^0(\la)\otimes H^0(\la^*)^{(1)}) = 
\begin{cases} \left\lceil\frac{p}{3}\right\rceil-2 &\mbox{ if }
\la = (p - 6)\omega_1 + \omega_2\\
0 &\mbox{ else. }
\end{cases}$
\item[(c)] $\dim \opH^{2p - 7}(G,H^0(\la)\otimes H^0(\la^*)^{(1)}) = 
\begin{cases} \left\lceil\frac{p}{3}\right\rceil-2 &\mbox{ if }
\la =(p - 5)\omega_1 \mbox{ or } (p-6)\omega_1 + 2\omega_2\\
0 &\mbox{ else. }
\end{cases}$
\item[(d)] $\opH^i(\gfp,k) = 0$ for $0 < i < 2p - 8$.
\end{itemize}
\end{theorem}

\begin{proof} Part (a) follows from Proposition \ref{P:KostantPartCohom} and 
Proposition \ref{P:G2prop1}.  Parts (b) and (c) follow from the preceding table
and discussion.  Part (d) follows from part (a) and 
Proposition~\ref{P:vanishing}.
\end{proof}

One would like to apply Theorem~\ref{T:nonvanishing} to conclude that
$\opH^{2p - 8}(\gfp,k) \neq 0$.  However, the weight $(p-5)\omega_1$ is less than
and linked to the weight $(p-6)\omega_1 + \omega_2$ and so the Theorem is not applicable.
The non-zero cohomology from the weight $(p-5)\omega_1$ in degree 
$2p - 7$ could ``cancel''
some or all of the cohomology coming from the weight $(p-6)\omega_1 + \omega_2$. 
We refer the interested
reader to \cite[Section 2.7]{BNP} for discussion of this
interplay.  

In a similar manner, cohomology in degree $2p-6$ coming from the weight
$(p-6)\omega_1 + \omega_2$ could cancel that in degree $2p - 7$ coming 
from the weight $(p-6)\omega_1 + 2\omega_2$.  So it is not even possible
to conclude that $\opH^{2p-7}(\gfp,k) \neq 0$.  In summary, 
alternate methods are needed to determine  the precise 
vanishing bound.


\section{Type $F_4$}\label{typeF}

Assume throughout this section that $\Phi$ is of type $F_4$ and
that $p > h = 12$ (so $p \geq 13$). Following the strategy laid out in Section 2, 
our goal is to find the least
$i > 0$ such that $\opH^i(G,H^0(\la)\otimes H^0(\la^*)^{(1)}) \neq 0$
for some $\la\in X(T)_{+}$.


\subsection{\bf Restrictions.}\label{SS:Fres} Suppose that 
$\opH^i(G,H^0(\la)\otimes H^0(\la^*)^{(1)}) \neq 0$ for some
$i > 0$ and $\la = p\mu + w\cdot 0$ with $\mu \in X(T)_{+}$ and
$w \in W$. From Proposition \ref{P:degreebound1}, $i \geq
(p-1)\langle\mu,\ta^{\vee}\rangle - 1$.
For $1 \leq i \leq 3$, we have
$\langle\omega_i,\ta^{\vee}\rangle \geq 2$,
while $\langle\omega_4,\ta^{\vee}\rangle = 1$.  
Therefore, unless $\mu = \omega_4 = \al_0$, we have
$\langle\mu,\ta^{\vee}\rangle \geq 2$ and $i \geq 2p - 3$.

Suppose now that $\la = p\omega_4 + w\cdot 0$ for some $w \in
W$.
With the aid of MAGMA, one can identify all $w$ for which $\la$
is
in fact dominant. From Proposition \ref{P:degreebound1}(a), with
$\si = \al_0$ and
$\la = p\omega_4 + w\cdot 0$, since
$\langle\omega_4,\al_0^{\vee}\rangle = 2$,
we have
\begin{equation}\label{icond}
i \geq 2(p - 1) + \ell(w) + \langle w\cdot
0,\al_0^{\vee}\rangle.
\end{equation}
By checking all possible cases, one finds that 
$\ell(w) + \langle w\cdot 0,\al_0^{\vee}\rangle \geq - 7$.
Combining this with
(\ref{icond}), we conclude that $i \geq 2p - 9$. From
Proposition \ref{P:vanishing},
we get the following.

\begin{theorem}\label{T:F4van} Suppose $\Phi$ is of type $F_4$ and $p\geq 13$. Let
$\la \in X(T)_{+}$.  Then
\begin{itemize}
\item[(a)] $\opH^i(G,H^0(\la)\otimes H^0(\la^*)^{(1)}) = 0$ for
$0 < i < 2p - 9$;
\item[(b)] $\opH^i(\gfp,k) = 0$ for $0 < i < 2p - 9$.
\end{itemize}
\end{theorem}

\subsection{} Based on the preceding discussion, the weights which could give
$\opH^i(G,H^0(\la)\otimes H^0(\la^*)^{(1)}) \neq 0$ for 
$i \leq 2p - 7$ are summarized in the following table.

\vskip.4cm
\begin{center}
\begin{tabular}{|c|c|c|c|}\hline
$\la = p\omega_4 + w\cdot 0$ & $\ell(w)$ & $\langle w\cdot 0,\al_0^{\vee}\rangle$ &
$i$\\
\hline 
$(p - 12)\omega_4 + \omega_2$ & 13 & $-20$ & $2p - 9$\\
\hline
$(p - 12)\omega_4 + \omega_3$ & 14 & $-21$ & $2p - 9$\\
\hline
$(p - 11)\omega_4$ & 15 & $-22$ & $2p - 9$\\
\hline
$(p - 11)\omega_4 + 3\omega_1$ & 10 & $-16$ & $2p - 8$\\
\hline
$(p - 12)\omega_4 + 2\omega_1 + \omega_3$ & 11 & $-17$ & $2p -
8$\\
\hline
$(p - 12)\omega_4 + \omega_1 + \omega_2$ & 12 & $-18$ & $2p -
8$\\
\hline
$(p - 11)\omega_4 + 2 \omega_1 + \omega_2$ & 9 & $-14$ & $2p -
7$\\
\hline
$(p - 12)\omega_4 + \omega_1 + \omega_2 + \omega_3$ & 10 & $-15$
& $2p - 7$\\
\hline
$(p - 12)\omega_4 + 2\omega_2$ & 11 & $-16$ & $2p - 7$\\
\hline
\end{tabular}
\end{center}
\vskip.4cm

As seen in Section \ref{SS:G1analysis}, $\la - \omega_4$ must be
a weight of
$S^{\frac{i - \ell(w)}{2}}(\ul^*)$, and hence $i$ is congruent
to $\ell(w)$ mod 2.
It follows that some of the above degree bounds are even higher.
For example,
consider $\la = (p - 12)\omega_4 + \omega_3 = p\omega_4 + w\cdot
0$. Since
$\ell(w) = 14$ but $2p - 9$ is odd, the least value $i$ could
take would
be $2p - 8$. A similar situation holds for $\la = (p -
12)\omega_4 + 2\omega_1 + \omega_3$
and $\la = (p - 12)\omega_4 + \omega_1 + \omega_2 + \omega_3$.
Similarly, for the other
weights in the above list, if the cohomology vanishes in the
degree $i$ listed, then the
next possible non-vanishing degree is $i + 2$. We summarize this
in the following lemma.

\begin{lemma}\label{L:F4nonvan} Suppose $\Phi$ is of type $F_4$, $p\geq 13$ and $\la \in X(T)_{+}$. Suppose that $\opH^i(G,H^0(\la)\otimes
H^0(\la^*)^{(1)}) \neq 0$.
Then
\begin{itemize}
\item[(a)] $i \geq 2p - 9$; 
\item[(b)] if $i = 2p - 9$, then $\la = (p - 12)\omega_4 +
\omega_2$ or $(p - 11)\omega_4$;
\item[(c)] if $i = 2p - 8$, then $\la = (p - 12)\omega_4 +
\omega_3$, $(p - 11)\omega_4 + 3\omega_1$, or
	$(p - 12)\omega_4 + \omega_1 + \omega_2$;
\item[(d)] if $i = 2p - 7$, then $\la = (p - 12)\omega_4 +
\omega_2$, $(p - 11)\omega_4$,
$(p - 12)\omega_4 + 2\omega_1 + \omega_3$, $(p - 11)\omega_4 +
2\omega_1 + \omega_2$, or
	$(p - 12)\omega_4 + 2\omega_2$.
\end{itemize}
\end{lemma}


\subsection{Conjectures} 
In principle, one could use Proposition \ref{P:KostantPartCohom}
to compute the dimension of
$$\opH^i(G,H^0(\la)\otimes H^0(\la^*)^{(1)})$$ 
in terms of partition functions
for the weights in Lemma \ref{L:F4nonvan}. For small $p$, one can use MAGMA to make
this
computation. For $p = 13, 17,$ or $19$, one finds that the two
candidates in
degree $2p - 9$ have zero cohomology. They do give cohomology in
degree $2p - 7$.
And in degree $2p - 8$, the only one weight (of the three) which
has cohomology is
$(p - 12)\omega_4 + \omega_3$.  We make the following

\begin{conjecture}\label{C:F4conj1} Suppose that $\Phi$ is of type $F_4$, $p\geq 13$, and $\la = p\mu + w\cdot 0 \in X(T)_{+}$. Then
\begin{itemize}
\item[(a)] $\opH^i(G,H^0(\la)\otimes H^0(\la^*)^{(1)}) = 0$ for
$0 < i < 2p - 8$;
\item[(b)] $\opH^{2p - 8}(G,H^0(\la)\otimes H^0(\la^*)^{(1)})
\neq 0$ for
$\la = (p - 12)\omega_4 + \omega_3$.
\end{itemize}
\end{conjecture}

If part (a) of the conjecture holds, then $\opH^i(\gfp,k) = 0$ for 
$0 < i < 2p - 8$ thus improving upon Theorem \ref{T:F4van}.  However,
even if part (b) of the conjecture also holds, it does not necessarily follow
that $\opH^{2p-8}(\gfp,k) \neq 0$.  Analogous to the situation for type $G_2$
(cf. Section \ref{SS:Gvan}), cohomology in degree $2p-7$ from the weight
$(p-11)\omega_4$ could cancel out the cohomology in degree $2p-8$ from
the weight $(p-12)\omega_4 + \omega_3$.

Conjecture \ref{C:F4conj1} is a special case of a more general conjecture on 
partition functions (known to hold for small values of $m$). 
Conjecture \ref{C:F4conj1}(a) would follow from parts (a) and (b)
while Conjecture \ref{C:F4conj1}(b) would follow from part (c).

\begin{conjecture}Suppose that $\Phi$ is of type $F_4$ and $m
\geq 1$. Then
\begin{itemize}
\item[(a)] $\displaystyle \sum_{u \in W}(-1)^{\ell(u)}P_{m +
1}(u\cdot(m\omega_4 + \omega_2) - \omega_4) = 0$;
\item[(b)] $\displaystyle \sum_{u \in
W}(-1)^{\ell(u)}P_{m-1}(u\cdot(m\omega_4) - \omega_4) =
\begin{cases}
0 &\text{ if } m \text{ is even},\\
1 &\text{ if } m \text{ is odd};
\end{cases}$
\item[(c)] $\displaystyle \sum_{u \in W}(-1)^{\ell(u)}P_{m + 1}(u\cdot(m\omega_4 +
\omega_3) - \omega_4) = 1.$
\end{itemize}
\end{conjecture}


\bibliographystyle{amsalpha}

\end{document}